\documentclass[12pt,leqno]{memo-l}
\usepackage{amsmath,amstext,amsthm,amssymb,amsxtra}
\usepackage[colorlinks,citecolor=red,pagebackref,hypertexnames=false]{hyperref}
\usepackage[backrefs]{amsrefs}

\usepackage{latexsym} 
\usepackage{pgf,tikz}
\usetikzlibrary{snakes}
\usetikzlibrary{arrows} 
\newtheorem{theorem}[equation]{Theorem}    
\newtheorem{proposition}[equation]{Proposition}
\newtheorem{conjecture}[equation]{Conjecture}
\newtheorem{stein}[equation]{E.\thinspace M.~Stein Conjecture}
\newtheorem{zygmund}[equation]{Zygmund Conjecture} 
\newtheorem{lipKakeya}[equation]{Weak $ L^2$ estimate for the Lipschitz Kakeya Maximal Function} 
\newtheorem{carleson}[equation]{Carleson's Theorem}
\newtheorem{fl}[equation]{Fourier Localization Lemma}

\newtheorem{lemma}[equation]{Lemma}

\theoremstyle{definition}
\newtheorem{definition}[equation]{Definition}
\newtheorem{remark}[equation]{Remark}
\newtheorem*{Acknowledgment}{Acknowledgment}

\numberwithin{section}{chapter}
\numberwithin{equation}{chapter}
\numberwithin{figure}{chapter}

\newcommand{\abs}[1]{\lvert#1\rvert}

\def\abs#1{\lvert#1\rvert}
\def\Abs#1{\bigl\lvert#1\bigr\rvert}
\def\ABs#1{\Bigl\lvert#1\Bigr\rvert}

\def\name #1 #2{\operatorname {\mathsf #1}(#2)}

\newcommand{\scl}{{\operatorname  {scl}}}

\def\norm#1.#2.{\lVert#1\rVert_{#2}}
\def\Norm#1.#2.{\bigl\lVert#1\bigr\rVert_{#2}}
\def\NOrm#1.#2.{\Bigl\lVert#1\Bigr\rVert_{#2}}
\def\NORm#1.#2.{\biggl\lVert#1\biggr\rVert_{#2}}
\def\NORM#1.#2.{\Biggl\lVert#1\Biggr\rVert_{#2}}

\def\ip#1,#2,{\langle #1,#2\rangle}
\def\Ip#1,#2,{\bigl\langle#1,#2\bigr\rangle}
\def\IP#1,#2,{\Bigl\langle#1,#2\Bigr\rangle}

\def\abs#1{\lvert#1\rvert}
\def\Abs#1{\bigl\lvert#1\bigr\rvert}
\def\ABs#1{\Bigl\lvert#1\Bigr\rvert}

\def\eqdef{\stackrel{\mathrm{def}}{=}}

\def\ind#1{ {\mathbf 1}_{#1}}
\def\size#1{\operatorname{ size}(#1)}

\def\sh#1{\operatorname{ sh}(#1)}

\def\dense#1{\operatorname{dense}(#1)}

\def\prm#1{{\operatorname{\rm ann}}(#1)}
\def\scl#1{\operatorname{\rm scl}(#1)}      

\def\Prm{{\operatorname{\mathsf {ann}}}}
\def\Scl{{\operatorname{\mathsf {scl}}}}

\def\fss#1{{\phi_{s#1}} }    

 \def\inr{\int_{\mathbb R}}        
 \def\inrr{\int_{\mathbb R^2}}  

\def\tree#1{$#1$--tree} 
\def\ipf{ \langle f,\varphi_{s} \rangle}

\def\vLip{\norm v.\text{\rm Lip}.}

\def\Xint#1{\mathchoice
   {\XXint\displaystyle\textstyle{#1}}%
   {\XXint\textstyle\scriptstyle{#1}}%
   {\XXint\scriptstyle\scriptscriptstyle{#1}}%
   {\XXint\scriptscriptstyle\scriptscriptstyle{#1}}%
   \!\int}
\def\XXint#1#2#3{{\setbox0=\hbox{$#1{#2#3}{\int}$}
     \vcenter{\hbox{$#2#3$}}\kern-.5\wd0}}

\def\dashint{\Xint-}

\def\thinrect#1#2#3#4{
	\pgfrect[stroke]{\pgfpoint{#3cm}{#4cm}}{\pgfpoint{#2cm}{.3cm}}
	 }

\def\ThinRect#1{\begin{pgfrotateby}{\pgfdegree{#1}}
	\pgfrect[stroke]{\pgfpoint{0cm}{0cm}}{\pgfpoint{0.6cm}{4cm}}
	\end{pgfrotateby}}

\begin{document}
\frontmatter
\title[Stein Conjecture]{On a Conjecture of E.~M.\thinspace Stein
on the Hilbert Transform on Vector Fields}

 \author[M. Lacey \& Xiaochun Li]{Michael Lacey and Xiaochun Li}

\address{Michael Lacey \\ School of Mathematics \\ 
Georgia Institute of Technology \\ Atlanta GA 30332 }

\address{        Xiaochun Li\\
         	Department of Mathematics\\
        University of Illinois\\
        Urbana IL 61801}

\email{xcli@math.uiuc.edu}
\thanks{The authors are supported in part by NSF grants.  M.L. was supported in 
part by the Guggenheim Foundation.}

\date{}

\begin{abstract} 
Let $ v$ be a smooth vector field on the plane, that is a map from the plane 
to the unit circle. We study sufficient conditions for the boundedness of the 
Hilbert transform 
\begin{equation*}
  \operatorname H_{v, \epsilon }f(x) := \text{p.v.}\int_{-\epsilon }^ \epsilon
  f(x-yv(x))\;\frac{dy}y\, 
\end{equation*}
where $ \epsilon $ is a suitably chosen parameter, determined by the smoothness 
properties of the vector field.  It is a conjecture, due to E.\thinspace M.\thinspace 
Stein, that if $ v$ is Lipschitz, there is a positive $ \epsilon $ for which the 
transform above is bounded on $ L ^{2}$.  Our principal result gives a sufficient 
condition in terms of the boundedness of a maximal function associated to $ v$, 
namely  that this new maximal function be bounded on some 
$ L ^{p}$, for some $ 1<p<2$. We show that the maximal function is bounded from 
$ L ^{2}$ to weak $ L ^{2}$ for all Lipschitz vector fields.  The relationship 
between our results and other known sufficient conditions is explored. 

\end{abstract}


\maketitle

\setcounter{page}{4}
\tableofcontents

%

\chapter*{Preface}

This memoir is devoted to a question in planar Harmonic Analysis,  a  
subject  which is a circle of problems all related to the Besicovitch set.  This 
anomalous set has zero Lebesgue measure, yet contains 
a line segment of unit length in each direction of the plane.  It is a 
known, since the 1970's, that such sets must necessarily have full 
Hausdorff dimension.  The existence of these sets, and the full Hausdorff 
dimension, are intimately related to other, independently interesting issues \cite{stein}. 
An important tool to study these questions is the so-called Kakeya Maximal Function, 
in which one computes the maximal average of a function over rectangles of 
a fixed eccentricity and arbitrary orientation. 

Most famously, Charles Fefferman showed \cite{MR45:5661} that the Besicovitch set is the obstacle to the 
boundedness of the disc multiplier in the plane.  But as well, this set is 
intimately related to finer questions of Bochner-Riesz summability of 
Fourier series in higher dimensions and space-time regularity of solutions of
 the wave equation.  
 
This memoir concerns one of the finer questions which center around the Besicovitch set in
the plane. (There are not so many of these questions, but our purpose here is not 
to catalog them!) It concerns a certain degenerate Radon transform.  Given a 
vector field $ v$ on $ \mathbb R ^2 $, one considers a Hilbert transform 
computed in the one dimensional line segment determined by $ v $, namely the 
Hilbert transform of a function on the plane computed on the line segment 
$ \{ x+t v (x)\mid \lvert  t\rvert\le 1 \}$.  

The Besicovitch set itself says that choice of $ v$ cannot be just measurable, for 
you can choose the vector field to always point into the set.  Finer constructions show 
that one cannot take it to be H\"older continuous of any index strictly less than one. 
Is the sharp condition of H\"older continuity of index one enough?  This is the question 
of E.~M.\thinspace Stein, motivated by an earlier question of A.\thinspace Zygmund, 
who asked the same for the question of differentiation of integrals.  

The answer is not known under any condition of just smoothness of the vector field.  
Indeed, as is known, and we explain, a positive answer would necessarily imply 
Carleson's famous theorem on the convergence of Fourier series, \cite{car}.  This memoir 
is concerned with reversing this implication: Given the striking recent successes 
related to Carleson's Theorem, what can one say about Stein's Conjecture?  
In this direction, we introduce a new object into the study, a 
\emph{Lipschitz Kakeya Maximal Function}, which is a variant of the more familiar 
Kakeya Maximal Function, which links the vector field $ v$ to the `Besicovitch sets' 
associated to the vector field.  One averages a function over rectangles of 
arbitrary orientation and---in contrast  to the classical setting---arbitrary 
eccentricity.  But, the rectangle must suitably localize the directions in which 
the vector field points.   This Maximal Function admits a favorable estimate on 
$ L ^{2}$, and this is one of the main results of the Memoir.   

On Stein's Conjecture, we prove a conditional result:  If the Lipschitz Kakeya Maximal 
Function associated with $ v$ maps is an estimate a  little better than our $ L ^2 $
 estimate, then the associated Hilbert transform is indeed bounded.  Thus, the 
 main question left open concerns the behavior of these novel Maximal Functions.  
 
While the main result is conditional, it does contain many of the prior results on the 
subject, and greatly narrows the possible avenues of a resolution of this conjecture.  

The principal results and conjectures are stated in the Chapter 1; following that 
we collect some of the background material for this subject, and prove some of the 
folklore results known about the subject.  The remainder of the Memoir is taken 
up with the proofs of the Theorems stated in the Chapter 1.  

\begin{Acknowledgment}
The efforts of a  strikingly generous referee has resulted in corrections of arguments, and 
improvements in presentation throughout this manuscript. We are indebted to that person.      
\end{Acknowledgment}
 
\aufm{Michael T.~Lacey and Xiaochun Li}


\mainmatter
%


\chapter{Overview of Principal Results}

We are interested in singular integral operators on functions of two variables,
 which act  by performing a one dimensional transform along a  particular 
 line in the plane. 
 The choice of lines is to be variable.
 Thus, for a measurable map,  $v$ from $\mathbb R^2$ to the unit
 circle in the plane, that is a vector field, and a Schwartz function $f$ on $\mathbb R^2$, define 
 \begin{equation*} 
  \operatorname H_{v, \epsilon }f(x) := \text{p.v.}\int_{-\epsilon }^ \epsilon 
  f(x-yv(x))\;\frac{dy}y\,.
 \end{equation*}
 This is   a  truncated Hilbert transform performed on the line segment
 $\{x+tv(x)\;:\; \abs t<1\}$.  We stress the limit of the truncation in the definition above 
 as it is important to  different scale invariant formulations of our questions of 
 interest. This is an example of a Radon transform, one that is degenerate in the sense 
 that we seek results independent of geometric assumptions on the vector field.  We are 
 primarily interested in assumptions of smoothness on the vector field.  
 
 Also relevant is the corresponding maximal function 
\begin{equation} \label{e.MVE}
\operatorname M _{v, \epsilon }f := 
\sup _{0<t\le \epsilon }  (2t) ^{-1} \int _{-t } ^{t} \lvert f (x-s v (x)) \rvert \; ds 
\end{equation} 
The principal conjectures here concern Lipschitz vector fields. 

\begin{zygmund} Suppose that $ v$ is Lipschitz.  Then, for all $ f\in L ^2 (\mathbb R ^2 )$ 
we have the pointwise convergence 
\begin{equation}\label{e.zygmund}
\lim _{t\to 0} (2t) ^{-1} \int _{-t } ^{t} f (x-s v (x))\; ds 
= f (x) \qquad \textup{a.\thinspace e.}
\end{equation}
More particularly, there is an absolute constant $ K >0$ so that if 
$ \epsilon ^{-1} = K \norm v . \textup{\rm Lip}. $, we have the weak type estimate 
\begin{equation}\label{e.weak-zygmund}
\sup _{\lambda >0} \lambda  \lvert  \{\operatorname M _{v,\epsilon  } f >\lambda \}\rvert ^{1/2} 
\lesssim \norm f.2. \,.
\end{equation}
\end{zygmund}

The origins of this question go back to the discovery of the Besicovitch
set in the 1920Õs, and in particular, constructions of this set 
show that the Conjecture is false under the
 assumption that $v$ is H\"older continuous for any index strictly less than $1$.  
 These constructions, known since the 1920's, were 
 the inspiration for A.~Zygmund to ask if integrals of, say, $L^2(\mathbb R^2)$ functions could be differentiated in a Lipschitz choice of 
 directions.  That is, for Lipschitz $v$, and $f\in L^2$, is it the case that 
 \begin{equation*}
 \lim_{\epsilon\to0}(2\epsilon)^{-1}\int_{-\epsilon}^\epsilon f(x-yv(x))\; dy=f(x)\qquad \text{a.e.}(x)
 \end{equation*} 
These and other matters  are reviewed in the next chapter.

Much later, E. M. Stein \cite{stein-icm} raised the singular integral variant of this 
conjecture. 

\begin{stein}
 There is an absolute constant $ K >0$ so that if 
$ \epsilon ^{-1} = K \norm v . \textup{\rm Lip}. $, we have the weak type estimate 
\begin{equation}\label{e.weak-stein}
\sup _{\lambda >0} \lambda  \lvert  \{\lvert \operatorname H _{v,\epsilon  } f \rvert
>\lambda \}\rvert ^{1/2} 
\lesssim \norm f.2. \,.
\end{equation}
\end{stein}

These are very difficult conjectures. Indeed, it is known that if
the Stein Conjecture holds  for, say, $C^2$  vector fields, then Carleson's 
Theorem on the pointwise convergence of  Fourier series
\cite{car}
would follow. This folklore result is recalled in the next Chapter. 

We will study these questions using modifications of the phase plane 
analysis associated with Carleson's Theorem \cites{laceythiele,
MR1491450,MR1619285,MR1689336,MR1425870,laceyli1} and a new tool, which we 
term a \emph{Lipschitz Kakeya Maximal Function.} 

Associated with the Besicovitch set is the Kakeya Maximal Function, a maximal function over all rectangles of a given eccentricity. A 
key estimate is that the $ L^2 \longrightarrow L^{2, \infty }$  norm of this operator grows 
logarithmically in the eccentricity, \cites{MR0481883,MR0487260}.

 Associated with a Lipschitz vector field, we define a class of maxi- 
mal functions taken over rectangles of \emph{arbitrary} eccentricity, but these 
rectangles are approximate level sets of the vector field.
Perhaps surprisingly, these maximal functions admit an $ L ^{2}$
 bound that is \emph{independent} of eccentricity. 
Let us explain. 
 \begin{figure}
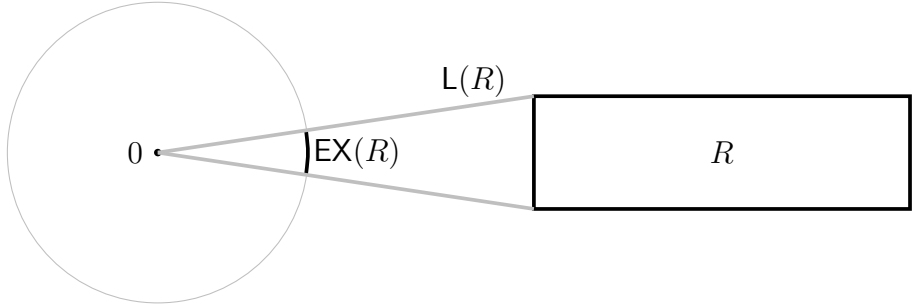

	 \begin{pgfpicture}{0cm}{0cm}{8cm}{4.5cm}
	 \begin{pgftranslate}{\pgfxy(-.15,2)}
	{\color{lightgray}  
		\pgfcircle[stroke]{\pgfxy(0,0)}{2cm}
	}
	\pgfcircle[fill]{\pgfxy(0,0)}{0.05cm}
	\pgfputat{\pgfxy(-.3,0)}{\pgfbox[center,center]{$0$}}
	\pgfputat{\pgfxy(2.65,0)}{\pgfbox[center,center]{$\mathsf{ EX}(R) $}}
	\pgfputat{\pgfxy(7.5,0)}{\pgfbox[center,center]{$R $}}
	\pgfsetlinewidth{1.4pt}
	\pgfrect[stroke]{\pgfxy(5,-.75)}{\pgfxy(5,1.5)}
	\pgfmoveto{\pgfxy(1.97,.3)}
	\pgfcurveto{\pgfxy(2.005,0)}{\pgfxy(2.005,0)}{\pgfxy(1.97,-.3)}\pgfstroke
			\pgfputat{\pgfxy(4.2,.95)}{\pgfbox[center,center]{$ \mathsf L (R)$}}
{\color{lightgray} 
			\pgfmoveto{\pgfxy(0,0)} 
			\pgflineto{\pgfxy(5,.75)} 	\pgfstroke
			\pgfmoveto{\pgfxy(0,0)} 
			\pgflineto{\pgfxy(5,-.75)}	\pgfstroke
	}
	\end{pgftranslate}
	\end{pgfpicture} 

	\caption{ An example eccentricity interval $ \mathsf {EX}(R)$.  The circle on 
	the left has radius one.}
	\label{f.EXR}
	\end{figure}

 A \emph{rectangle} is determined as follows.  Fix a choice 
of unit vectors in the plane $ (\operatorname e,\operatorname e ^{\perp})$, with $ \operatorname 
e ^{\perp}$ being the vector $ \operatorname e$ rotated by $ \pi /2$.  Using these vectors 
as coordinate axes, a rectangle is a product of two intervals $ R= I\times J$.  We will 
insist that $ \abs{ I}\ge \abs{ J}$, and use the notations 
\begin{equation}\label{e.LW}
\name L R=\abs{ I},\qquad \name W R=\abs{ J}
\end{equation}
for the length and width respectively of $ R$. 
 
The \emph{interval of uncertainty of  $ R$} is the subarc  $ \name EX R $ of the unit circle in the plane, 
centered at $ \operatorname e$, and of length $ \name W R/\name L R$.  See Figure~\ref{f.EXR}.

We now fix a Lipschitz map $ v$ of the plane into the unit circle. 
We only consider rectangles $ R$ with 
\begin{equation} \label{e.shortlength}
\name L R\le{} (100 \norm v. \text{Lip} .) ^{-1}\,.
\end{equation}  
For such a rectangle $ R$, set $ \name V R=R\cap v ^{-1}(\name EX R )$. 
It is essential to impose a restriction of this type on the length of the rectangles, for 
with out it, one can modify constructions of the Besicovitch set to provide examples 
which would contradict the main results and conjectures of this work. 

For $ 0<\delta <1$, we consider the maximal functions 
\begin{equation}\label{e.Mdef}
\operatorname M _{v,\delta }f(x) \eqdef \sup _{\substack{ \abs{ \name V R }\ge \delta \abs{ R} }}
\frac {\mathbf 1 _{R}(x)} {\abs{ R} } \int _{R} \abs{f(y)}\; dy .
\end{equation}
That is we only form the supremum over rectangles for which the vector field
 lies in the interval of uncertainty for a fixed positive proportion $ \delta $ of 
 the rectangle, see Figure~\ref{f.density}.

\begin{lipKakeya}\label{t.lipKakeya}
The maximal function $\operatorname M_{\delta,v}$ is bounded from $L^2({\mathbb R}^2)$
to $L^{2, \infty}({\mathbb R}^2)$ with norm at most $ \lesssim \delta ^{-1/2}$. That is, 
for any $\lambda >0$,  and $ f\in L^2 (\mathbb R^2)$, this inequality holds: 
\begin{equation}\label{max1}
 \lambda ^{2}
 \abs{\{x\in {\mathbb R}^2: \operatorname M_{\delta,v}f(x)>\lambda \}} 
 \lesssim \delta^{-1}\norm f.2.^2\,.
\end{equation}
The norm estimate in particular is independent of the Lipschitz vector field $ v$. 
\end{lipKakeya}

\begin{figure}
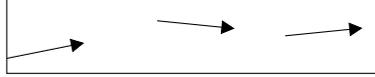

 \begin{pgfpicture}{0cm}{0cm}{5cm}{1.1cm}
%
\pgfrect[stroke]{\pgfxy(0,0)}{\pgfxy(5,1)}
{ \pgfsetendarrow{\pgfarrowtriangle{4pt}} 
  \pgfline{\pgfxy(0,.2)}{\pgfxy(1,.4)} 
   \pgfline{\pgfxy(2,.7)}{\pgfxy(3,.6)} 
    \pgfline{\pgfxy(3.7, .5)}{\pgfxy(4.7,.6)} 
  \pgfclearendarrow}
  \end{pgfpicture}
\caption{A rectangle, with the vector field pointing in the long direction 
of the rectangle at three points. }
  \label{f.density}  
\end{figure}

A principal Conjecture of this work is: 

\begin{conjecture} \label{j.laceyli1}
For some $ 1<p<2$, and some finite $N $ and all $ 0<\delta <1$ and all 
Lipschitz vector fields $ v$, 
the maximal function $\operatorname M_{\delta,v}$ is bounded from $L^p({\mathbb R}^2)$
to $L^{p, \infty}({\mathbb R}^2)$ with norm at most $ \lesssim \delta ^{-N}$.
\end{conjecture}

We cannot verify this Conjecture, only establishing that the norm of the 
operator can be controlled by a slowly growing function of eccentricity.  

In fact, this conjecture is stronger than what is needed below. Let us 
modify the definition of the Lipschitz Kakeya Maximal Function, by 
restricting the rectangles that enter into the definition to have 
an approximately fixed width.  
For $ 0<\delta <1$, and choice of $ 0<\mathsf w<\tfrac 1 {100} \norm v. \textup{Lip}.$,
parameterizing the width of the rectangles we consider, 
define 
\begin{equation}\label{e.MdefW}
\operatorname M _{v,\delta, \mathsf w }f(x) \eqdef 
\sup _{\substack{ \abs { \name V R }\ge \delta \abs{ R}  \\ 
\mathsf w\le \name W R\le 2 \mathsf w}}
\frac {\mathbf 1 _{R}(x)} {\abs{ R} } \int _{R} \abs{f(y)}\; dy .
\end{equation}
We can restrict attention to this case as the primary interest below 
is the Hilbert transform on vector fields applied to functions with 
frequency support in a fixed annulus.  By Fourier uncertainty, the 
width of the fixed annulus is the inverse of the parameter $ \mathsf w$ above. 

\begin{conjecture}\label{j.laceyli2}
For some $ 1<p<2$, and some finite $N $ and all $ 0<\delta <1$, 
all Lipschitz vector fields $ v$ and $ 0<\mathsf w< \tfrac 1 {100} \norm v. \textup{Lip}.$ 
the maximal function $\operatorname M_{\delta,v, \mathsf w}$
is bounded from $L^p({\mathbb R}^2)$
to $L^{p, \infty}({\mathbb R}^2)$ with norm at most $ \lesssim \delta ^{-N}$.
\end{conjecture}

These conjectures are stated as to be universal over Lipschitz vector fields. 
On the other hand, we will state conditional results below in which we 
assume that a given vector field satisfies the Conjecture above, and then 
derive consequences for the Hilbert transform on vector fields.  
We also show that e.\thinspace g.\thinspace real-analytic vector fields 
\cite{bourgain} satisfy these conjectures.

\bigskip  

We turn to the Hilbert transform on vector fields. As it turns out, 
it is useful to restrict functions in frequency variables to an annulus. 
Such operators are given by 
\begin{equation*}
\operatorname S_t f (x)=\int _{1/t \le \lvert  \xi \rvert \le 2/ t  } 
\widehat f (\xi ) \operatorname e ^{i \xi \cdot x}\; d \xi \,.  
\end{equation*}
The relevance in part is explained in part by this result of the authors 
\cite{laceyli1}, valid for measurable vector fields. 

\begin{theorem}\label{t.laceyli1}
For any \emph{measurable} vector field $ v$ we have the $ L ^2 $ into $ L ^{2,\infty }$ 
\begin{equation*}
\sup _{\lambda >0} \lambda  \lvert  
\{\lvert  \operatorname H _{v,\infty } \circ \operatorname S_t f\rvert>\lambda  \}\rvert ^{1/2} 
\lesssim \norm f.2. \,. 
\end{equation*}
The inequality holds uniformly in $ t>0$. 
\end{theorem}

It is critical that the Fourier restriction $ \operatorname S_t$ enters in, for otherwise 
the Besicovitch set would provide a counterexample, as we indicate in the 
first section of Chapter 2. This is one point 
at which the difference between the maximal function and the Hilbert 
transform is striking. The maximal function variant of the estimate 
above holds, and is relatively easy to prove, yet the Theorem above 
contains Carleson's Theorem on the pointwise convergence of Fourier series as a Corollary.

The weak $ L ^{2}$ estimate is sharp for measurable vector fields, and so we raise 
the conjecture 

\begin{conjecture}\label{j.LipH}
 There is a universal constant K for which we 
have the inequalities 
\begin{equation}\label{e.LipH}
\sup _{0<t< \norm v. \textup{Lip}.} \norm \operatorname H _{v, \epsilon } 
\circ \operatorname S_t .2\to 2. < \infty \,, 
\end{equation}
where $ \epsilon =\norm v. \textup{Lip}./K$. 
\end{conjecture}

Modern proofs of the pointwise convergence of Fourier series use 
the so-called \emph{restricted weak type approach}, invented by Muscalu, Tao and Thiele in 
\cite{mtt}. This 
method uses refinements of the weak $ L ^2 $ estimates, together with appropriate
maximal function estimates, to  derive  $ L^p$ inequalities, 
for $1<p<2$. In the case of Theorem~\ref{t.laceyli1}---for which this approach 
can not possibly work---the appropriate maximal function is the maximal
function over all possible line segments.  This is the unbounded 
Kakeya Maximal function for rectangles with zero eccentricity. One
might suspect that in the Lipschitz case, there is a bounded maximal 
function. This is another motivation for our Lipschitz Kakeya Maximal 
Function, and our main Conjecture~\ref{j.laceyli1}. 
We illustrate how these issues play out in our current setting, with 
this conditional result, one of the main results of this memoir.

\begin{theorem}\label{t.conditional} Assume that Conjecture~\ref{j.laceyli2} holds
for a choice of Lipschitz vector field $ v$.  
Then we have the inequalities 
\begin{equation}\label{e.cond1}
\norm  \operatorname H _{v, \epsilon } \circ \operatorname S_t .2. 
\lesssim 1\,, \qquad 0<t< \norm v. \textup{Lip}. \,. 
\end{equation}
Here, $ \epsilon $ is as in \eqref{e.LipH}. 
Moreover, if the vector field as $ 1+\eta  $ derivatives, we have the estimate
\begin{equation}\label{e.cond2}
\norm  \operatorname H _{v, \epsilon }   .2. 
\lesssim  (1+ \log  \norm v. C ^{1+\eta  }.) ^2 \,. 
\end{equation}
In this case, $ \epsilon = K/\norm v.C ^{1+\eta  }. $ and $ \eta >0$. 
\end{theorem}

While this is a conditional result, we shall see that it sheds new light on 
 prior results, such as one of Bourgain \cite{bourgain} on real analytic vector fields. 
See Proposition~\ref{p.bourgain}, and the discussion of that Proposition. 

The authors are not aware of any conceptual obstacles to the following extension of the 
Theorem above to be true, namely that one can establish $ L ^{p}$ estimates, 
for all $ p>2$. As our argument currently stands, we could only prove this 
result for $ p$ sufficiently close to $ 2$, because of our currently crude understanding 
of the underlying orthogonality arguments. 

\begin{conjecture}\label{j.conditional}
Assume that Conjecture~\ref{j.laceyli2} holds
for a choice of  vector field $ v$ with  $ 1+\eta  >1$ derivatives,
then we have the inequalities below
\begin{equation}\label{e.Jcond2}
\norm  \operatorname H _{v, \epsilon }   .p. 
\lesssim  (1+ \log  \norm v. C ^{1+\eta  }.) ^2 \,, 
\qquad 2<p<\infty \,. 
\end{equation}
In this case, $ \epsilon = K/\norm v.C ^{1+\eta  }. $.
\end{conjecture}

For a brief remark on what is required to prove this conjecture, see
Remark~\ref{r.conditional}.

The results of  Christ, Nagel, Stein and Wainger \cite{MR2000j:42023} apply to certain vector
fields $v$.  This work is a beautiful culmination of the `geometric' approach to questions 
concerning the boundedness of Radon transforms. 
Earlier, a positive result for 
analytic vector fields followed from Nagel, Stein and Wainger \cite{MR81a:42027}.  
E.M.~Stein \cite{stein-icm}  specifically raised the question of the boundedness of $\operatorname H_v$ 
for smooth vector fields $v$. And the results of  D.~Phong and Stein \cites{MR88i:42028b
,MR88i:42028a} also give results about $\operatorname H_v$.  J.~Bourgain \cite{bourgain} considered 
real--analytic vector field.  N.~H.~Katz \cite{MR1979942} has  made an interesting contribution to
maximal function question.  Also see the partial  results of Carbery, Seeger, Wainger and Wright \cite{carbery}.

%


\chapter[Connections]{Connections to Besicovitch Set and Carleson's Theorem}

\section*{Besicovitch Set} 

The Besicovitch set is a compact set that contains a line segment of 
unit length in each direction in the plane. Anomalous constructions of 
such sets show that they can have very small measure. Indeed, given 
$ \epsilon >0$ one can select rectangles $ R_1 ,\dotsc,  R_n $,
with disjoint eccentricities, 
$ \lvert  \operatorname {EX}(R)\rvert \simeq n ^{-1} $, 
 and of unit length, so that 
$ \lvert  B\rvert\le \epsilon  $ for $ B:= \bigcup _{n=1} ^{n} R_j$. 
On the other hand, letting $e_j\in \operatorname {EX} (R_j)$,  
one has that the rectangles 
$R_j + e_j$  
 are essentially disjoint. See Figure~\ref{f.besi}. Call the `reach' of 
the Besicovitch set 
\begin{equation*}
\operatorname {Reach} := 
\bigcup _{j=1} ^{n} R_j + e_j\,.
\end{equation*}
This set has measure about one.  
On the Reach, one can define a vector field with points to a line segment 
contained in the Besicovitch set. Clearly, one has 
\begin{equation*}
\lvert  \operatorname H _{v} \mathbf 1_{B} (x)\rvert \simeq 1\,, 
\qquad x\in \operatorname {Reach}.
\end{equation*}
Further, constructions of this set permit one to take the vector field 
to be Lipschitz continuous of any index strictly less than one. And conversely, 
if one considers a Besicovitch set associated to a  vector 
field of sufficiently small Lipschitz norm, of index one, the corresponding Besicovitch 
set must have large measure. Thus, Lipschitz estimates are critical. 

\begin{figure}
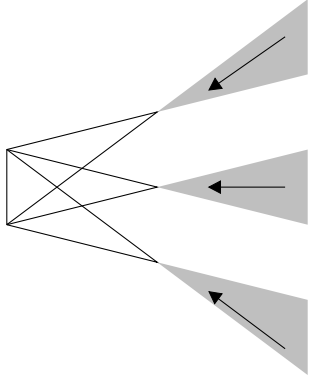

 \begin{pgfpicture}{0cm}{0cm}{5cm}{5.5cm}
 \begin{pgftranslate}{\pgfxy(0,.8)}

\pgfmoveto{\pgfxy(0,1)}
\pgflineto{\pgfxy(2,2.5)}
\pgflineto{\pgfxy(0,2)}
\pgflineto{\pgfxy(0,1)}
\pgfstroke
\pgfmoveto{\pgfxy(0,1)}
\pgflineto{\pgfxy(2,1.5)}
\pgflineto{\pgfxy(0,2)}
\pgfstroke
\pgfmoveto{\pgfxy(0,1)}
\pgflineto{\pgfxy(2,.5)}
\pgflineto{\pgfxy(0,2)}
\pgfstroke
{\color{lightgray}
\pgfmoveto{\pgfxy(2.,1.5)}
\pgflineto{\pgfxy(4,1)}
\pgflineto{\pgfxy(4,2)}
\pgflineto{\pgfxy(2.,1.5)}
\pgffill
\pgfmoveto{\pgfxy(2.,2.5)}
\pgflineto{\pgfxy(4,3)}
\pgflineto{\pgfxy(4,4)}
\pgflineto{\pgfxy(2.,2.5)}
\pgffill
\pgfmoveto{\pgfxy(2.,.5)}
\pgflineto{\pgfxy(4,0)}
\pgflineto{\pgfxy(4,-1)}
\pgflineto{\pgfxy(2.,.5)}
\pgffill}

 \pgfsetendarrow{\pgfarrowtriangle{4pt}} 
  \pgfline{\pgfxy(3.7,1.5)}{\pgfxy(2.7,1.5)} 
   \pgfline{\pgfxy(3.7,3.5)}{\pgfxy(2.7,2.8)} 
    \pgfline{\pgfxy(3.7, -.65)}{\pgfxy(2.7,0.1)} 
  \pgfclearendarrow
\end{pgftranslate}
  \end{pgfpicture}
\caption{A Besicovitch Set on the left, and it's Reach on the right.} 
\label{f.besi}
\end{figure}

\section*{The Kakeya Maximal Function} 

The Kakeya maximal function is typically defined as 
\begin{equation}\label{e.MK}
\operatorname M _{K,\epsilon } f (x) 
:=
\sup _{\lvert  \operatorname {EX} (R)\rvert \ge \epsilon }
\frac {\mathbf 1_{R} (x)} {\lvert  R\rvert } 
\int_R \lvert  f (y)\rvert \; dy \,, \qquad \epsilon >0\,.  
\end{equation}
One is forced to take $\epsilon>0$ due to the existence of the Besicovitch set.
It is a critical fact that the norm of this operator admits a norm bound 
on $L^2$ that is logarithmic in $\epsilon$. See C\'ordoba and Fefferman 
\cite{MR0476977}, and 
Str\"omberg \cites{MR0487260,MR0481883}. Subsequently, there have been several refinements 
of this observation, we cite only Nets H. Katz \cite{nets}, Alfonseca, Soria 
and Vargas \cite{MR1960122}, and Alfonseca \cite{MR1942421}. These papers contain additional 
references. 
For the $L^2$ norm, the following is the sharp result.

\begin{theorem}\label{t.strom} We have the estimate below valid for all $ 0<\epsilon <1$. 
\begin{equation*}
\norm \operatorname M _{K, \epsilon }.2\to 2. \lesssim 1+\log 1/ \epsilon \,. 
\end{equation*}
\end{theorem}

The standard example of taking $ f$ to be the indicator of a small disk show 
that the estimate above is sharp, and that the norm grows as an inverse power 
of $ \epsilon $ for $ 1<p<2$.

\section*{Carleson's Theorem} 

We explain the connection between the Hilbert transform on vector 
fields and Carleson's Theorem on the pointwise convergence of Fourier 
series. Since smooth functions have a convergent Fourier expansion, 
the main point of Carleson's Theorem is to provide for the control of 
an appropriate maximal function. We recall that maximal function in 
this Theorem.

\begin{carleson}
For all measurable functions $ N\;:\; \mathbb R \longrightarrow \mathbb R $, the operator 
below maps $ L ^2 $ into itself. 
\begin{equation*}
\mathcal C_N f (x) 
:= \textup{p.v.}
\int \operatorname e ^{i N (x) y} f (x-y) \frac {dy}y\,. 
\end{equation*}
The implied operator norm is independent of the choice of measurable $ N (x)$. 
\end{carleson}

For fixed function $ f$, an appropriate choice of $ N$ will give us 
\begin{equation*}
\sup _{N} \ABs{ \textup{p.v.} \int \operatorname e ^{i N  y} f (x-y) \frac {dy}y} 
\lesssim 
\lvert  \mathcal C_N f (x) \rvert\,.  
\end{equation*}
Thus, in the Theorem above we have simply linearized the supremum.  Also, 
we have stated the Theorem with the un-truncated integral.  The content of the 
Theorem is unchanged if we make a truncation of the integral, which we will do below.

\begin{figure}
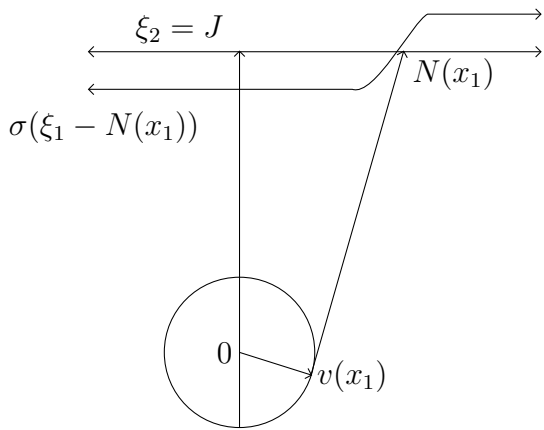

  \begin{pgfpicture}{0cm}{0cm}{8cm}{6cm}
   \pgfcircle[stroke]{\pgfxy(2,1)}{1cm}
    \pgfsetendarrow{\pgfarrowpointed}
   
   \pgfmoveto{\pgfxy(2,1)}
  \pgflineto{\pgfxy(2.95,.7)}
  \pgfstroke
  \pgfmoveto{\pgfxy(2,0)}
  \pgflineto{\pgfxy(2,5)}
  \pgfstroke
  
  \pgfmoveto{\pgfxy(2.95,.7)}
  \pgflineto{\pgfxy(4.18,5)}
  \pgfstroke
  
  \pgfsetstartarrow{\pgfarrowpointed}
  \pgfmoveto{\pgfxy(0,5)}
  \pgflineto{\pgfxy(6,5)}
  \pgfstroke
  \pgfputat{\pgfxy(1.7,1)}{\pgfbox[left,center]{$0$}}
  \pgfputat{\pgfxy(3.5,.9)}{\pgfbox[center,top]{$v(x_1)$}}
  
  \pgfputat{\pgfxy(4.3,4.7)}{\pgfbox[top,center]{$N(x_1)$}}

  \pgfputat{\pgfxy(1.2,5.3)}{\pgfbox[center,center]{$\xi_2=J$}}
    \pgfputat{\pgfxy(.2,4)}{\pgfbox[center,center]{$\sigma (\xi _1-N (x_1))$}}
\pgfmoveto{\pgfxy(0,4.5)}
\pgflineto{\pgfxy(3.5,4.5)}
\pgfcurveto{\pgfxy(3.8,4.4)}{\pgfxy(4.3,5.4)}{\pgfxy(4.5,5.5)}
\pgflineto{\pgfxy(6,5.5)}
\pgfstroke
\end{pgfpicture}   
\caption{Deducing Carleson's Theorem from Stein's Conjecture.}
\label{f.carleson}
\end{figure}

Let us now show how to deduce this Theorem from an appropriate 
bound on certain bound on Hilbert transforms on vector fields.
(This observation is apparently due to R.\thinspace Coifman from the 1970's.)

\begin{proposition}\label{p.Stein=>Carleson}
 Assume that we have, say, the bound
 \begin{equation*}
\norm \operatorname H _{v,1} .2\to 2. \lesssim 1\,, 
\end{equation*}
assuming that $ \norm v.C ^2 . \le 1$.  
It follows that the Carleson maximal operator is bounded on $L^2 (\mathbb R )$. 
\end{proposition}

\begin{proof}
The Proposition and the proof are only given in their most 
obvious formulation. 
Set $ \sigma (\xi )=\int _{-1} ^{1} \operatorname e ^{i \xi y}\frac {dy}y$. 
For a $ C ^{2}$ function $ N\;:\; \mathbb R \longrightarrow \mathbb R $ 
we deduce that the operator with symbol $ \sigma (\xi - N (x))$ 
maps $ L ^{2} (\mathbb R )$ into itself with norm that is independent 
of the $ C ^2 $ norm of the function $ N (x)$.  A standard limiting argument 
then permits one to conclude the same bound for all measurable choices of 
$ N (x)$, as is required for the deduction of Carleson's inequality. 

This argument is indicated in Figure~\ref{f.carleson}. Take 
the vector field to be $ v (x_1, x_2)=(1,-N (x_1)/n)$ where 
$ n$ is chosen 
much larger than the $ C ^2 $ norm of the function  $ N (x_1)$. 
Then, $ \operatorname H _{v,1}$ is bounded on $ L ^2 (\mathbb R ^2 )$ 
with norm bounded by an absolute constant. The symbol of $ H_ {v,1}$ is 
\begin{equation*}
\sigma (\xi _1, \xi _2)=\sigma (\xi _1- \xi _2 N (x_1)/n)\,.
\end{equation*}
The trace of this symbol along 
the line $\xi _2 = J$  defines a symbol of a bounded operator on $L^2 (\mathbb R )$. 
Taking $ J$ very large, we obtain a very good approximation to symbol 
$ \sigma (\xi _1- \xi _2 N (x_1)/n)$, deducing that it maps $L^2 (\mathbb R )$  
into itself with a bounded 
constant. Our proof is complete. 
\end{proof}

\section*{The Weak $ L ^2 $ Estimate in Theorem~\ref{t.laceyli1} is Sharp} 
An example shows that under the assumption that the vector field is measurable, the sharp
conclusion is that $\operatorname H_v \circ \operatorname S _1$ maps $L^2$ into
$L^{2,\infty}$. And a variant of the approach to Carleson's theorem by Lacey and Thiele
\cite{laceythiele} will prove this norm inequality. This method will also show, under
only the measurability assumption, that $\operatorname H_v{\operatorname S}_1$ maps $L^p$
into itself for $p>2$, as is shown by the current authors \cite{laceyli1}. The results and
techniques of that paper are critical to this one.


\chapter[Lipschitz Kakeya]{The Lipschitz Kakeya Maximal Function}

\section*{The Weak $ L ^2 $ Estimate} 

We prove Theorem~\ref{t.lipKakeya}, the weak $ L ^{2}$ estimate 
for the maximal function defined in \eqref{e.Mdef}, by suitably adapting 
classical covering lemma arguments.

\subsection*{The Covering Lemma Conditions} 

We adopt the covering lemma approach of 
C{\'o}rdoba and R.~Fefferman \cite{MR0476977}.  To this end, we regard the choice of vector field 
$ v$ and $ 0<\delta <1$ as fixed.  Let $ \mathcal R$ be any finite collection of rectangles 
obeying the conditions \eqref{e.shortlength} and $ \abs{ \name V R }\ge \delta \abs{ R}$. 
We show that $ \mathcal R$ has a decomposition into disjoint collections $ \mathcal R'$ 
and $ \mathcal R''$ for which these estimates hold.   
\begin{align}\label{e.2<1}
\NOrm \sum _{R\in \mathcal R'} \mathbf 1 _{R}.2.^2 &\lesssim \delta ^{-1 } 
\NOrm \sum _{R\in \mathcal R'} \mathbf 1 _{R}.1.\,,
\\  \label{e.bigcup}
\ABs{\bigcup _{R\in \mathcal R''}R} &\lesssim 
\NOrm \sum _{R\in \mathcal R'} \mathbf 1 _{ R}.1.
\end{align}
The first of these conditions is the stronger one, as it bounds the $ L^2$ norm squared 
by the $ L^1$ norm; the verification of it will occupy most of the proof. 

\smallskip 

Let us see how to deduce Theorem~\ref{t.lipKakeya}.  Take $ \lambda >0$ and $ f\in L^2$  which is non negative and 
 of norm one. Set $ \mathcal R$ to be all the rectangles  $ R$ 
 of prescribed maximum length as given in \eqref{e.shortlength}, 
 density with respect to the vector field, namely $ \abs{ \name V R}\ge \delta \abs{ R}$,  and  
\begin{equation*}
\int _{R} f(y)\; dy\ge{} \lambda \abs{ R} \,.
\end{equation*}
We should verify the weak type inequality 
\begin{equation} \label{e.wweak}
\lambda \ABs{\bigcup _{R\in \mathcal R}R} ^{1/2} \lesssim \delta ^{-1/2}\,.
\end{equation}

Apply the decomposition to $ \mathcal R$.  Observe that 
\begin{align*}
\lambda \NOrm \sum _{R\in \mathcal R'} \mathbf 1 _{R}.1. 
&\le \IP f, \sum _{R\in \mathcal R'} \mathbf 1 _{R},
\\
&\le \NOrm \sum _{R\in \mathcal R'} \mathbf 1 _{R}.2.
\\
& \lesssim \delta ^{-1/2} \NOrm \sum _{R\in \mathcal R'} \mathbf 1 _{R}.1. ^{1/2}\,.
\end{align*}
Here of course we have used \eqref{e.2<1}.  This implies that 
\begin{equation*}
\lambda \NOrm \sum _{R\in \mathcal R'} \mathbf 1 _{R}.1. ^{1/2} \lesssim  \delta ^{-1/2}.
\end{equation*}
Therefore clearly \eqref{e.wweak} holds for the collection $ \mathcal R'$. 

Concerning the collection $ \mathcal R''$, apply \eqref{e.bigcup} to see that  
\begin{align*}
\lambda \ABs{\bigcup _{R\in \mathcal R''}R} ^{1/2}  
\lesssim \lambda \NOrm \sum _{R\in \mathcal R'} \mathbf 1 _{R}.1. ^{1/2}
\lesssim \delta ^{-1/2}\,.
\end{align*}
This completes our proof of \eqref{e.wweak}.

\smallskip  

The remainder of the proof is devoted to  the proof of \eqref{e.2<1} and \eqref{e.bigcup}.

\subsection*{The Covering Lemma Estimates} 

\subsubsection*{Construction of $ \mathcal R'$ and $ \mathcal R''$.}
In the course of the proof, we will need several recursive procedures.  The 
first of these occurs in the selection of $ \mathcal R'$ and $ \mathcal R''$. 

We will have need of one large constant $ \kappa $, of the order of say $ 100$, but whose 
exact value does not concern us. Using this notation hides 
distracting terms.

Let $\operatorname M_{\kappa}$ be a maximal function given as 
\begin{equation*}
\operatorname M _{\kappa } f(x)=\sup _{s>0} \max \Bigl\{ 
 s ^{-2}\int _{x+sQ} \abs{ f(y)} \; dy\, ,\ \sup _{\omega \in \Omega } s ^{-1 }
 \int _{-s} ^{s} \abs{ f(x+ \sigma \omega  )} \; d \sigma \Bigr\}\,.
\end{equation*}
Here, $ Q$ is the unit square in plane, and $ \Omega $ is a set of uniformly distributed 
points on the unit circle of cardinality equal to $ \kappa $. 
It follows from the usual weak type bounds 
that this operator maps $ L^1(\mathbb R^2)$ into weak $ L^1(\mathbb R^2)$. 

\smallskip 

To initialize the recursive procedure, set 
\begin{align*}
\mathcal R'&\leftarrow \emptyset\,,
\\
\mathsf{STOCK} & \leftarrow \mathcal R\,.
\end{align*}

The main step is this while loop.  While $ \mathsf{STOCK}$ is not empty, 
select $ R \in \mathsf{STOCK}$ subject to the criteria that first it have a 
maximal  length $ \mathsf L (R)$, and second that it have minimal value 
of $ \abs{ \name EX R}$. 
Update 
\begin{equation*}
\mathcal R'\leftarrow \mathcal R'\cup\{R\}. 
\end{equation*}
Remove $ R$ from $ \mathsf{STOCK}$. As well, remove any rectangle $ R'\in \mathsf{STOCK}$
which is also contained in 
\begin{equation*}
\Bigl\{ \operatorname M_{\kappa} \sum _{R\in \mathcal R'} \mathbf 1 _{\kappa R}\ge{} \kappa ^{-1}
\Bigr\}.
\end{equation*}

As the collection $ \mathcal R$ is finite, the while loop will terminate, and at this point 
we set $ \mathcal R''{} \eqdef  \mathcal R- \mathcal R'$. 
In the course of the argument below, we will refer the order in which rectangles 
were added to $ \mathcal R'$. 

\smallskip

With this construction, it is obvious that \eqref{e.wweak} holds, with a bound 
that is a function of $ \kappa $.  Yet, $ \kappa $ is an absolute constant, so 
this dependence does not concern us. And so the 
rest of the proof is devoted to the verification of \eqref{e.2<1}. 

\smallskip

An important aspect of the qualitative nature of the interval of eccentricity is 
encoded into this algorithm.  We will choose $ \kappa $ so large that this is true: 
Consider two rectangles $ R$ and $ R'$ with $ R\cap R'\neq\emptyset$, 
$ \name L R\ge \name L {R'}$, $ \name W R\ge \name W {R'} $,  
$ \abs{\name {EX} {R}}\le{} \abs{\name {EX} {R'} }$ and $\name EX R\subset 10\name EX {R'} $
then we have    
\begin{equation}\label{e.ex==>}
R'\subset \kappa R\,.
\end{equation}
See Figure~\ref{f.kappa}.

 \begin{figure}
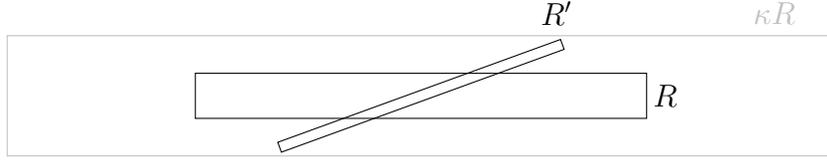
  
\begin{center}
  \begin{pgfpicture}{0cm}{0cm}{8cm}{4cm}
  \begin{pgftranslate}{\pgfxy(4,2.5)}
	\pgfrect[stroke]{\pgfxy(-3,-.3)}{\pgfxy(6,.6)}
	\pgfputat{\pgfxy(3.25,0)}{\pgfbox[center,center]{$ R$}}
	\begin{pgfrotateby}{\pgfdegree{20}}
	\pgfrect[stroke]{\pgfxy(-2,-.07)}{\pgfxy(4,.14)}
	\end{pgfrotateby}
	\pgfputat{\pgfxy(1.8,1.1)}{\pgfbox[center,center]{$ R'$}}
 {\color{lightgray} 
 	\pgfrect[stroke]{\pgfxy(-5.5,-.8)}{\pgfxy(11,1.6)}
	\pgfputat{\pgfxy(4.7,1.1)}{\pgfbox[center,center]{$\kappa R$}}
}
	\end{pgftranslate}
	\end{pgfpicture}
	\end{center} 
\caption{The rectangle $ R'$ would have been removed from $ \mathsf {STOCK}$ 
upon the selection of $ R$ as a member of $ \mathcal R'$. }
\label{f.kappa}
	\end{figure}

\subsection*{Uniform Estimates}

We estimate the left hand side of \eqref{e.2<1}.  In so doing we expand the square, 
and seek certain uniform estimates.  Expanding the square on the left hand side 
of \eqref{e.2<1}, we can estimate 
\begin{equation*}
\textup{l.h.s.\thinspace of } \eqref{e.2<1}
\le \sum _{R\in \mathcal R'} \abs{ R}+2 \sum _{(\rho ,R)\in \mathcal P} \abs{ \rho \cap R}
\end{equation*}
where $ \mathcal P$ consists of all pairs $ (\rho ,R)\in \mathcal R' \times \mathcal R'$ 
such that $ \rho \cap R\neq \emptyset$, and 
$ \rho $ was selected to be a member of $ \mathcal R'$ before 
$ R$ was.  It is then automatic that $ \mathsf L (R )\le \mathsf L (\rho )$.
And since the density of all tiles is positive, it follows that 
$ \operatorname {dist} (\mathsf {EX} (\rho ), \mathsf {EX} (R))
\le 2 \norm v.\textup{Lip}. \mathsf L (\rho )<\tfrac 1 {50}$.

We will split up the collection $ \mathcal P$ into 
sub-collections $ \{\mathcal S _{R} \;:\; R\in \mathcal R'\}$ 
and $ \{\mathcal T _{\rho }\;:\; \rho \in \mathcal R'\}$.

\medskip 

For a rectangle $ R\in \mathcal R'$, we take 
$ \mathcal S _{R}$ to consist of all rectangles $ \rho $ such that 
(a) $ (\rho ,R)\in \mathcal P$;  and (b) $\name EX \rho \subset 
10\name EX R $. 
We assert that 
\begin{equation}\label{e.uni1}
\sum_{\rho \in \mathcal S_R} \abs{ R\cap \rho } \le 
\abs{ R},\qquad R\in \mathcal R. 
\end{equation}

This estimate is in fact easily available to us. 
Since the rectangles $ \rho \in \mathcal S_R$ were selected to be 
in $ \mathcal R'$ before $ R$ was, we cannot have the inclusion 
\begin{equation} \label{e.Xuni}
R\subset \Bigl\{ \operatorname M_{\kappa} \sum _{\rho \in \mathcal S_R} \mathbf 1 _{\kappa \rho } 
> \kappa ^{-1} \Bigr\}\,.
\end{equation}
Now the  rectangle $ \rho $ are  also longer.  Thus, if \eqref{e.uni1} 
 does not hold,
 we would compute the maximal function of 
 \begin{equation*}
\sum _{\rho \in \mathcal S_R} \mathbf 1 _{\kappa \rho }
\end{equation*}
in a direction which is close, within an error of $ 2 \pi /\kappa $, of being orthogonal 
to the long direction of $ R$.  In this way, we will contradict \eqref{e.Xuni}.

\medskip 

The second uniform estimate that we need is as follows.  
For fixed $ \rho $, set $ \mathcal T _{\rho }$ to be the set of all rectangles 
$ R$ such that 
(a) $ (\rho ,R)\in \mathcal P$ and (b) 
$\name EX \rho \not\subset  10\name EX R $. 
We assert that 
\begin{equation}\label{e.uni2}
\sum_{R\in \mathcal T _{\rho }} 
\abs{ R\cap \rho } \lesssim \delta ^{-1} \abs{ \rho },\qquad \rho \in \mathcal R'. 
\end{equation}
This proof of this inequality is more involved, and taken up in the next subsection. 

\begin{remark}\label{r.rho} In the proof of \eqref{e.uni2}, it is not 
necessary that $ \rho \in \mathcal R'$.  Writing $ \rho =I _{\rho } \times J _{\rho }$, 
in the coordinate basis $ \operatorname e$ and $ \operatorname e _{\perp}$, 
we could take any rectangle of the form $ I \times J _{\rho }$. 
\end{remark}

\medskip 

These two estimates conclude the proof of \eqref{e.2<1}. 
For any two distinct rectangles $\rho, R\in \mathcal P$,
we will have either $ \rho \in \mathcal S_R$ or $ R\in \mathcal T _{\rho }$.
Thus \eqref{e.2<1} follows by summing \eqref{e.uni1} on $ R$ and \eqref{e.uni2}
on $ \rho $. 

\subsection*{The Proof of \eqref{e.uni2}} 

We do not need this Lemma for the proof of \eqref{e.uni2}, but this is the most 
convenient place to prove it.

\begin{lemma}\label{l.fixscale}
Let $ \mathcal S$ be any finite collection of rectangles with 
$ \name L R\le2\name L {R'}$, and with 
$ \abs{\name V R }\ge{} \delta \abs R$ for all $ R,R'\in \mathcal S$.  Then it is the 
case that 
\begin{equation}\label{folap}
\NOrm \sum_{R\in \mathcal S}
 \mathbf 1 _{R} .\infty . 
\le2\delta ^{-1}\,.
\end{equation}
\end{lemma} 

\begin{proof}
Fix a point $ x$ at which we give an upper bound on the sum above. 
Let $C(x)$ be any circle centered at $x$. We shall 
show that there exists at most one $R$ in $\mathcal S$ 
such that $\name V R\cap C(x)\neq \emptyset$. 
By the assumption   $ \abs{ \name V R }\ge{} \delta \abs{ R}$ 
this proves the Lemma.

\smallskip 
We prove this  last claim by contradiction of the 
Lipschitz assumption on the vector field 
$ v$. Assume that there exist 
at least two rectangles $ R,R'\in \mathcal S$ for which the 
sets $ \name V R $ and $ \name V {R'}$ intersect $ C(x)$. 
Thus there exist $y$ and $y'$ in $C(x)$ such that 
$v(y)\in \name EX R$ and $v(y')\in \name EX {R'} $. 
Since $v$ is Lipschitz, we have 
\begin{align*}
\abs{v(y)-v(y')}&\leq \norm v. \text{Lip}.\abs{y-y'}
\\
&\leq 4 \norm v. \text{Lip}. \name L R 
            \abs{v(y)-v(y')}\,,
\end{align*}
but this is a contradiction to our assumption \eqref{e.shortlength}.  
See Figure~\ref{f.radial}. 
\end{proof}

\begin{figure}
	\begin{pgfpicture}{0cm}{0cm}{9cm}{4cm}
	\begin{pgftranslate}{\pgfxy(2,.3)}
	\begin{pgfrotateby}{\pgfdegree{-80}}
 \ThinRect{-10} 	\ThinRect{0} 
 \pgfputat{\pgfxy(-.5,3)}{\pgfbox[center,center]{$\ldots$}}
 \ThinRect{30}	\ThinRect{40} \ThinRect{50}
 \pgfsetendarrow{\pgfarrowtriangle{4pt}}
 
 	\begin{pgfrotateby}{\pgfdegree{-10}}
	 \pgfmoveto{\pgfxy(.3,3.5)}\pgflineto{\pgfxy(.3,4.5)}\pgfstroke
	\end{pgfrotateby}
	\begin{pgfrotateby}{\pgfdegree{40}}
	\pgfmoveto{\pgfxy(.3,3.5)}\pgflineto{\pgfxy(.3,4.5)}\pgfstroke
	\end{pgfrotateby}	
	\end{pgfrotateby}
	\end{pgftranslate}
	\end{pgfpicture}

\caption{Proof of Lemma~\ref{l.fixscale}}
\label{f.radial}
\end{figure}

We fix $ \rho $, and begin by making a decomposition of the collection $ \mathcal T _{\rho }$. 
Suppose that the coordinate axes for $ \rho $ are given by $ \operatorname e _{\rho }$, 
associated with the long side of $ R$, and 
$ \operatorname e ^{\perp} _{\rho }$, with the short side.  Write the rectangle as a product of intervals $ I _{\rho }\times J$, 
where $ \abs{ I _{\rho }}=\name L \rho $.  
Denote one of the endpoints of $ J$ as $ \alpha $.  
See Figure~\ref{f.rho}.

%
%
 \begin{figure}
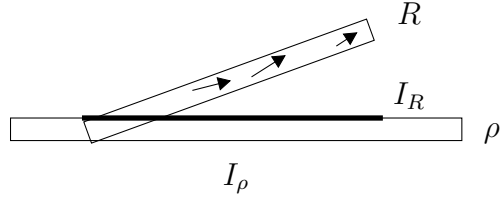
  
\begin{center} 
 \begin{pgfpicture}{0cm}{0cm}{8cm}{4cm}
\begin{pgftranslate}{\pgfxy(2,1)}
{\thinrect 0600}
 \pgfputat{\pgfxy(6.4,0.1)}{\pgfbox[center,center]{$\rho $}}
  \pgfputat{\pgfxy(3,-0.5)}{\pgfbox[center,center]{$I_\rho  $}}
 				\pgfsetendarrow{\pgfarrowtriangle{4pt}}
  		\begin{pgfrotateby}{\pgfdegree{20}}
	\thinrect 041{-.4}
			\pgfmoveto{\pgfxy(4.5,-.3)}\pgflineto{\pgfxy(4.8,-.23)}\pgfstroke
			\pgfmoveto{\pgfxy(3.3,-.3)}\pgflineto{\pgfxy(3.8,-.19)}\pgfstroke
			\pgfmoveto{\pgfxy(2.5,-.2)}\pgflineto{\pgfxy(3,-.25)}\pgfstroke
	\end{pgfrotateby}
	\pgfputat{\pgfxy(5.3,1.7)}{\pgfbox[center,center]{$R$}}
	\pgfclearendarrow
  \pgfsetlinewidth{2pt}
 \pgfmoveto{\pgfxy(.95,.3)}\pgflineto{\pgfxy(4.95,.3)}\pgfstroke
 \pgfputat{\pgfxy(5.3,.6)}{\pgfbox[center,center]{$I_R$}}

 \end{pgftranslate}  
 \end{pgfpicture}
 \end{center} 
 \caption{Notation for the proof of \eqref{e.uni2}.}
\label{f.rho}
 \end{figure}

For rectangles $ R\in \mathcal T _{\rho }$, let $ I_R$ denote the orthogonal projection 
$ R$ onto the line segment $ 2I _{\rho }\times \{\alpha \}$.  Subsequently, we will consider 
different subsets of this line segment.  The first of these is as follows. 
For $ R\in \mathcal T _{\rho }$, let $ \mathsf V_R$ be the projection 
of the set $ \name V R $ onto $2I _{\rho }\times \{\alpha \}$.  
The angle $ \theta $ between $ \rho $ and $ R$ is at most 
$ \lvert  \theta \rvert \le 2\norm v. \textup{Lip}. \name L {\rho } \le \tfrac 1 {50}  $. 
It follows that 
\begin{equation} \label{e.B}
\tfrac 12 \name L R \le\abs{ I_R}\le 2 \name L R ,
\quad\text{and}\quad
\delta \name L R\lesssim  \abs{ \mathsf V_R}. 
\end{equation}

A recursive mechanism is used to decompose $ \mathcal T _{\rho }$.  Initialize 
\begin{align*}
\mathsf{STOCK} & \leftarrow \mathcal T_\rho \,,
\\
\mathcal U & \leftarrow \emptyset \,.
\end{align*}
While $ \mathsf{STOCK}\neq\emptyset$ select $ R\in \mathsf{STOCK}$ of maximal 
length. Update 
\begin{equation}\label{e.UU}
\begin{split}
\mathcal U&\leftarrow \mathcal U\cup\{R\},
\\
\mathcal U(R) &\leftarrow \{R'\in \mathsf{STOCK}\;:\; \mathsf V_R\cap \mathsf V_{R'}\neq\emptyset\}. 
\\
\mathsf{STOCK}&\leftarrow\mathsf{STOCK}-\mathcal U(R). 
\end{split}
\end{equation}
When this while loop stops, it is the case that $ \mathcal T_ \rho =\bigcup _{R\in \mathcal U} 
\mathcal U(R)$.

\smallskip

With this construction, the sets $ \{ \mathsf V_R \;:\; R\in \mathcal U\} $ are 
disjoint.  By \eqref{e.B}, we have 
\begin{equation}\label{e.U}
\sum _{R\in \mathcal U} \name L R \lesssim \delta ^{-1} \name L \rho \,.
\end{equation}
The main point, is then to verify the uniform estimate 
\begin{equation}\label{e.uni3}
\sum _{R'\in \mathcal U(R)} \abs{ R'\cap \rho } 
\lesssim  \name L R\cdot \name W \rho  \,, 
\qquad R\in \mathcal U. 
\end{equation}
Note that both estimates immediately imply \eqref{e.uni2}.

\subsubsection*{Proof of \eqref{e.uni3}.}

There are three important, and more technical, facts to observe about the collections 
$ \mathcal U(R)$.

\begin{figure}
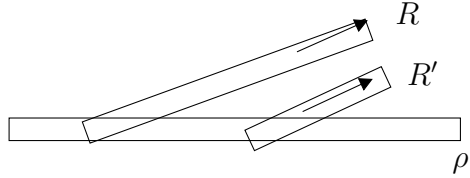
   
\begin{center} 
 \begin{pgfpicture}{0cm}{0cm}{7.5cm}{4cm}
\begin{pgftranslate}{\pgfxy(1,1)}
 \thinrect 0600
 \pgfputat{\pgfxy(6,-.3)}{\pgfbox[center,center]{$\rho $}}
 \pgfsetendarrow{\pgfarrowtriangle{4pt}}
 	\begin{pgfrotateby}{\pgfdegree{20}}
	\thinrect 041{-.4}
\pgfmoveto{\pgfxy(4,-.2)}\pgflineto{\pgfxy(5,-.13)}\pgfstroke
			\end{pgfrotateby}
	\pgfputat{\pgfxy(5.3,1.7)}{\pgfbox[center,center]{$R$}}
		 	\begin{pgfrotateby}{\pgfdegree{25}}
	\thinrect 02{2.9}{-1.5}
\pgfmoveto{\pgfxy(3.7,-1.3)}\pgflineto{\pgfxy(4.7,-1.3)}\pgfstroke
		\end{pgfrotateby}
		\pgfputat{\pgfxy(5.5,.9)}{\pgfbox[center,center]{$R'$}}
 \end{pgftranslate}  
 \end{pgfpicture}
 \end{center} \caption{Proof of Lemma~\ref{l.samedirection}:  The rectangles $R,R'\in \mathcal U(\rho ) $, 
 and so the angles $ R$ and $ R'$ form with $ \rho $ are nearly the same. }
 \label{f.samedirection}
 \end{figure}

For any rectangle $ R'\in \mathcal U(R)$, denote its coordinate 
axes as $ \operatorname e _{R'}$ and $ \operatorname e _{R'} ^{\perp}$, 
associated to the long and short sides of $ R'$ respectively.

\begin{lemma}\label{l.samedirection}
For any rectangle $ R'\in \mathcal U (R)$ we have  
\begin{equation*}
\abs{ \operatorname e _{R'}-\operatorname e _{R}} \le \tfrac 12 
\abs{ \operatorname e _{\rho }-\operatorname e _{R}}
\end{equation*}
\end{lemma}

\begin{proof}
There are by construction, points $ x\in \name V R$ and $ x'\in  \name V {R'} $ which get projected 
to the same point on the line segment $ I_\rho \times \{\alpha \}$. 
See Figure~\ref{f.samedirection}. Observe that 
\begin{align*}
\abs{ \operatorname e _{R'}-\operatorname e _{R}}&\le
\abs{ \name EX {R'} }+\abs{\name  EX R}+\abs{ v(x')-v(x)}
\\
&\le 
\abs{ \name EX {R'} }+\abs{\name  EX R}+\norm v.\text{Lip}. \cdot \name L R \cdot \abs{ \operatorname e _{\rho }-\operatorname e _{R}}
\\
&\le{} \abs{ \name EX {R'} }+\abs{\name  EX R}+\tfrac1{100} \abs{ \operatorname e _{\rho }-\operatorname e _{R}}
\end{align*}
Now, $ \abs{ \name EX R}\le{} \tfrac1{5} \abs{ \operatorname e _{\rho }-\operatorname e _{R}}$, 
else we would have $ \rho \in \mathcal S _{R}$. 
Likewise, $\abs{ \name EX {R'} }\le{} \tfrac1{5} \abs{ \operatorname e _{R' }-\operatorname e _{R}} $.
And this proves the desired inequality. 
 
\end{proof}

\begin{lemma}\label{l.geo}
Suppose that there is an interval $I\subset I _{\rho }$ such that
\begin{equation}\label{1}
\sum_{\substack{R'\in \mathcal U(R )\\ \name L {R'} \geq 8\abs I}} \abs{R'\cap I\times J} 
\geq     \abs{I\times J}\,.
\end{equation}
Then 
there is no $R''\in \mathcal U(R )$ such that $ \name L {R''} < \abs{I}$ and 
$  R''\cap 4I\times J\neq\emptyset  $. 
\end{lemma}

\begin{proof}
There is a natural angle $ \theta $ between the rectangles $ \rho $ and $ R$, 
which we can assume is positive, and is 
given by $ \abs{\operatorname e _{\rho }-\operatorname e _{R}}$.  
Notice that we have $ \theta\ge 10\abs{\name EX R} $, else we would have $ \rho \in \mathcal S _R$, 
which contradicts our construction.

Moreover, there is an important consequence of  Lemma~\ref{l.samedirection}: For any 
$ R'\in \mathcal U(R)$, there is a natural angle $ \theta '$ between $ R'$ and $ \rho $. 
These two angles are close.  For our purposes below, these two 
angles can be regarded as the same.

For any $ R'\in \mathcal U(R)$, we will have 
\begin{align*}
\frac{\abs{ \kappa R'\cap \rho }} {\abs{I \times J}} 
&\simeq \kappa \frac{ \name W {R'} \cdot \name W \rho }{\theta \abs{ I } \name W \rho  }
\\
&= \kappa \frac{\name W {R'} }{\theta \cdot  \abs {I } }.
\end{align*}

Recall $\operatorname M_{\kappa}$ is larger than the maximal function over $ \kappa $ uniformly 
distributed directions. Choose a direction $\operatorname e'$ from this set of  $ \kappa $ directions
that is closest to $\operatorname e ^\perp _{\rho }$. Take a line segment $\Lambda  $ in
direction $\operatorname e'$  of length  $\kappa \theta \abs{ I} $, 
and the center of $ \Lambda $ is in $4I\times J$.  See Figure~\ref{f.5kappa}. 
Then we have 
\begin{align*}
\frac{\abs{\kappa R'\cap  \Lambda }}{\abs{\Lambda }}& \geq 
\frac{ \name W {R'}   } {  \theta\cdot  \abs I  }
\end{align*}
Thus by our assumption \eqref{1},  
$$
\frac1 {\abs{ \Lambda }} \sum _{R'\in \mathcal U(R)} \abs{ \kappa R'\cap \Lambda } \gtrsim  1 \,.
$$
That is, any of the lines $ \Lambda $ are contained in the set 
\begin{equation*}
\Bigl\{  \operatorname M _{\kappa }\sum _{R\in \mathcal R'} \mathbf 1 _{\kappa R'}>\kappa ^{-1} \Bigr\}.  
\end{equation*}

\smallskip

 \begin{figure}   
\begin{center} 
\begin{pgfpicture}{0cm}{0cm}{5cm}{4cm}
\begin{pgftranslate}{\pgfxy(1.5,0)}
%
{\color{gray} \pgfrect[stroke]{\pgfxy(-3.3,0.5)}{\pgfxy(6.6,.2)} }
\pgfputat{\pgfxy(-2.7,1)}{\pgfbox[center,center]{$4I\times J$}}
\pgfrect[stroke]{\pgfxy(-1.3,0.5)}{\pgfxy(2.3,.2)}
\pgfputat{\pgfxy(.8,.2)}{\pgfbox[center,center]{$I\times J$}}
\pgfsetlinewidth{0.85pt}
\pgfline{\pgfxy(2.1,-0.7)}{\pgfxy(2.1,2.0)}
\pgfputat{\pgfxy(2.45,0)}{\pgfbox[center,center]{$\Lambda $}}
\begin{pgftranslate}{\pgfxy(-1,.5)}
\begin{pgfrotateby}{\pgfdegree{15}}
\pgfrect[stroke]{\pgfxy(-3,-.1)}{\pgfxy(8,.2)}
\end{pgfrotateby}
\end{pgftranslate}
\pgfputat{\pgfxy(3.5,1.4)}{\pgfbox[center,center]{$\kappa R'$}}
\end{pgftranslate}
\end{pgfpicture}
\end{center}
\caption{}
\label{f.5kappa}
\end{figure}

Clearly our construction does not permit 
any rectangle $ R''\in \mathcal U(R) $ contained in this set. 
To conclude the proof of our Lemma, we seek a contradiction.  Suppose that there 
is an $ R''\in \mathcal U(R)$ with $ \name L {R''}<\abs{ I}$ and $ R''$ intersects 
$ 2I \times J$.  The range of line segments $ \Lambda  $ we can permit is however quite 
broad.  The only possibility permitted to us is that the rectangle $ R''$ is quite 
wide.  We must have 
\begin{equation*}
\name W {R''}\ge{}\tfrac 14 \abs{  \Lambda  } =\tfrac \kappa 4 \cdot \theta \cdot  \abs I .  
\end{equation*}
This however forces us to have $ \abs{\name EX {R''} }\ge \tfrac \kappa 4  {\theta} $. 
And this implies that $ \rho \in \mathcal S _{R''}$, as in \eqref{e.uni1}.  This is 
the desired contradiction.

\end{proof}

Our third and final fact about the collection $ \mathcal U(R)$ is a consequence 
of Lemma~\ref{l.samedirection} and a geometric observation of 
J.-O.~Stromberg \cite[Lemma 2, p. 400]{MR0481883}.

\begin{lemma}\label{l.stromberg}  For any interval $ I\subset I _{R}$ we have the inequality 
\begin{equation}\label{e.removeafew}
\sum _{\substack{R'\in \mathcal U(R)\\ \name L {R'}\le \abs I \le 
\sqrt\kappa \name L {R'} }}
\abs{ R'\cap I\times J } \le 5\abs I \cdot \name W \rho  \,.
\end{equation}
\end{lemma}

 \begin{figure}
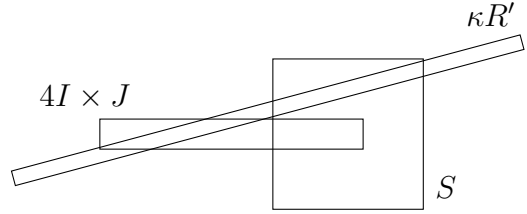
   
\begin{center} 
\begin{pgfpicture}{0cm}{0cm}{5cm}{4cm}
\begin{pgftranslate}{\pgfxy(1.5,0)}
\pgfputat{\pgfxy(-1.5,1.2)}{\pgfbox[center,center]{$4I\times J$}}
\pgfrect[stroke]{\pgfxy(-1.3,0.5)}{\pgfxy(3.5,.4)}
\pgfrect[stroke]{\pgfxy(1.0,-0.3)}{\pgfxy(2,2)}
\pgfputat{\pgfxy(3.3,0)}{\pgfbox[center,center]{$S$}}
\begin{pgftranslate}{\pgfxy(-1,.5)}
\begin{pgfrotateby}{\pgfdegree{15}}
\pgfrect[stroke]{\pgfxy(-1.5,-.1)}{\pgfxy(7,.2)}
\end{pgfrotateby}
\end{pgftranslate}
\pgfputat{\pgfxy(3.9,2.3)}{\pgfbox[center,center]{$\kappa R'$}}
\end{pgftranslate}
\end{pgfpicture}
\end{center}
\caption{The proof of Lemma~\ref{l.stromberg}} 
\label{f.strom}
\end{figure}

\begin{proof}  
For each point $ x\in 4 I\times J$, consider the square $ S$ centered at $ x$ of 
side length equal to $  \sqrt \kappa \cdot  \abs I 
\cdot \abs {\operatorname e _R - \operatorname e _\rho }$.
See Figure~\ref{f.strom}. 
It is Stromberg's observation 
that for $ R'\in \mathcal U(R)$ we have 
\begin{equation*}
\frac{ \abs{\kappa R'\cap I\times J } }{\abs{ I\times J}} \simeq  \frac {\abs{S\cap \kappa R' }}{\abs {S} }
\end{equation*}
with the implied constant  independent of $ \kappa $. 
Indeed, by Lemma~\ref{l.samedirection}, we have that 
\begin{align*}
\frac{ \abs{\kappa R'\cap I\times J } }{\abs{ I\times J}}
& \simeq  \frac{ \kappa \name W {R'} }
	{\abs{ \operatorname e _{R}-\operatorname e _{\rho } }\cdot \abs I} 
\\
& \simeq \frac{ \kappa \name W {R'} \cdot \abs I 
\cdot \abs {\operatorname e _R - \operatorname e _\rho }}
{(\abs{ \operatorname e _{R}-\operatorname e _{\rho } }\cdot \abs I)^2}  
\\
&\simeq \frac {\abs{S\cap \kappa R' }}{\abs {S} }\,,
\end{align*}
as claimed.

Now, assume that \eqref{e.removeafew} does not hold and seek a contradiction.  
Let $ \mathcal U'\subset \mathcal U(R)$ denote the collection of rectangles $ R'$ 
over which the sum is made in \eqref{e.removeafew}.  The rectangles in $ \mathcal U'$ 
were added in some order to the collection $ \mathcal R'$, and in particular 
there is a rectangle $ R_0\in \mathcal U'$ that was the last to be added to 
$ \mathcal U'$.  Let $ \mathcal U''$ be the 
collection $ \mathcal U'-\{R_0\}$. We certainly have 
\begin{equation*}
\sum _{R'\in \mathcal U''} \abs{ R''\cap I\times J}\ge{}4\abs{ I\times J}. 
\end{equation*}

Since we cannot have $ \rho \in \mathcal S _{R_0}$, Stromberg's observation implies that 
\begin{equation*}
R_0\subset\Bigl\{ \operatorname M _{\kappa } \sum _{R'\in \mathcal U''} \mathbf 1 _{\kappa R'}> 
\kappa ^{-
1}\Bigr\}\,.
\end{equation*}
Here, we rely upon the fact that the maximal function  $ \operatorname M _{\kappa }$ is larger than the usual maximal 
function over squares. 
But this is a contradiction to our construction, and so the proof is complete.  
\end{proof}

The principal line of reasoning to prove \eqref{e.uni3} can now begin with it's 
initial recursive procedure.  Initialize 
\begin{equation*}
C (R') \leftarrow R'\cap \rho \,. 
\end{equation*}
We are to bound the sum  $ \sum _{R'\in \mathcal U (R)} \lvert  C (R')\rvert $.  Initialize 
a collection of subintervals of $ I_R$ to be 
\begin{equation*}
\mathcal I\leftarrow \emptyset 
\end{equation*}

\textsf{WHILE} there is an interval $ I\subset I_R $ satisfying 
\begin{gather} \label{e..f}
\sum _{R'\in \mathcal V (I)} \abs{  C (R')\cap I \times J } 
\ge 40\abs I \cdot \name W \rho  \,, 
\\ \label{e..g}
\mathcal V (I) = \{ R' \in \mathcal U (R)
 \mid  \lvert   C (R')\cap I \times J\rvert\neq \emptyset \,,\   
\name L {R'} \ge 8\abs I \} \,, 
\end{gather}
we take $ I$ to be an interval of maximal length with this property, and update 
\begin{gather*}
\mathcal I\leftarrow \mathcal I\cup \{I\}\,;
\\
C (R',I)= C (R')\cap I \times J \,, \qquad R'\in \mathcal V (I); 
\\
C (R') \leftarrow C (R') - I \times J, \qquad R'\in \mathcal V (I) \,. 
\end{gather*}
[We remark that this last updating is not needed in the most important special case 
when all rectangles have the same width.  But the case we are considering, rectangles 
can have variable widths, so that $ \lvert  C (R)\rvert $ can be much larger than 
any $ \lvert  I\rvert \cdot \lvert  J\rvert  $ that would arise from this algorithm.] 

Once the \textsf{WHILE} loop stops, we have 
\begin{equation*}
R'\cap \rho = C (R')\cup  \bigcup \{ C (R',I)\mid I\in \mathcal I\,,\ R'\in \mathcal V (I)\}\,. 
\end{equation*}
Here the union is over pairwise disjoint sets.  

\smallskip 
We first consider the collection of sets $ \{C (R') \mid R' \in \mathcal U (R)\}$ 
that remain after the \textsf{WHILE} loop has finished.  
Since we must not have $ R'\subset 1/4 \kappa \cdot \rho $, it follows 
that the minimum value of $ \name L {R'}$ is $ \tfrac 1 {32}  \name W \rho $.  Thus, if 
in \eqref{e..f}, 
we consider an interval $ I$ of  length $ \tfrac 1 {256}  
\name W \rho$, the condition $ \name L {R'} \ge 8\abs I  $ 
in the definition of $ \mathcal V (I)$ in \eqref{e..g} is vacuous.  Thus, we necessarily have 
\begin{equation*}
\sum _{R'\in \mathcal V (I)} \abs{  C (R')\cap I \times J } 
\le 40\abs I \cdot \name W \rho \,. 
\end{equation*}
For if this inequality failed, the \textsf{WHILE} loop would not have stopped.  
We can partition $ I _R$ by intervals of length close to $ \tfrac 1 {256}  \name W \rho$, 
showing that we have 
\begin{equation*}
\sum _{R'\in \mathcal U (R)} \lvert  C (R')\rvert \lesssim \lvert  I_R \rvert \cdot \name W
\rho \,.   
\end{equation*}

\smallskip

Turning to the central components of the argument,  namely the bound for the 
terms associated with the intervals in $ \mathcal I$, consider $ I\in \mathcal I$.  
The inequality \eqref{e..f} and Lemma~\ref{l.stromberg}  implies that each $ I \in \mathcal I$ must 
have length $ \abs{ I}\le \kappa ^{-1/2} \abs{ I_\rho }$. 
But we choose intervals in $ I$ to be of maximal length. 
Thus, 
\begin{equation} \label{e.V}
\begin{split}
\sum _{R'\in \mathcal V(I)} \abs{ C(R',I) } 
\le 100\cdot  \abs I \cdot \name W \rho  \,.
\end{split}
\end{equation}
Indeed, suppose this  last inequality fails.  Let $ I\subset \widetilde I\subset I _{\rho }$ 
be an interval twice as long as $ I$.  By Lemma~\ref{l.stromberg}, we conclude that 
\begin{equation*}
\sum _{\substack{R'\in \mathcal V(I)\\ \name L {R'}\le 8 \lvert  \widetilde I\rvert  }} 
\abs{ R'\cap \widetilde I \times J} 
\le 10  \abs I \cdot \name W \rho  \,.
\end{equation*}
Notice that we are restricting the sum on the left by the length of $ \lvert  \widetilde
I\rvert $. 
Therefore, we have the inequalities 
\begin{align*}
\sum _{\substack{R'\in \mathcal V(I)\\ \name L {R'}>8 \lvert  \widetilde I\rvert  }} \abs{ C (R',I)} 
\ge 90\cdot  \abs I \cdot \name W \rho > 40 \cdot \abs {\widetilde I} \cdot \name W \rho  \,.
\end{align*}
That is, $ \widetilde I$ would have been selected, contradicting our construction. 

\medskip 
 Lemmas~\ref{l.geo}  and~\ref{l.stromberg} place significant restrictions on  
the collection of intervals $ \mathcal I$.  If we have $ I\neq I'\in \mathcal I$
with $ \tfrac 32 I\cap \tfrac 32 I'\neq\emptyset$, then we must have e.g.~$\sqrt \kappa \abs{ I'}<\abs{ I} $, 
as follows from Lemma~\ref{l.stromberg}. 
Moreover, $ \mathcal V (I')$ must contain a rectangle $R'$ with $\name L {R'}< \abs{I}$. 
But this contradicts Lemma~\ref{l.geo}.

Therefore, we must have 
\begin{equation*}
\sum _{I\in \mathcal I} \abs{ I} \lesssim \abs{ I_R} \lesssim \mathsf L (R). 
\end{equation*}
With \eqref{e.V}, this completes the proof of \eqref{e.uni3}.

\section*{An Obstacle to an $ L^p$ estimate, for $ 1<p<2$} 

We address one of the main conjectures of this memoir, namely Conjecture~\ref{j.laceyli1}. 
Let us first observe 

\begin{proposition}\label{p.3/2}  We have the estimate below valid for 
all $ 0<\mathsf w<\norm v. \textup{Lip}.$. 
\begin{equation*}
\sup _{\lambda >0}\lambda \lvert  \{ \operatorname M _{v, \mathsf w} f > \lambda \}\rvert ^{2/3}
\lesssim 
\delta ^{-1/3} (1 + \log w ^{-1}  \norm v. \textup{Lip}. ) ^{1/3} \norm f. 3/2. 
\end{equation*}

\end{proposition}

\begin{proof}
Let $ \norm v. \textup{Lip}.=1$. 
This just relies upon the fact that  with $ 0<\mathsf w < \tfrac12 $ fixed, there are 
only about $ \log 1/ \mathsf w $ possible values of $ \name L R $.  This leads very easily 
to the following two estimates.  Following the earlier argument, consider  an 
arbitrary collection of rectangles $ \mathcal R$  with 
each $ R\in \mathcal R$ satisfying 
 \eqref{e.shortlength} and $ \abs{ \name V R }\ge \delta \abs{ R}$. 
We can then  decompose $ \mathcal R$ into disjoint collections $ \mathcal R'$ 
and $ \mathcal R''$ for which these estimates hold.   
\begin{align}\label{e.V2<1}
\NOrm \sum _{R\in \mathcal R'} \mathbf 1 _{R}.3.^3 &\lesssim \delta ^{-1 } 
(\log 1/\mathsf w)
\NOrm \sum _{R\in \mathcal R'} \mathbf 1 _{R}.1.\,,
\\  \label{e.Vbigcup}
\ABs{\bigcup _{R\in \mathcal R''}R} &\lesssim 
\NOrm \sum _{R\in \mathcal R'} \mathbf 1 _{ R}.1.
\end{align}
Compare to \eqref{e.2<1} and \eqref{e.bigcup}.  Following the same line of 
reasoning that was used to prove \eqref{e.wweak}, we prove our Proposition. 

\end{proof}

We can devise proofs of smaller bounds on the norm of the maximal function 
than that given by this proposition.  But no argument that we can find avoids the 
some logarithmic term in the width of the rectangle.  Let us illustrate the 
difficulty in the estimate with an object pointed out to us by Ciprian Demeter.  
We term it a \emph{pocketknife}, and it is pictured in Figure~\ref{f.pocketknife}. 

\begin{figure}
\begin{center}
\begin{tikzpicture}
\draw (-1,0) rectangle  (6,0.1) node[below] {handle};
\foreach \ang/\length in {10/3, 15/2, 25/1, 40/.5 }
\draw[rotate=\ang] (0,0) rectangle (\length, 0.1) ;
\draw[snake=brace] (0,.4) -- ++(2.7,.5) node [midway,above,sloped] {blades};
\draw (2.5,.3) node {$ \dots$};
\begin{pgftranslate}{\pgfxy(3,0)}
\foreach \ang/\length in {10/3, 15/2, 25/1, 40/.5 }
\draw[rotate=\ang] (0,0) rectangle (\length, 0.1);
\end{pgftranslate}
\draw (5,-.6)  node[left,below] (n1) {hinges}; 
\draw (0,-.1)  node (j1) {}; 
\draw (3,-.1) node (j2) {}; 
\path[->] (n1) edge [out=-180, in=-60] (j1);
\path[->] (n1) edge [out=160, in=-90] (j2);
\end{tikzpicture} 
\end{center}

\caption{A \emph{pocketknife}.}
\label{f.pocketknife}
\end{figure}
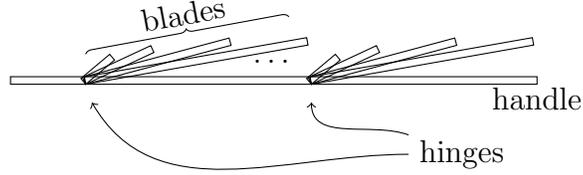

A pocketknife comes with a \emph{handle}, namely a rectangle $ R _{\textup{handle}}$ 
that is longer than any other rectangle in the pocketknife.  We call a collection 
of rectangles $ \mathcal B$ a set of \emph{blades} if these two conditions are met. 
In the first place, 
\begin{equation} \label{e.hinge}
R _{\textup{handle}}
\cap \bigcap _{R\in \mathcal B} R\neq \emptyset \,. 
\end{equation}
In the second place, we have 
\begin{equation*}
\operatorname {angle} (R,R _{\textup{handle}} ) 
\simeq 
\operatorname {angle} (R',R _{\textup{handle}} ) \,, \qquad R,R'\in \mathcal B\,. 
\end{equation*}
Let $ \theta (\mathcal B) $ denote the angle between $ R _{\textup{handle}}$ and the rectangles in the 
blade $ \mathcal B$.  We refer to as a \emph{hinge} a  rectangle of dimensions 
$ \mathsf w/ \theta (\mathcal B)$ by $ \mathsf w$, in the same coordinate system of 
$ R _{\textup{handle}}$ that contains the intersection in \eqref{e.hinge}. 

Now, let $ \mathbb B $ be a collection of blades for the handle $ R _{\textup{handle}}$. 
Our proof of the weak  $ L ^{2}$ estimate for the Lipschitz Kakeya Maximal function 
shows that we can assume 
\begin{equation*}
\sum _{\mathcal B\in \mathbb B } \sharp B \cdot \mathsf w ^2 \cdot 
\theta (\mathcal B) ^{-1} \lesssim \lvert  R _{\textup{handle}}\rvert\,.  
\end{equation*}
This is essentially the estimate \eqref{e.uni1}. 

But, to follow the covering lemma approach to the $ L ^{3/2}$ estimate for the maximal 
function, we need to control 
\begin{equation*}
\sum _{\mathcal B\in \mathbb B } (\sharp B) ^2  \cdot \mathsf w ^2 \cdot 
\theta (\mathcal B) ^{-1} \,. 
\end{equation*}
We can only find control of expressions of this type in terms of 
some slowly varying function of $ \mathsf  w ^{-1} $.

\section*{Bourgain's Geometric Condition} 

Jean Bourgain \cite{bourgain} gives a geometric condition on the Lipschitz vector field 
that is sufficient for the  $ L ^{2}$ boundedness of the maximal function associated with 
$ v$.  We describe the condition, and show how it immediately proves that the 
corresponding Lipschitz Kakeya maximal function admits a weak type bound on $ L ^{1} $. 
In particular our Conjecture~\ref{j.laceyli1} holds for these vector fields.

To motivate Bourgain's condition, let us recall the earlier condition considered 
by Nagel, Stein and Wainger in \cite{MR81a:42027}.  This condition imposes 
a restriction on the maximum and minimum curvatures of the integral curves of the 
vector field through the assumption that 
\begin{equation*}
0< \frac 
{\sup _{x\in \Omega } \operatorname {det}[\nabla v (x) v (x), v (x)]}
{\inf _{x\in \Omega } \operatorname {det}[\nabla v (x) v (x), v (x)]} 
<\infty \,. 
\end{equation*}
Here, $ \Omega $ is a domain in $ \mathbb R ^2 $, and 
one can achieve an upper bound on the norm of the maximal function associated to 
$ v$, appropriately restricted to $ \Omega $, in terms of this ratio. 

Bourgain's condition permits the vector field to have integral curves which are 
flat.  Suppose that $ v$ is defined on all of $ \mathbb R ^2 $.  Define 
\begin{equation}\label{e.zw}
\omega (x; t) := 
\lvert  \operatorname {det}[v (x+t v (x)), v (x)]\rvert\,, 
\qquad 
\lvert  t\rvert\le \tfrac 12 \norm v .\operatorname {Lip}. \,.  
\end{equation}
Assume a uniform estimate of the following type: For absolute constants 
$ 0<c,C <\infty $ and $ 0<\epsilon _0< \tfrac 12 \norm v. \operatorname {Lip}.$, 
\begin{equation}\label{e.bour-un} 
\lvert  \{ \lvert  t\rvert\le \epsilon 
\mid 
\omega (x; t)< \tau \sup _{\lvert  s\rvert\le \epsilon  } \omega (x,s )
\}\rvert 
\le C \tau ^{c} \epsilon \,, 
\end{equation}
this condition holding for all $
 x\in \mathbb R ^2 $, $0<\tau  <1$ and $  \ 0<\epsilon < \epsilon _0$. 

The interest in this condition stems from the fact \cite{bourgain} that 
real-analytic vector fields satisfy it.  Also see Remark~\ref{r.bourgain}. 
Bourgain proved: 

\begin{theorem}\label{t.bourgain} Assume that \eqref{e.bour-un} holds.  Then, the 
maximal operator 
$
\operatorname M _{v, \epsilon _0}
$ defined in \eqref{e.MVE} maps $ L ^2 $ into itself.  
\end{theorem}

This paper claims that the same methods would prove the 
bounds 
\begin{equation*}
\norm \operatorname M _{v, \epsilon _0} . p. 
\lesssim \norm f.p. \,, \qquad 1<p<\infty \,. 
\end{equation*}
And suggests that similar methods would apply to the localized Hilbert 
transform with respect to these vector fields.

Here, we prove 

\begin{proposition}\label{p.bourgain} 
Assume that \eqref{e.bour-un} holds.  Then, the Lipschitz Kakeya Maximal Functions 
\begin{equation*}
\operatorname M _{v, \delta , \mathsf w}
\,, 
\qquad 0<\delta <1\,, \ 0<\mathsf w<\epsilon _0
\end{equation*}
 defined in \eqref{e.MdefW} satisfy the weak $ L ^{1}$ estimate 
\begin{equation*}
\sup _{\lambda >0} 
\lambda \lvert   \{ \operatorname M _{v, \delta , \mathsf w} f > \lambda \}\rvert 
\lesssim 
 \delta ^{-1} ( 1+\log 1/\delta) \norm f.1.  \,. 
\end{equation*}
The implied constants depend upon the constants in \eqref{e.bour-un}.
\end{proposition}

That is, these vector fields easily fall within the scope of our analysis. 
As a corollary to Theorem~\ref{t.conditional}, we see that $ \operatorname H_v$ 
maps $ L ^{2}$ into itself.

\begin{figure}
\begin{tikzpicture}
\draw (0,0) rectangle (6,.75) node[right] {R};
\begin{pgftranslate}{\pgfxy(2,.6)}
\draw[rotate=20] (3, 0) rectangle(0,.75) node[above] {$ R'$};
\end{pgftranslate}
\draw  (0,.5) -- (6,.5); 
\draw[*->,thick] (5,.45) -- (5.9, .85); 
\draw[*->,thick] (2.8,1.3) -- (3.7, 2); 
\draw[-*,thick]  (2.87,1.3) -- (2.87,.45) ;
\draw (2.87,1.3) node (x'') {};
\draw (2.87,.5) node (x') {};
\draw (5.1,.5) node (x) {};
\draw (1,.5) node (l) {}; 
\draw (.5,-.6) node (l1) {$ \ell $};
\path[->] (l1) edge [out=-180, in=-60] (l);
\draw (2.8,-.6) node (x'1) {$ x'_0 $};
\path[->] (x'1) edge [out=160, in=-90] (x');
\draw (4.9,-.6) node (x1) {$ x_0 $};
\path[->] (x1) edge [out=160, in=-90] (x);
\draw (5.4,2) node (x''1) {$x''_0$};
\path[->]  (x''1) edge [out=-120,in=-10] (x'');
\end{tikzpicture}
\caption{Proof of \eqref{e.fortune}.}
\label{f.bourgain}
\end{figure}

\begin{proof}
Let us assume that $ \norm v . \textup{Lip}.=1$. 
Fix $ \delta >0$ and $ 0<\mathsf w< \epsilon _0$.  Let $ \mathcal R$ be 
the class of rectangles with $ \name L R < \kappa $ and satisfying 
$ \lvert  \name V R\rvert\ge \delta \lvert  R\rvert  $.  

Say that $ \mathcal R'\subset \mathcal R$ has \emph{scales separated by $ s>3$} 
iff for $ R,R'\in \mathcal R'$ the condition  $ 4\name L R<  \name L {R'}$ implies that 
$ 2 ^{s} \name L R<  \name L {R'} $.  One sees that $ \mathcal R$ can be decomposed 
into $ \simeq s$ sub-collections with scales separated by $ s$. 

The fortunate observation is this: Assuming \eqref{e.bour-un}, and taking $ s \simeq \log 1/ \delta $, 
any subset $ \mathcal R'\subset \mathcal R$ with scales separated by $ s$ further 
enjoys this property:  If $ R, R'\in \mathcal R'$  with  
$ C\name {EX} R \cap C\name {EX} {R'} =\emptyset $, with $ C$ a fixed constant, 
then 
\begin{equation}\label{e.fortune}
\name L R \simeq \name L {R'} \quad \textup{or} \quad 
 R\cap R'=\emptyset\,. 
\end{equation}

Let us see why this is true, arguing by contradiction. 
Thus we assume that $ \name L {R'}\le 2 ^{-s} \name L {R} $, $ R\cap R'\neq \emptyset $ 
and $ C\name {EX} R \cap C\name {EX} {R'} =\emptyset$. 
Since the rectangles have an essentially fixed width, it follows that 
$ 2 \lvert  \name {EX} {R'}\rvert\ge \lvert  \name {EX} R\rvert  $. 
Fix a line $ \ell $ in the long direction of $ 2 R$ with 
\begin{equation*}
\lvert  \{x\in \ell \mid v (x)\in \name V R \}\rvert \ge \frac \delta 2 \lvert  \ell
\rvert
= \frac \delta 2 \name L R \,. 
\end{equation*}
Let $ x_0$ be in the set above,  $x''_0\in V (R')$ and $x'_0$ is the projection 
of $x''_0$ onto the line $\ell $.  See Figure~\ref{f.bourgain}.  
Observe that we can estimate
\begin{align} \label{e..a}
\abs{v (x_0'')-v (x_0')} & \le 2 \lvert  v (x_0)- v (x_0'')\rvert \name L {R'} 
\end{align}
Therefore, for $C$ sufficiently large, we have 
\begin{align*}
\abs{v (x_0)- v (x_0')}&\ge 
\Abs{ \lvert  v (x_0')-v (x_0 '')\rvert - \lvert  v (x_0'') - v (x_0)\rvert  }
\\
& \ge \lvert  v (x_0'') - v (x_0)\rvert (1 - 2 \name L {R'}) 
\\
& \ge \lvert  \name {EX} {R'}\rvert 
\end{align*}
provided $ C$ is large enough.

 Now, after a moments thought, one sees that 
 \begin{equation*}
\lvert  \operatorname {det}[v (x_0), v (x_0')]\rvert
\simeq \operatorname {angle}( v (x_0), v (x_0'))\,. 
\end{equation*}
 Therefore, for any $ x\in \ell $
\begin{equation*}
\sup _{s \le  \name L R } \omega (x; s) \gtrsim 
\lvert  \name {EX} {R'}\rvert\,.  
 \end{equation*}

 But the vector field satisfies \eqref{e.bour-un}, which we will 
 apply with  
\begin{equation*}
\tau \simeq \frac {\name {EX} {R} } {\name {EX} {R'}} 
\simeq 
\frac {\name L {R'} }  {\name L R} \,. 
\end{equation*}
It follows that 
\begin{align*}
\frac \delta 2 \name L R 
& \le
\lvert  \{x \in \ell \mid \omega  (x;s)\le 
c \tau \lvert  \name {EX} {R'}\rvert
\}\rvert 
\\
& \le 
\lvert  \{x \in \ell \mid \omega  (x;s)\le 
\tau \sup _{\lvert  s\rvert\le \epsilon  } \omega (x,s )
\}\rvert       
\\
& \lesssim \tau ^{c} \name L R\,. 
\end{align*}
Therefore, we see 
\begin{equation*}
(\delta /2) ^{1/c} \lesssim 
\frac {\name L {R'} }  {\name L R} \,, 
\end{equation*}
which is a contradiction to $ \mathcal R'$ have scales separated by $ s$, and 
$ s \simeq 1+ \log 1/\delta $. 

\medskip 

Let us see how to prove the Proposition now that we have proved \eqref{e.fortune}. 
Take $ s \simeq \log 1/\delta $, and a finite sub-collection $ \mathcal R'\subset \mathcal R$
 of rectangles with scales separated by $ s$.  We may 
take a further subset $ \mathcal R''\subset \mathcal R'$ such that 
\begin{gather}\label{e.1cover}
\NOrm \sum _{R''\in \mathcal R''} \mathbf 1_{R''} . \infty . \lesssim \delta ^{-1} \,, 
\\ \label{e.2cover}
\ABs{ \bigcup _{R\in \mathcal R'- \mathcal R''} R''} 
\lesssim 
\sum _{R\in \mathcal R''}  \lvert  R''\rvert\,.  
\end{gather}
These are precisely the covering estimates needed to prove the weak $ L ^{1}$ estimate 
claimed in the proposition. 

But, in choosing $ \mathcal R''$ to satisfy \eqref{e.1cover}, it is clear that 
we need only be concerned about rectangles with a fixed length, and the separation in 
scales are \eqref{e.fortune} will control rectangles of distinct lengths.  

The procedure that we apply to select $ \mathcal R''$ is inductive.  
Set 
\begin{align*}
\mathcal R''&\leftarrow \emptyset\,,
\\
\mathcal S & \leftarrow \emptyset\,, 
\\
\mathsf{STOCK} & \leftarrow \mathcal R'\,.
\end{align*}
\textsf{WHILE} $ \mathsf {STOCK}\neq\emptyset$, select $ R\in \mathsf {STOCK}$ 
with maximal length, 
and update $ \mathcal R''\leftarrow \mathcal R'' \cup \{R\}$, as well as 
$ \mathsf{STOCK}  \leftarrow \mathsf {STOCK}- \{R\} $.  
In addition, for any $ R'\in \mathsf {STOCK}$ with $ R'\subset 4C R$,
where $ C\ge1$ is the constant that insures that  \eqref{e.fortune} holds, 
remove these rectangles from $ \mathsf {STOCK}$ and add them to $ \mathsf S$.  

Once the \textsf{WHILE} loop stops, we will have $ \mathsf {STOCK}=\emptyset$ 
and we have our decomposition of $ \mathcal R'$.  By construction, it is 
clear that \eqref{e.2cover} holds.  We need only check that \eqref{e.1cover} 
holds.  Now, consider $ R,R'\in \mathcal R'$, with  two rectangles have their 
scales separated, thus $ 2 ^{s}\name L {R'} < \name L R $.  
If it is the case that $ R\cap R'\neq \emptyset$ and $ C \name {EX} R \cap 
C \name {EX} {R'} \neq \emptyset$, then $ R$ would been selected to be in $ \mathcal R'$ 
first, whence $ R'$ would have been placed in $ \mathcal S$.  

Therefore, $  C \name {EX} R \cap 
C \name {EX} {R'} =\emptyset$, but then \eqref{e.fortune} implies that $ R\cap R'=\emptyset$.
Thus, the only contribution to the $ L ^{\infty }$ norm in \eqref{e.1cover} can come from 
rectangles of about the same length.  But Lemma~\ref{l.fixscale} then implies that 
such rectangles can overlap only about $ \delta ^{-1} $ times.  Our proof is complete. 
(As the interest in \eqref{e.bour-un} is in small values of $ c$, it will be more 
efficient to use Lemma~\ref{l.fixscale} to handle the case of the rectangles having 
approximately the same length.)
\end{proof}

\begin{remark}\label{r.bourgain}  To conclude that the Hilbert transform on 
vector fields is bounded, one could weaken Bourgain's condition \eqref{e.bour-un}
to 
\begin{equation*}
\lvert  \{ \lvert  t\rvert\le \epsilon 
\mid 
\omega (x; t)< \tau \sup _{\lvert  s\rvert\le \epsilon  } \omega (x,s )
\}\rvert 
\le C \operatorname {exp}(-(\log 1/\tau) ^{c} )\epsilon \,. 
\end{equation*}
This inequality is to hold universally in $x\in \mathbb R ^2 $, 
$0<\tau <1$, and $0<\epsilon < \norm v . \textup{Lip}.$. 
This is of interest for $ 0<c<1$. The proof above can be modified to show that 
the maximal functions $ \operatorname M _{v, \delta , \mathsf w}$ 
satisfy the weak $ L ^{1}$ inequality, with constant at most $ \lesssim \delta ^{-1-1/c}$. 
\end{remark}

\section*{Vector Fields that are a Function of One Variable} 

We specialize to the vector fields that are a function of just one real variable. 
Assume that the vector field $ v$ is of the form 
\begin{equation} \label{e.v2}
v (x_1,x_2)= (v_1 (x_2), v_2 (x_2))\,, 
\end{equation}
and for the moment we do not impose the condition that the vector 
field take values in the unit circle.  The point is simply this: 
If we are interested in transforms where the kernel is not localized, 
the restriction on the vector field is immaterial.  Namely, 
for \emph{any} vector field $ v$
\begin{align*}
\operatorname H _{v,\infty } f (x)
& = \textup{p.v.}\int _{-\infty } ^{\infty } f (x-y v (x))\; \frac {dy} y 
\\
& = \textup{p.v.}\int _{-\infty } ^{\infty } f (x-y \widetilde v (x))\; \frac {dy} y 
\,, \qquad \widetilde v (x)=\frac {v (x)} {\lvert  v (x)\rvert }\,. 
\end{align*}

We return to a theme implicit in the proof of Proposition~\ref{p.Stein=>Carleson}.  
This  proof only relies upon vector fields that are only a 
function of one variable. Thus, it is a significant subcase of the Stein 
Conjecture to verify it for Lipschitz vector fields of just one variable. 
Indeed, the situation is this. 

\begin{proposition}\label{p.justOne} Suppose that a choice of vector field 
$ v (x_1,x_2)=(1,v_1 (x_1)) $ is just a function of, say, the first coordinate.  Then, 
$ \operatorname H _{v,\infty }$ maps $ L ^2 (\mathbb R ^2 )$ into itself.  
\end{proposition}

\begin{proof}
The symbol of $ H _{v,\infty }$ is 
\begin{equation*}
\operatorname {sgn} (\xi _1+ \xi _2v_1 (x_1))\,. 
\end{equation*}
For each fixed $ \xi _2$, this is a bounded symbol.  And in the special case of 
the $ L ^2 $ estimate, this is enough to conclude the boundedness of the operator.  
\end{proof}

It is of interest to extend this Theorem in any $ L ^{p}$, for $ p\neq 2$, for 
some reasonable choice of vector fields.   

The corresponding questions for the maximal function are also of interest, and 
here the subject is much more developed.  
The paper  \cite{carbery} studies the 
maximal function $ \operatorname M _{v,\infty }$.  They proved the 
boundedness of this maximal function on $ L ^{p}$, $ p>1$, assuming that the 
vector field was of the form $ v (x)= (1,v_2 (x))$, that $ \operatorname D v_2$ 
was  positive, and increasing, and satisfied a third  more technical condition.
More recently, \cite{kim} has showed that the third condition is not needed. 
Namely the following is true. 

\begin{theorem}\label{t.kim} Assume that $ v (x)= (1,v_2 (x))$, 
and moreover that $\operatorname D v_2\ge0 $ and is 
monotonically increasing. Then, $ \operatorname M _{v, \infty }$ is 
bounded on $ L ^{p}$, for $ 1<p<\infty $.  
\end{theorem}

These vector fields present far fewer technical difficulties 
than a general Lipschitz vector field, and there are a richer set of 
proof techniques that one can bring to bear on them, as indicated in part  
in the proof of Proposition~\ref{p.justOne}. The papers \cites{carbery,kim} 
cleverly exploit the Plancherel identity (in the independent variable), 
and other orthogonality considerations to prove their results.

These considerations are not completely consistent with the dominant theme 
of this monograph, in which the transforms are localized. Nevertheless, it would 
be interesting to explore methods, possibly modifications of this memoir, that 
could provide an extension of Proposition~\ref{p.justOne}. 

In this direction, let us state a possible direction of study.  
The definition of the the sets $ \name V R $ for vector fields of magnitude 
$ 1$ is given as $ \name V R = R \cap v ^{-1} (\name {EX} R)$.  For vector 
fields of arbitrary magnitude, we define these sets to be 
\begin{equation*}
\name V R = \{x\in R \mid  \tfrac {v (x)} {\lvert  v (x)\rvert } \in \name {EX} R \}\,. 
\end{equation*}
Define a maximal function---an extension of our Lipschitz Kakeya Maximal Function---by 
\begin{equation}\label{e.tildeM}
\widetilde {\operatorname M} _{v,\delta }
f (x)=\sup _{\substack{ x\in R\\ \lvert  \name V R \rvert  \ge \delta \lvert  R\rvert  }} 
\lvert  R\rvert ^{-1} \int _{R} f (y)\; dy \,.  
\end{equation}
In this definition, we require the rectangles to have density $ \delta $, but do 
not restrict their eccentricities, or lengths.  

\begin{conjecture}\label{j.v2} Assume that the vector field is of the form $ v (x)
= (1, v_2 (x_2))$, and the derivative $ \operatorname D v \ge 0$ and is monotone. 
Then for all $ 0<\delta <1 $, we have the estimate
\begin{equation*}
\norm \widetilde {\operatorname M} _{v,\delta } .p. \lesssim \delta ^{-1} \,, 
\qquad  1<p<\infty \,.
\end{equation*}
\end{conjecture}

One can construct examples which show that the $ L ^{1}$ to weak $ L ^{1}$ norm 
of the maximal function is not bounded in terms of $ \delta $.  
Indeed, recalling the `pocketknife' examples of Figure~\ref{f.pocketknife}, 
we comment that one can construct examples of vector fields with these properties, 
which we describe with the terminology associated with the pocketknife examples. 
\begin{itemize}
\item The width of all rectangles are fixed. And all rectangles have density $ \delta $. 
 \item The `handle' of the pocketknife has positive angle $ \theta $ with the $ x_1$ axis. 
 \item There is `hinge' whose blades have angles which are positive, and greater than 
 $ \theta $.  The number of blades can be unbounded, as the width of the rectangles 
 decreases to zero. 
\end{itemize}
The assumption that the vector field is only a function of $ x_2$ then greatly restricts, 
but does not completely forbid, the existence of additional hinges.  So the combinatorics 
of these vector fields, as expressed in the Lipschitz Kakeya Maximal Function, 
are not so simple.

\chapter[$ L^2$ Estimate for $ \operatorname H_v$]{The $ L^2$ 
Estimate for Hilbert Transform on Lipschitz Vector Fields}

We prove one of our main conditional results about the Hilbert transform on Lipschitz 
vector fields, the inequality \eqref{e.cond1} which is  the estimate at $ L ^2 $, for functions 
with frequency support in an annulus, assuming an appropriate estimate for the 
Lipschitz Kakeya Maximal Function.

We begin the proof by setting notation appropriate for phase plane analysis 
for functions $ f$ on the plane supported on an annulus.  With this notation, we 
can define appropriate discrete analogs of the Hilbert transform on vector fields. 
The Lemmas~\ref{l.model} and \ref{l.modelsumed} are the combinatorial analogs of 
our Theorem~\ref{t.laceyli1}.  We then take up the proofs. 
The main step in the proof is Lemma~\ref{l.dense+size} which combines the 
(standard) orthogonality considerations with the conjectures about the 
Lipschitz Kakeya Maximal Functions.

 \section*{Definitions and Principle Lemmas}

Throughout this chapter, $\kappa$ will denote a fixed small positive constant, 
whose exact value need not concern us.  $\kappa$ of the order of $10^{-3}$ would suffice. 
The following definitions are as in the authors' previous paper \cite{laceyli1}.

\begin{definition}\label{d.grid}
  A {\em grid} is a collection of intervals $\mathcal G$ so that for all $I,J\in\mathcal G$, we have 
$I\cap J\in\{\emptyset, I, J\}$.  The dyadic intervals are a grid. 
A grid $\mathcal G$ is {\em central} iff for all $I,J\in\mathcal G$, with $I\subset_{\not=}J$ we have $500 
\kappa ^{-20} I\subset J$. 
\end{definition}

The reader can find the details on how to construct such a central 
grid structure in \cite{gl1}.

Let $\rho$ be rotation on $\mathbb T$ by an angle of $\pi/2$.  Coordinate axes
for $\mathbb R^2$ are a pair of unit orthogonal vectors 
$(\operatorname e,\operatorname e_\perp)$  with $\rho \operatorname  e= \operatorname e_\perp$.

\begin{definition}\label{d.rectangle}

We say that $\omega\subset \mathbb R^2$ is a {\em rectangle} if it is a product of
intervals with respect to a choice of axes $(\operatorname e,\operatorname e_\perp)$ of $\mathbb R^2$.
We will say that $\omega$ is an {\em annular rectangle}
 if   $\omega=(-2^{l-1},2^{l-1})\times(a,2a)$ for an integer $l$ with
$2^l<\kappa a$, with respect to the axes $(\operatorname e,\operatorname e_\perp)$.
The dimensions of $\omega$ are said to be $2^{l}\times a$.  
 Notice that the face $(-2^{l-1},2^{l-1})\times a $ is tangent to the circle
$\abs \xi=a$ at the midpoint to the face, $(0,a )$.
We say that the {\em scale of}
$\omega$ is $\scl \omega:=2^l$ and that the {\em annular parameter of } $\omega$ is 
$\prm \omega :=a$.
  In referring to the coordinate axes of an annular
rectangle, we shall always mean  $(\operatorname e,\operatorname e_\perp)$ as above.  
\end{definition}

 Annular rectangles will decompose our functions in the frequency variables.  But 
our
methods must be sensitive to spatial considerations;  it is this and the
uncertainty principle that motivate the next definition.

\begin{definition}\label{d.dual}
  Two rectangles $R$ and $\mathsf R$ are said to be {\em dual }
if they are rectangles with respect to the same basis  $(\operatorname e,\operatorname e_\perp)$, thus
$R=r_1\times r_2$ and ${\mathsf R}= {\mathsf r}_1\times {\mathsf  r}_2$ for
intervals $r_i, {\mathsf r}_i$, $ i=1,2$.  Moreover, 
$1\le{}\abs{r_i}\cdot\abs{{\mathsf
r}_i}\le4$ for $ i=1,2$.   The product of two dual rectangles we shall refer to 
as a {\em phase rectangle.}  The first coordinate of a phase 
rectangle 
 we think of as a frequency component and
the second as a spatial component.     
\end{definition}


\begin{figure}
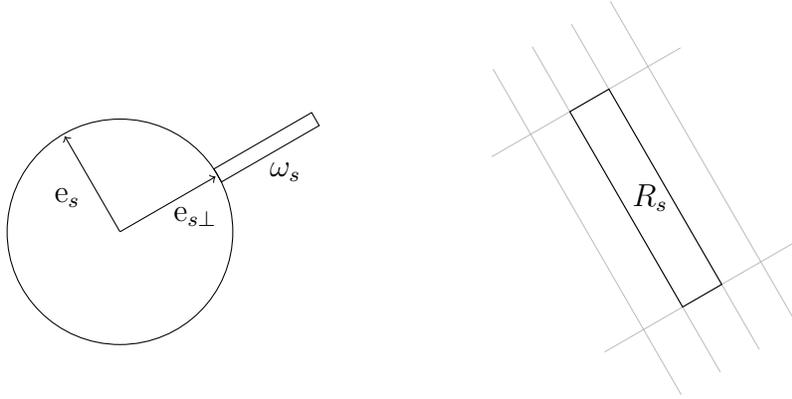
 
\begin{center}
  \begin{pgfpicture}{0cm}{0cm}{5cm}{5cm}
%
%
	\begin{pgftranslate}{\pgfpoint{-1.5cm}{2cm}}
	\pgfcircle[stroke]{\pgfpoint{0cm}{0cm}}{1.5cm} 
		\begin{pgfrotateby}{\pgfdegree{30}}
		\pgfrect[stroke]{\pgfpoint{1.5cm}{-.1cm}}{\pgfpoint{1.5cm}{0.2cm}}
		
		\pgfsetendarrow{\pgfarrowto}
		\pgfmoveto{\pgfpoint{0cm}{0cm}}\pgflineto{\pgfpoint{1.45cm}{0cm}}\pgfstroke
		\pgfmoveto{\pgfpoint{0cm}{0cm}}\pgflineto{\pgfpoint{0cm}{1.45cm}}\pgfstroke
		
		\end{pgfrotateby}
	\pgfputat{\pgfpoint{2.2cm}{.8cm}}{\pgfbox[center,center]{$\omega_s$}}
	\pgfputat{\pgfpoint{1cm}{.2cm}}{\pgfbox[center,center]{$\operatorname e_{s\perp}$}}
	\pgfputat{\pgfpoint{-.7cm}{.45cm}}{\pgfbox[center,center]{$\operatorname e_{s}$}}
	
	\end{pgftranslate}
	\begin{pgftranslate}{\pgfpoint{6.5cm}{1.3cm}}
		\begin{pgfrotateby}{\pgfdegree{30}}
			{\color{lightgray}
			\pgfgrid[step={\pgfpoint{0.6cm}{3cm}}]{\pgfpoint{-1.8cm}{-1cm}}{\pgfpoint{1.1cm}{4cm}}
			}		
		\pgfrect[stroke]{\pgfpoint{-.6cm}{0cm}}{\pgfpoint{0.6cm}{3cm}}
		\end{pgfrotateby}
	\pgfputat{\pgfpoint{-.95cm}{1.15cm}}{\pgfbox[center,center]{$R_s$}}
	\end{pgftranslate}
	\end{pgfpicture}
	\end{center} 
	\caption{ The two rectangles $\omega_s $ and $R_s $ whose product is a tile. The gray rectangles 
	are other possible locations for the rectangle $R_s $.} 
	 \label{f.tile}
	\end{figure}
   

We consider collections of phase rectangles $\mathcal {AT}$ which satisfy these
conditions.   For $s,s'\in\mathcal {AT}$ we write $s=\omega_s\times R_s$, and  require that 
\begin{gather}
\label{e.R1} 
\text{$\omega_s$ is an annular rectangle,}
\\\label{e.R2} 
\text{$R_s$ and  $\omega_s$ are  dual,}
\\
\text{The rectangles $R_s $ are from the product of  central grids.} 
\label{e.R-central}
\\
\label{e.R3.5}
\{ 1000\kappa ^{-100} R \mid  \omega_s \times R \in\mathcal {AT}\}\quad\text{covers $\mathbb R^2$, for all $\omega_s$.}
\\\label{e.R3} 
\prm{\omega_s}=2^j\quad \text{for some integer $j$,}
\\  \label{e.R4} 
\sharp\{\omega_s\mid \scl s=\Scl,\ \prm s=\Prm\}\ge{}c\frac {\Prm} {\Scl}\,,
\\ \label{e.R5}   
\scl s\le\kappa\prm s.
 \end{gather}
We assume that
there are auxiliary sets ${\boldsymbol \omega}_{s},{\boldsymbol \omega}_{s1},{\boldsymbol \omega}_{s2}\subset\mathbb T$ associated to $s$---or
more specifically $\omega_s$---which satisfy these properties.
\begin{gather}
\label{e.bw1} 
\text{$\boldsymbol \Omega:=
\{{\boldsymbol \omega}_{s},{\boldsymbol \omega}_{s1},{\boldsymbol \omega}_{s2}\mid s\in\mathcal {AT}\}$
is a grid in $\mathbb T$,}
\\\label{e.bw2} 
{\boldsymbol \omega}_{s1}\cap{\boldsymbol \omega}_{s2}=\emptyset,\qquad  \abs{{\boldsymbol \omega}_{s}}\ge32(\abs{{\boldsymbol \omega}_{s1}}+\abs{{\boldsymbol \omega}_{s2}}+\operatorname{dist}({\boldsymbol \omega}_{s1},{\boldsymbol \omega}_{s2}))
\\\label{e.bw2/3}
\text{${\boldsymbol \omega}_{s1}$ lies clockwise from ${\boldsymbol \omega}_{s2}$ on $\mathbb T$,}
\\\label{e.bw3} 
\abs{{\boldsymbol \omega}_{s}}\le{}K\frac{\scl{\omega_s}}{\prm {\omega_s}},
\\\label{e.bw4} 
\{\tfrac\xi{\abs{\xi}}\mid \xi\in\omega_s\}\subset\rho{\boldsymbol \omega}_{s1}. 
\end{gather}
In the top line, the intervals ${\boldsymbol \omega}_{s1}$ and ${\boldsymbol \omega}_{s2}$ are small subintervals of the unit circle, 
and we can define their dilate by a factor of 2 in an obvious way.
 Recall that $\rho$ is the rotation that takes $e$ into $\operatorname e_\perp$. Thus,
 $\operatorname e_{\omega_s}\in{\boldsymbol \omega}_{s1}$.  See the figures
 Figure~\ref{f.tile} and Figure~\ref{f.omega} for an illustration of these definitions.

	
	\begin{figure}
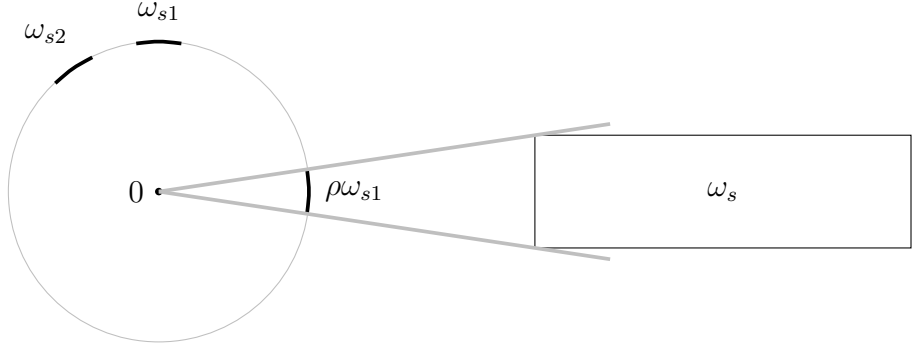

	
	 \begin{pgfpicture}{0cm}{0cm}{8cm}{4.5cm}
	 \begin{pgftranslate}{\pgfxy(-.15,2)}
	
	{\color{lightgray}  
		\pgfcircle[stroke]{\pgfxy(0,0)}{2cm}
	}
	
	\pgfcircle[fill]{\pgfxy(0,0)}{0.05cm}
	\pgfputat{\pgfxy(-.3,0)}{\pgfbox[center,center]{$0$}}

	\pgfrect[stroke]{\pgfxy(5,-.75)}{\pgfxy(5,1.5)}

		\pgfputat{\pgfxy(2.6,0)}{\pgfbox[center,center]{$\rho \omega_{s1} $}}
	\pgfputat{\pgfxy(0,2.4)}{\pgfbox[center,center]{$\omega_{s1} $}}
	\pgfputat{\pgfxy(-1.5,2.1)}{\pgfbox[center,center]{$\omega_{s2} $}}
	
	\pgfputat{\pgfxy(7.5,0)}{\pgfbox[center,center]{$\omega_s $}}
	
	\pgfsetlinewidth{1.4pt}
	\pgfmoveto{\pgfxy(1.97,.3)}
	\pgfcurveto{\pgfxy(2.005,0)}{\pgfxy(2.005,0)}{\pgfxy(1.97,-.3)}\pgfstroke
	\pgfmoveto{\pgfxy(-.3,1.97)}\pgfcurveto{\pgfxy(0,2.005)}{\pgfxy(0,2.005)}{\pgfxy(.3,1.97)} \pgfstroke
	\begin{pgfrotateby}{\pgfdegree{35}}
	\pgfmoveto{\pgfxy(-.3,1.97)}\pgfcurveto{\pgfxy(0,2.005)}{\pgfxy(0,2.005)}{\pgfxy(.3,1.97)} \pgfstroke
	\end{pgfrotateby}

	{\color{lightgray}  
			\pgfmoveto{\pgfxy(0,0)} 
			\pgflineto{\pgfxy(6,.9)} 	\pgfstroke
			\pgfmoveto{\pgfxy(0,0)} 
			\pgflineto{\pgfxy(6,-.9)}	\pgfstroke
	}

	\end{pgftranslate}
	\end{pgfpicture}
	
	\caption{An annular rectangular $\omega_s $, and three associated subintervals of $\rho \omega_{s1} $, $\omega_{s1} $, 
	and $\omega_{s2} $. }
	\label{f.omega} 
	\end{figure}
	

Note that 
$\abs{{\boldsymbol \omega}_{s}}\ge\abs{{\boldsymbol \omega}_{s1}}\ge\scl{\omega_s}/\prm{\omega_s}$.  Thus, $\operatorname e_{\omega_s}$ is in
${\boldsymbol \omega}_{s1}$, and  ${\boldsymbol \omega}_{s}$ serves as 
`the angle of uncertainty  associated to $R_s$.'
Let us be more precise
about the geometric information encoded into the angle of uncertainty.  Let $R_s=r_s\times
r_{s\perp}$ be as above. Choose another set of coordinate axes $(\operatorname e',\operatorname e'_\perp)$ with
$e'\in{\boldsymbol \omega}_{s}$ and let $R'$ be the product of the intervals $r_s$ and $r_{s\perp}$ in the new coordinate
axes.  Then $K^{-1}_0R'\subset R_s\subset K_0R'$ for an absolute constant $K_0>1$.

We say that {\em annular tiles} are  collections $\mathcal {AT}$ satisfying the
conditions \eqref{e.R1}---\eqref{e.bw4} above.   We
extend the definition of $\operatorname e_\perp$, $\operatorname e_{\omega\perp}$, $\prm{\omega}$ and $\scl{\omega}$ to
annular tiles in the obvious way, using the notation $\operatorname e_s$, $\operatorname e_{s\perp}$, $\prm s $ and
$\scl s $.

\medskip

A phase rectangle will have two distinct functions
associated to it.   In order to define these functions, set 
\begin{gather*}
\operatorname{T}_y f(x):=f(x-y),\quad y\in\mathbb R^2\quad \text{(Translation operator)}
\\
\operatorname{Mod}_\xi f(x):=\operatorname e^{i\xi\cdot x}f(x),
\quad \xi\in\mathbb R^2\quad \text{(Modulation operator)}
\\
\operatorname{D}_{R_1\times R_2}^p f(x_1,x_2):= \frac1{(\abs{R_1}\abs{R_2})^{1/p}} 
f\Bigl(\frac{x_1}{\abs{R_1}},\frac{x_2}{\abs{R_2}}\Bigr)\quad \text{(Dilation operator)}.
\end{gather*}
In the last display, $0<p\le\infty$, and $R_1\times R_2$ is a rectangle,
and the coordinates $(x_1,x_2)$ are those of the rectangle.  Note that the definition depends only on the side lengths of the rectangle, and not the location.  And that it preserves $L^p$ norm.

For a function $\varphi$ and tile $s\in\mathcal {AT}$ set 
\begin{equation} \label{e.zvfs}
 \varphi_s:=\operatorname{Mod}_{c(\omega_s)}
\operatorname{ T}_{c(R_s)}\operatorname{D}^2_{R_s} \varphi
\end{equation}
We shall  consider $\varphi$ to be a Schwartz function for which
$\widehat\varphi\ge0$ is supported  in a small ball, of radius $\kappa$,
about the origin in $\mathbb R^2$, and is identically $1$ on another smaller ball around the origin.  
(Recall that $\kappa$ is a fixed small constant.)

We introduce the tool to decompose the singular integral kernels.
In so doing, we consider a class of functions $\psi_t $, $t>0 $, so that 
\begin{gather}
\label{e.zc-Fourier}  \text{Each $\psi_t $ is supported in frequency in $[- \theta-\kappa, -\theta+\kappa ] $. } 
\\
\label{e.zc-Space}  \abs {\psi _t (x) } \lesssim C_N (1+\abs x )^{-N}\,, \qquad N>1\,.  
\end{gather}
 In the top line, $\theta $ is 
 a fixed positive constant  so that the second half of \eqref{e.dfs2} is true.

Define 
\begin{align}\label{e.dfs2} 
	\begin{split}
\phi_s(x):=&\inr \varphi_s(x-yv(x))\psi_s( y)\;dy
\\=&  
\ind {{\boldsymbol \omega}_{s2}}(v(x))\inr \varphi_s(x-yv(x))\psi_s( sy)\;dy.
	\end{split} 
\\  \label{e.zc_s} 
\psi_s(y ):=&\scl s \psi_{\scl s}(\scl s y).
\end{align}
An essential feature of this definition is that the support of the integral  is contained in the set 
${\{v(x)\in{\boldsymbol \omega}_{s2}\}}$, a fact which can be routinely verified.
That is, we can insert the indicator $\ind {{\boldsymbol \omega}_{s2}}(v(x))$ without loss of generality.
The set ${\boldsymbol \omega}_{s2}$ serves to localize the vector field, while 
${\boldsymbol \omega}_{s1}$ serves to identify the location of $\varphi_s $ in the frequency
coordinate.

The model operator we consider acts on a   Schwartz functions $f $,
and it  is defined by 
\begin{equation} \label{e.Cj}
\mathcal C_\Prm f:=  \sum_{\substack{s\in\mathcal {AT}(\Prm)\\ \scl s\ge\norm
v.\textup{Lip}.} }\ipf \phi_s. 
 \end{equation}
 In this display, $\mathcal {AT}(\Prm):=\{s\in\mathcal {AT}\mid \prm s=\Prm\}$, and we have deliberately formulated the 
 operator in a dilation invariant manner.


 \begin{lemma}\label{l.model}
   Assume that the vector field is Lipschitz, 
   and satisfies Conjecture~\ref{j.laceyli2}.   
   Then, for all $ \Prm\ge \norm v . \textup{Lip}. ^{-1} $, 
   the operator $\mathcal C_\Prm$ extends to a bounded map  from $L^2$ into itself, 
 with norm bounded by an absolute constant. 
\end{lemma}  

We remind the reader that for $2<p<\infty$ the only condition needed for the boundedness of $\mathcal C_\Prm$
is the measurability of the vector field, a principal result of Lacey and Li \cite{laceyli1}.
It is of course of great importance to add up the $\mathcal C_\Prm$ over $\Prm$.  
The method we use  for doing this are purely $L^2$  in nature, and lead to the estimate for $\mathcal C:=\sum_{j=1}^\infty \mathcal C_{2^j}$.

\begin{lemma}\label{l.modelsumed}
Assume that the vector field is of norm at most one in  $C^\alpha$ for some $\alpha>1$, 
and satisfies Conjecture~\ref{j.laceyli2}.   
 Then $\mathcal C$ maps $L^2$ into itself.
 In addition we have the estimate below, holding for all values of ${\mathsf {scl}}$. 
\begin{equation} \label{e.scalefixed}
\NOrm \sum_{\Prm=-\infty}^\infty{} \sum_{\substack{s\in\mathcal {AT}(\Prm)\\ \scl s={\mathsf {scl}}}} 
\ipf \phi_s .
2.\lesssim{}(1+ \log( 1+{\mathsf {scl}}^{-1} \lVert v\rVert_{C^\alpha})).
\end{equation}
Moreover, these operators are  unconditionally convergent in $s\in\mathcal {AT}$. 
\end{lemma}

These are the principal steps towards the proof of Theorem~\ref{t.laceyli1}. 
In the course of the proof, we shall not invoke the additional notation needed to account 
for the unconditional convergence, as it is entirely notational. They can be added in by
 the reader. 

 Observe that \eqref{e.scalefixed} is only of interest when 
 ${\mathsf {scl}}<\norm v.C^\ensuremath{\alpha}.$. This inequality depends critically 
on the fact that the kernel ${\mathsf {scl}} \ensuremath{\psi}({\mathsf {scl}} y)$ 
has mean zero.  Without this assumption, this inequality is certainly 
false.

The proof of   Theorem~\ref{t.laceyli1} from these two lemmas is an argument in which one 
averages over translations, 
dilations and rotations of grids.  
The specifics of the approach are very close to the corresponding argument in \cite{laceyli1}. 
The details are omitted.

The operators $\mathcal C_\Prm$ and $\mathcal C$ are constructed from a a kernel 
which is a smooth analog of 
the truncated kernel ${\rm p.v.} \frac{1}{t}
1_{\{|t|\leq 1\}}$.  Nevertheless, our main theorem follows,\footnote{In the typical 
circumstance, one uses a maximal 
function to pass back and forth between truncated and smooth kernels.  
This route is  forbidden to us; there is no 
appropriate maximal function to appeal to.}
due to the observation that we can choose  a sequence of Schwartz kernels  
$\psi_{(1+\kappa)^n} $, for $n\in\mathbb Z $, which 
satisfy \eqref{e.zc-Fourier} and \eqref{e.zc-Space}, and so that  for 
\begin{equation*}
K(t):=\sum_{n\in \ensuremath{\mathbb Z} } 
a_n (1+\kappa)^n \ensuremath{\psi}_{(1+\kappa)^n}((1+\kappa)^n t).
\end{equation*}
we have ${\rm p.v.} \frac{1}{t}
1_{\{|t|\leq 1\}}=K(t)-\overline{K(t)} $.  
Here, for $n\ge0$ we have $\abs{a_n} \lesssim1$.  And for $n<0$, we have 
$\abs{a_n} \lesssim(1+\kappa)^n$.  
The principal sum is thus over $n\ge\max(0,\norm v.C^\alpha.)$, and this corresponds to the operator $\mathcal C$.   For those
$n<\max(0,\norm v.C^\alpha.)$, we use the  estimate \eqref{e.scalefixed}, and the rapid decay of the coefficient $a_n$.

\section*{Truncation and an Alternate Model Sum }
 
There are significant obstacles to proving the boundedness of 
the model sum  $ \mathcal C _{\Prm}$ 
on an $L^p$ space, for $1<p<2$.  In this section, we rely upon some naive $L^2$ estimates 
to define a new model sum which is bounded on  $L^p$, for some $1<p<2$. 

Our next Lemma is indicative of the  estimates we need.  For choices of  ${\mathsf {scl}}<\kappa\Prm$, set 
\begin{equation*}
\mathcal {AT}(\Prm,{\mathsf {scl}}):=\{s\in\mathcal {AT}(\Prm)\mid \scl s={\mathsf {scl}}\}.
\end{equation*}
 
\begin{lemma}\label{l.fixedscale}
  For   measurable vector fields $v$ and all choices of $\Prm$ and ${\mathsf {scl}}$. 
\begin{equation*}
\NOrm \sum_{s\in\mathcal {AT}(\Prm,{\mathsf {scl}})}\langle f,\varphi_s\rangle \phi_s .2.\lesssim\norm f.2.
\end{equation*}
\end{lemma} 

\begin{proof} The scale and annulus are fixed in this sum, making the Bessel inequality 
\begin{equation*}
\sum_{s\in\mathcal {AT}(\Prm,{\mathsf {scl}})}\abs{ \ip f,\varphi_s, }^2\lesssim{} \lVert f\rVert_2^2
\end{equation*}
 evident. For any two tiles $s$ and $s'$ that contribute to this sum, if 
 ${\boldsymbol \omega}_{s}\not={\boldsymbol \omega}_{s'}$, then 
$\phi_s$ and $\phi_{s'}$ are disjointly supported. And if  
${\boldsymbol \omega}_{s}={\boldsymbol \omega}_{s'}$, then $R_s$ and $R_s'$ 
are disjoint, but share the same dimensions and orientation in the plane.  The rapid decay of the 
functions $\phi_s$ then gives us the estimate 
\begin{align*}
\NOrm \sum_{s\in\mathcal {AT}(\Prm,{\mathsf {scl}})}
\langle f,\varphi_s\rangle \phi_s .2.\lesssim 
&
\Biggl[ \sum_{s\in\mathcal {AT}(\Prm,{\mathsf {scl}})}\abs{ \ip  f, \varphi_s, }^2\Biggr]^{1/2}
\\&\lesssim \lVert f\rVert_2
 \end{align*}
 \end{proof}

 Consider the variant of the operator \eqref{e.Cj} given by 
 \begin{equation}\label{e.Phi}
 \Phi f=\sum_{\substack{s\in\mathcal {AT}(\Prm) \\ \scl s\ge\kappa^{-1}\vLip } }\ipf \fss {} .
 \end{equation}
 As $\Prm$ is fixed, we shall begin to suppress it in our notations for operators. 
  The difference between $\Phi$ and $\mathcal C_{\Prm}$ is the absence of the 
  initial $\lesssim\log (1+\vLip)$ scales 
  in the former.  The $L^2$ bound for these missing scales is clearly provided by
  Lemma~\ref{l.fixedscale}, 
  and so it remains for us to establish 
  \begin{equation} \label{e.Phi2} 
  \norm \Phi .2.\lesssim{}1,
  \end{equation}
  the implied constant being independent of $\Prm$, and the Lipschitz norm of $v$.

  It is an important fact, the main result of  Lacey and Li \cite{laceyli1}, that 
  \begin{equation}\label{e.Phip} 
  \norm \Phi.p.\lesssim{}1,\qquad 2<p<\infty.
  \end{equation}
 This holds without the Lipschitz assumption.

We are now at a point where we can be more directly engaged with the 
construction of our alternate model sum. 
 We only consider tiles with $\kappa^{-1}\vLip\le\scl s\le\kappa\Prm $. 
 A parameter is introduced which is used 
 to make a spatial truncation of the functions $\varphi_s $; it is 
 \begin{equation}\label{e.gamma} 
 \gamma_s^2:= 100 ^{-2}\frac{\scl s}{\vLip}
 \end{equation}
  Write $\varphi_s=\alpha_s+\beta_s$ where 
$\alpha_s=(T_{c(R_s)}D^{\infty}_{\gamma_s R_s}\zeta)\varphi_s$, and $\zeta$ is a
smooth Schwartz function supported on $\abs x<1/2$, and equal to $1$ on $\abs x<1/4 $. 
 
 Write  for choices of tiles $s$, 
 \begin{equation} \label{e.psi_s} 
\psi _{s}( y)=\psi_{s-}(y)+\psi_{s+}(y)
 \end{equation}
 where $\psi_{s-}(y)$ is a Schwartz function on $\mathbb R$, 
 with  
\begin{equation*}
\operatorname {supp} (\psi _{s-})\subset \frac12\gamma_s (\scl s) ^{-1} [-1,1]\,,
\end{equation*}
 and equal to $ \psi_{\scl s}( y)$ 
 for $\abs y<\frac14\gamma_s  (\scl s) ^{-1} $. Then define 
 \begin{equation} \label{e.apm}
 a_{s\pm}(x)=\ind {{\boldsymbol \omega}_{s2}}(v(x))\int \phi_s(x-yv(x))\psi_{s\pm}(y)\; dy.
 \end{equation}
 Thus, $\phi_s=a_{s-}+a_{s+}$. Recalling 
 the notation $\operatorname S _{\Prm}$ in Theorem~\ref{t.laceyli1}, define 
 \begin{equation}\label{e.A}
 \operatorname A_\pm f:=\sum_{\substack{s\in\mathcal {AT}(\Prm) \\ \scl s\ge\kappa^{-1}\vLip }} 
 \ip \operatorname S_{\Prm} f, \alpha_s, a_{s\pm} 
 \end{equation}
 We will write $\Phi=\Phi\operatorname S_\Prm=\operatorname A_++\operatorname A_-+
 \operatorname B$, where $B$ is an operator defined in \eqref{e.Bdef} below.  
 The main fact we need concerns $\operatorname A_-$. 
 
 \begin{lemma}\label{l.A}
  There is a choice of $1<p_0<2$ so that 
 \begin{equation*}
 \norm \operatorname  A_- .p.\lesssim{}1,\qquad p_0<p<\infty.
 \end{equation*}
 The implied constant is independent of the value of $\Prm$, and the Lipschitz norm of $v$. 
 \end{lemma}

 The proof of this Lemma is given in the next section, modulo three additional Lemmata stated 
therein.    The following Lemma is  important for our  approach to the previous Lemma. 
 It is proved below. 
 
 \begin{lemma}\label{l.bessel}
 
  For each choice of   $\kappa^{-1}\vLip <{\mathsf {scl}}<\kappa\Prm$, we have the estimate 
 \begin{equation*}
 \sum_{s\in\mathcal {AT}(\Prm,{\mathsf {scl}})} \abs{\ip  \operatorname S_{\Prm} f, \alpha_s,}^2\lesssim{}\lVert f\rVert_2^2.
 \end{equation*}
 \end{lemma}

 Define 
 \begin{equation} \label{e.Bdef}
\operatorname  Bf:=\sum_{\substack{s\in\mathcal {AT}(\Prm) \\ \scl s\ge\kappa^{-1}\vLip }} \ip  \operatorname S_{\Prm} f, \beta_s,\phi_s 
 \end{equation}
 
 \begin{lemma}\label{l.B}
  For a Lipschitz vector field $v$, we have 
 \begin{equation*}
 \norm \operatorname B.p.\lesssim1,\qquad 2\le{}p<\infty.
 \end{equation*}
 \end{lemma}
\begin{proof}  For choices of integers $\kappa^{-1}\vLip\le{\mathsf {scl}}<\kappa\Prm$, consider the vector valued operator 
 \begin{align*}
 \operatorname T_{j,k}f:=\Bigl\{ \frac{\ip   \operatorname S_{\Prm} f, \beta_s,}
 {\sqrt{\abs{R_s}}} {\bf 1}_{ \{v(x)\in {\boldsymbol \omega}_{s2} \}}
  & {\operatorname{T}}_{c(R_s)}
 {\operatorname{D}}_{R_s}^\infty(\frac{1}{(1+|\cdot|^2)^N})(x)
 \\
 &\mid s\in\mathcal {AT}(\Prm,{\mathsf {scl}})\Bigr\}\,,
 \end{align*}
where $N$ is a fixed large integer.

Recall that $\beta_s$ is supported off of $\frac12\gamma_s R_s$. 
 This is bounded linear operator from $L^\infty(\mathbb R^2)$ to
 $\ell^\infty(\mathcal {AT}(\Prm,{\mathsf {scl}}))$. It has norm 
$\lesssim({\mathsf {scl}}/\vLip)^{-10}$.  Routine considerations will verify that 
\begin{equation*}
\operatorname T_{j,k} \,:\, L^2(\mathbb R^2) \longrightarrow \ell^2(\mathcal {AT}(\Prm,{\mathsf {scl}}))
\end{equation*}
 with a similarly favorable estimate on 
its norm. By interpolation, 
we achieve the same estimate for $\operatorname T_{j,k}$ from $L^p(\mathbb R^2)$ into $\ell^p(\mathcal {AT}(\Prm,{\mathsf {scl}}))$, $2\le{}p<\infty$. 

It is now very easy to conclude the Lemma by summing over scales in a brute force way, and using the 
methods of Lemma~\ref{l.fixedscale}. 
\end{proof}

We turn to  $\operatorname A_+$, as defined in \eqref{e.A}.  

\begin{lemma}\label{l.A+}
  We have the estimate 
\begin{equation*}
\norm \operatorname A_+.p.\lesssim{}1\qquad 2\le{}p<\infty.
\end{equation*}
\end{lemma} 

\begin{proof}  We redefine the vector valued operator $\operatorname T_{j,k}$ to be  
 \begin{align*}
 \operatorname T_{j,k}f:=\Bigl\{ \frac {\ip  \operatorname S_{\Prm} f, \alpha_s,}
 {\sqrt{\abs{R_s}}} & {\bf 1}_{\{v(x)\in {\boldsymbol \omega}_{s2}\}} 
 {\operatorname{T}}_{c(R_s)}
 {\operatorname{D}}_{R_s}^\infty\bigl( \frac{1}{(1+ \lvert  x\rvert ^2)^N } \bigr)
\\& \qquad 
\mid s\in\mathcal {AT}(\Prm,{\mathsf {scl}})\Bigr\}\,,
 \end{align*}
where $N$ is a fixed large integer.
 This operator is bounded from 
 \begin{equation*}
L^p(\mathbb R^2) \longrightarrow \ell^p(\mathcal {AT}(\Prm,{\mathsf {scl}}))
\,, \qquad 2\le{}p<\infty
\end{equation*}
 Its norm is at most ${}\lesssim1$. 
 
 But, for $s\in\mathcal {AT}(\Prm,{\mathsf {scl}})$, we have 
 \begin{equation}  \label{e.AaA}
 \abs{a_{s+}}\lesssim{}({\mathsf {scl}}/\vLip)^{-10}\abs{R_s}^{-1/2} (\operatorname M \ind {R_s}) ^{100}.
 \end{equation}
 Here $\operatorname M$ denotes the strong maximal function in the plane in the coordinates determined by $R_s$. 
 This permits one to again adapt the estimate of Lemma~\ref{l.fixedscale} to conclude the Lemma. 
 \end{proof}

 Now we  conclude that $\norm \Phi.2.\lesssim1$.  And since $\Phi=\operatorname A_-+
 \operatorname A_++\operatorname B$, it follows from the Lemmata of this section.
 	
	
\section*{Proofs of Lemmata } \label{s.lemmata}

 \subsection*{Proof of Lemma~\ref{l.A}}

 We have $\Phi=\operatorname A_-+\operatorname A_++\operatorname B$, so from \eqref{e.Phip}, Lemma~\ref{l.B} and   
Lemma~\ref{l.A+}, we deduce that 
$\norm \operatorname A_-.p.\lesssim{}1$ for all $2<p<\infty$.  
It remains for us to verify that $\operatorname A_-$ is of 
restricted weak type $p_0$ for some choice of $1<p_0<2$.  That is, we should verify that 
for all sets $F, G\subset\mathbb R^2$ of finite measure 
\begin{equation}\label{e.weak} 
\abs{ \ip  \operatorname A_- \ind F, \ind G ,}\lesssim{}\abs F^{1/p}\abs G^{1-1/p},\qquad p_0<p<2. 
\end{equation}
Since $\operatorname A_-$ maps $L^p$ into itself for $2<p<\infty$, it suffices to 
consider the case of $\abs F <\abs G$.  Since we assume only that 
the vector field is Lipschitz, we can use a dilation to assume that $1<\abs G<2$, 
and so this set will not explicitly enter into 
our estimates.

We fix the data  $F\subset\mathbb R^2$ of finite measure, ${\Prm}$, and 
vector field $v$ with $\norm v.\text{Lip}.\le\kappa\Prm$.  
Take $p_0=2-\kappa^2$.  
 We need a set of definitions that are inspired by the approach of Lacey and Thiele 
 \cite{laceythiele}, and are also used in  Lacey and Li \cite{laceyli1}.  
 For subsets $\mathcal S\subset\mathcal {A}_v:=\{s\in\mathcal {AT}(\Prm)\mid \kappa^{-1}\vLip\le{}\scl s<\kappa \Prm\}$, set 
 \begin{equation*}
 A^\mathcal S=\sum_{s\in\mathcal S}\ip  \operatorname S_{\Prm}\ind F, \alpha_s, a_{s-}
 \end{equation*}
 Set $\chi(x)=(1+\abs x)^{-1000/\kappa}$.  Define 
 \begin{equation}\label{e.zq}
 \chi_{R_s}^{(p)}:=\chi_s^{(p)}= \operatorname T_{c(R_s)}\operatorname D^p_{R_s}\chi, \qquad 0\le{}p\le\infty.
 \end{equation}
 And set $\widetilde{\chi}_s^{(p)}=\ind {\gamma_s R_s}\chi_s ^{(p)}$.

\begin{remark}\label{r.noPartialOrder}    It is typical to define a partial order on 
tiles, following an observation of C.~Fefferman \cite{feff}.  In this case, there doesn't 
seem to be an appropriate partial order.  Begin with this assumption on 
the  order relation `$ <$' on tiles: 
\begin{equation}\label{e.<property}
\text{If ${\boldsymbol \omega}_{s}\times R_{s}\cap {\boldsymbol \omega}_{s'}
\times R_{s'}\not=\emptyset$, then  $s$ and $s'$ are comparable under `$<$'.}
\end{equation}
It follows from transitivity of a partial order that  
that one can have tiles $ s_1 ,\dotsc, s_J$, with $ s_{j+1}<s_{j}$ for $ 1\le j<J$, 
$ J \simeq \log (\vLip \cdot \Prm)$, and yet  the rectangles $ R _{s_J}$ and $ R _{s_1}$ can 
be  far apart, namely $  R _{s_J} \cap  (cJ) R _{s_1}$, for a positive constant $ c$.  See 
Figure~\ref{f.noPartialOrder}.  (We thank the referee for directing us towards this
conclusion.)  Therefore, one cannot make the order relation transitive, and maintain 
control of the approximate localization of spatial variables, as one would wish. 
The partial order is essential to the argument of \cite{feff}, but 
while it is used in \cite{laceythiele}, it is not essential to that argument.  

\end{remark}

We recall a fact about the eccentricity.  There is an absolute constant $ K'$ so that 
for any two tiles $ s,s'$
\begin{equation}\label{e.K'}
\omega _{s}\supset \omega _{s'}\,, \ R_s\cap R _{s'}\neq \emptyset 
\quad \textup{implies} \quad  R_s \subset K' R _{s'}\,. 
\end{equation}
Figure~\ref{f.kappa} illustrates this in the case where the two rectangles 
$ R_s$ and $ R _{s'}$ have different widths, which is not the case here.

We define an order relation on tiles by  $ s \lesssim  s' $ iff 
$ \omega _{s} \supsetneq \omega _{s'}$ and $ R_s \subset \kappa  ^{-10} R _{s'}$. 
Thus, \eqref{e.<property} holds for this order relation, and it is certainly 
not transitive.

\begin{figure}
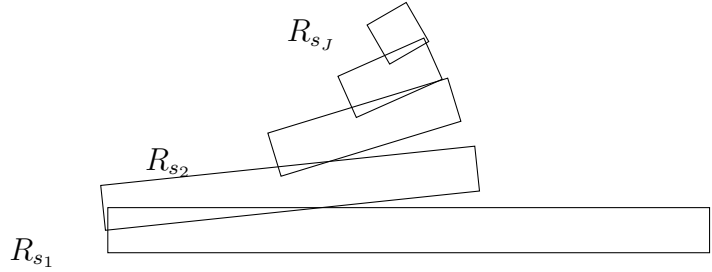
 
\begin{center} 
	\begin{pgfpicture}{0cm}{0cm}{5cm}{4cm}
	\pgfrect[stroke]{\pgfpoint{0cm}{0cm}}{\pgfpoint{8cm}{0.6cm}}
	\begin{pgfrotateby}{\pgfdegree{6}}
	\pgfrect[stroke]{\pgfpoint{0cm}{.3cm}}{\pgfpoint{5cm}{0.6cm}}
	\end{pgfrotateby}
	\begin{pgfrotateby}{\pgfdegree{17}}
	\pgfrect[stroke]{\pgfpoint{2.5cm}{.3cm}}{\pgfpoint{2.5cm}{0.6cm}}
	\end{pgfrotateby}
	\begin{pgfrotateby}{\pgfdegree{24}}
	\pgfrect[stroke]{\pgfpoint{3.75cm}{.3cm}}{\pgfpoint{1.25cm}{0.6cm}}
	\end{pgfrotateby}
	\begin{pgfrotateby}{\pgfdegree{30}}
	\pgfrect[stroke]{\pgfpoint{4.5cm}{.3cm}}{\pgfpoint{.6cm}{0.6cm}}
	\end{pgfrotateby}
	
	\pgfputat{\pgfxy(-1,0)} {\pgfbox[center,center]{$ R_ {s_1}$}}
	\pgfputat{\pgfxy(.8,1.2)} {\pgfbox[center,center]{$ R_ {s_2}$}}
	\pgfputat{\pgfxy(2.7,2.9)} {\pgfbox[center,center]{$ R_ {s_J}$}}
	\end{pgfpicture}
\end{center} 
	\caption{The rectangles $ R _{s_1} ,\dotsc, R _{s_J}$ of
	Remark~\ref{r.noPartialOrder}.}
	\label{f.noPartialOrder}	
\end{figure}

A {\em tree} is   a collection of tiles $\mathbf  T\subset\mathcal {A}_v$, for 
which there is a (non--unique) tile $\omega_{\mathbf T}\times R_{\mathbf T}
\in\mathcal {AT}(\Prm)$ with $R_s \subset 100 \kappa ^{-10} R_{\mathbf T}$, and
$ \omega _{s}\supset \omega _{\mathbf T}$ for all $s\in\mathbf  T$.  
Here, we deliberately use a somewhat larger constant $ 100 \kappa ^{-10} $ than 
we used in the definition of the order relation `$ \lesssim $.'


For $j=1,2$, call $\mathbf  T$ a $i$--tree if the tiles for all $ s,s'\in \mathbf T$, 
if $ \scl s \neq \scl {s'}$, then $ \omega _{si}\cap \omega _{s'i}=\emptyset$. 
$1 $--trees are especially important.  A few tiles in such a tree are depicted in Figure~\ref{f.tree}. 


\begin{remark}\label{r.pseudo}  This remark about the partial order `$ \lesssim $' and 
trees is useful to us below.  Suppose that we have two trees $ \mathbf T$, with 
top $ s (\mathbf T)$ and $ \mathbf T'$ with top $ s (\mathbf T')$.   Suppose in 
addition that $ s (\mathbf T') \lesssim s (\mathbf T)$.  Then, it is the case that 
$ \mathbf T \cup \mathbf T'$ is a tree with top $ s (\mathbf T)$.   
Indeed,  we must necessarily have $ \omega _{\mathbf T} 
\subsetneq \omega _{\mathbf T}$, since the $ R_s$ are from products of a central grid. 
Also, $ 100 \kappa ^{-1} R _{\mathbf T'}\subset 100 \kappa ^{-1} R _{\mathbf T}$.  
And so every tile in $ \mathbf T'$ could also be a tile in $ \mathbf T$.  
\end{remark}

Our proof is organized around these parameters and functions  associated to tiles and sets of tiles. 
Of particular note here are the first definitions of `density,' which have to be formulated 
to accommodate the lack of transitivity in the partial order.  
Note that in the first definition, the supremum is taken over tiles $ s'\in \mathcal {AT}$ 
of the same annular parameter as $ s$.  We choose the collection $ \mathcal {AT}$ as it 
is `universal,' covering all scales in a uniform way, due to different assumptions including 
\eqref{e.R3.5}. 
\begin{gather}
 \label{e.DENSE} 
\begin{split}
\dense {\mathbf  S}&:=\sup _{s'\in \mathcal {\mathcal AT} }
\Bigl\{  
 \int_{G\cap v^{-1}({\boldsymbol \omega}_{s'})} \widetilde\chi^{(1)}_{s'} \; dx
 \mid  \exists\, s\,,\, s'' \in \mathbf S \,:\, 
 \\ & \qquad 
  \boldsymbol \omega _{s}\supset \boldsymbol \omega _{s'} \supset \boldsymbol \omega _{s''}
  \,,\  R_s\subset 100  \kappa ^{-10} R _{s'}\,,
 \\ & \qquad 
 \,  R _{s'} \subset100  \kappa ^{-10} R _{s'} \Bigr\}
\end{split}
\\  \label{e.zD}
\Delta(\mathbf  T)^2:=\sum_{s\in\mathbf  T}\frac{\abs{\ip  \operatorname S_{\Prm} 1_F, \alpha_s,}^2}{\abs{R_{s}}}\ind {R_{ s}},\qquad \text{$\mathbf  T$ is a $1$--tree,} 
\\  \label{e.size}
\size {\mathbf  S}:=\sup_{\substack{ \mathbf  T\subset \mathbf  S \\
\text{$\mathbf  T$ is a $1$--tree}}} \dashint_{R_{\mathbf T}}\Delta(\mathbf  T)\; dx .
\end{gather}
Observe that $ \dense {\mathbf S}$ only really applies to `tree-like' sets of tiles, 
and that---and this is important---the tile $ s'$ that appear in \eqref{e.DENSE} are 
\emph{not} in $ \mathbf S$, but only assumed to be in $ \mathcal {AT}$.  
Finally, note that 
\begin{equation*}
\dense {s} \simeq 
 \int_{G\cap v^{-1}({\boldsymbol \omega}_{s})} \widetilde\chi^{(1)}_{s} \; dx
\end{equation*}
with the implied constants only depending upon $ \kappa $, $ \chi $, and other 
fixed quantities.

Observe these points about size.  First, it is computed relative to the truncated 
functions $ \alpha _{s}$, recall \eqref{e.gamma}.  Second, that for $p>1$,
\begin{equation}\label{e.square}
\norm \Delta(\mathbf  T).p. \lesssim \abs{ F}^{1/p}\,,
\end{equation}
because of a standard $L^p$ estimate for a Littlewood-Paley square function.
Third, that $\size {{\mathcal A}_v(\Prm) } \lesssim1$.   
 And fourth, that one has an estimate of John-Nirenberg type. 
 \begin{lemma}\label{l.jn} For a $ 1$-tree $ \mathbf T$ we have the estimate 
 \begin{equation*}
 \norm \Delta (\mathbf T). p. \lesssim \size {\mathbf T} \lvert  R _{\mathbf T}\rvert ^{1/p}
 \,, \qquad 1<p<\infty \,. 
 \end{equation*}
 \end{lemma}
 Proofs of results of this  type  are well represented in the literature. See 
 \cites{gl1,MR1945289}.

%
%


\begin{figure}
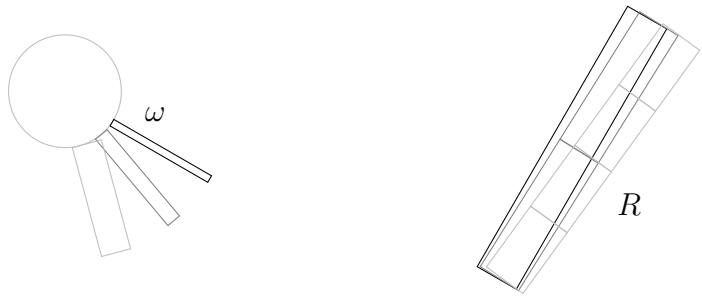
  
\begin{center} 
	\begin{pgfpicture}{0cm}{0cm}{5cm}{4cm}
	\begin{pgftranslate}{\pgfpoint{-1.5cm}{3cm}}
	{\color{lightgray}
	\pgfcircle[stroke]{\pgfpoint{0cm}{0cm}}{0.75cm} 
	}
		\begin{pgfrotateby}{\pgfdegree{-30}}
		\pgfrect[stroke]{\pgfpoint{.75cm}{-.1cm}}{\pgfpoint{1.5cm}{0.1cm}}
		\end{pgfrotateby}
		\pgfputat{\pgfpoint{1.2cm}{-.3cm}}{\pgfbox[center,center]{$\omega$}}
		{\textcolor{gray}{ 
		\begin{pgfrotateby}{\pgfdegree{-50}}
		\pgfrect[stroke]{\pgfpoint{.75cm}{-.1cm}}{\pgfpoint{1.5cm}{0.2cm}}
		\end{pgfrotateby}
		}}
		{\textcolor{lightgray}{ 
		\begin{pgfrotateby}{\pgfdegree{-75}}
		\pgfrect[stroke]{\pgfpoint{.75cm}{-.1cm}}{\pgfpoint{1.5cm}{0.4cm}}
		\end{pgfrotateby}
		}}
	\end{pgftranslate}
	\begin{pgftranslate}{\pgfpoint{4cm}{-.5cm}}
	\begin{pgfrotateby}{\pgfdegree{-30}}
	{
	\pgfrect[stroke]{\pgfpoint{-0.6cm}{1cm}}{\pgfpoint{0.6cm}{4cm}}
	}\end{pgfrotateby}
	{\textcolor{gray}{ 
	\begin{pgfrotateby}{\pgfdegree{-32}}
	\pgfrect[stroke]{\pgfpoint{-0.6cm}{1cm}}{\pgfpoint{0.6cm}{2cm}}
	\pgfrect[stroke]{\pgfpoint{-0.6cm}{3cm}}{\pgfpoint{0.6cm}{2cm}}
	\end{pgfrotateby}
	}}
	{\textcolor{lightgray}{ 
	\begin{pgfrotateby}{\pgfdegree{-36}}
	\pgfrect[stroke]{\pgfpoint{-0.6cm}{1cm}}{\pgfpoint{0.6cm}{1cm}}
	\pgfrect[stroke]{\pgfpoint{-0.6cm}{2cm}}{\pgfpoint{0.6cm}{1cm}}
	\pgfrect[stroke]{\pgfpoint{-0.6cm}{3cm}}{\pgfpoint{0.6cm}{1cm}}
	\pgfrect[stroke]{\pgfpoint{-0.6cm}{4cm}}{\pgfpoint{0.6cm}{1cm}}
	\end{pgfrotateby}
	}}
	\pgfputat{\pgfpoint{2cm}{2cm}}{\pgfbox[center,center]{$R$}}
	\end{pgftranslate}
	\end{pgfpicture}
	\end{center} 
	
	\caption{A few possible tiles in a $1 $--tree.  Rectangles $\omega_s $ 
	are on the left in different shades of gray. 
	Possible locations of $R_s $ are in the same shade of gray. } 
	\label{f.tree}
	\end{figure}

Given a  set of tiles, say that $\operatorname{count}(\mathbf  S)<A$ iff $\mathbf  S$ is a union of 
 trees 
$\mathbf  T\in\mathcal T$ for which 
\begin{equation*}
\sum_{T\in\mathcal T}\abs{\sh{\mathbf T} }<A.
\end{equation*}
We will also use the notation $\operatorname{count}(\mathbf  S)\lesssim{}A$, implying the existence of an absolute 
constant $K$ for which $\operatorname{count}(\mathbf  S)\le{}KA$.  

  The principal organizational Lemma is

  
\begin{lemma}\label{l.dense+size}
  
Any finite  
 collection of tiles $\mathbf  S\subset\mathcal A_v$ is a union of four 
 subsets 
\begin{equation*}
\mathbf  S_{\text{\rm light}},\quad \mathbf  S_{\text{\rm small}},\quad  
\mathbf  S_{\text{\rm large}}^\ell,\ \ell=1,2. 
\end{equation*}
They  satisfy these properties.  
\begin{gather} \label{e.small}
\size {\mathbf  S_{\text{\rm small}}}<\tfrac12 \size{\mathbf S} ,
\\ \label{e.light}
\dense {\mathbf  S_{\text{\rm light}}}<\tfrac12 \dense{\mathbf S} ,
\end{gather}
and both  $\mathbf  S_{\text{\rm large}}^\ell$ are  unions of  
 trees $\mathbf  T\in \mathcal T^\ell $, for which we have the estimates
\begin{gather}  \label{e.large-1}
\operatorname{count}(\mathbf  S_{\text{\rm large}}^1)\lesssim{}
	\begin{cases}
	{\size{\mathbf S}}^{-2-\kappa}\abs F  &\\
	\size{\mathbf S}^{-p}\dense {\mathbf  S}^{-M}\abs F 
	 \\ \qquad+\,{\size{\mathbf S}}^{1/\kappa}
            \dense{\mathbf S}^{-1}   &\\ 
	\dense{\mathbf S}^{-1}
	\end{cases}
\\ \label{e.large-2}
\operatorname{count}(\mathbf  S_{\text{\rm large}}^2)\lesssim{}
	\begin{cases}
	{\size{\mathbf S}}^{-2}(\log 1/\size{\mathbf S} )^3\abs F &\\
	\size{\mathbf S}^{\kappa/50}\dense{\mathbf S}^{-1}
	\end{cases}
\end{gather}
\end{lemma} 
What is most important here is the middle estimate in \eqref{e.large-1}.  Here, 
$ p$ is as in Conjecture~\ref{j.laceyli2}, and $ M>0$ is only a function of $N  $ 
in that Conjecture.

The estimates that involve $\size{\mathbf S}^{-2}\abs F$ are those that follow from 
orthogonality considerations.  
The estimates in $\dense{\mathbf S} ^{-1}$ are those that follow from density 
considerations which are less complicated.
However, in the second half of \eqref{e.large-2},  the small positive power of
size is essential for us. 
All of these estimates are all variants of those in  \cite{laceythiele}.

The middle estimate of \eqref{e.large-1} is not of this type, and is the key ingredient that permits us to 
obtain an estimate below $L^2$.  Note that it gives the best bound for collections with moderate density and size. 
For it we shall appeal to our assumed Conjecture~\ref{j.laceyli2}.

Logarithms, such as those that arise in \eqref{e.large-2}, arise from our truncation 
arguments, associated with the parameters $ \gamma _s$ in \eqref{e.gamma}.

For individual trees, we need two estimates. 

\begin{lemma}\label{l.1-tree}
  If $\mathbf  T$ is a $1$--tree with 
$\dashint_{R_{\mathbf T}}\Delta(\mathbf  T)\ge\sigma$, then we have  
\begin{equation}\label{e.Tlowersigma} 
\abs{F\cap \sigma^{-\kappa}R_{\mathbf T}}\gtrsim\sigma^{1+\kappa}\abs{R_{\mathbf T}}.
\end{equation}
\end{lemma}

\begin{lemma}\label{l.tree}
  For 
 trees $\mathbf  T$ we have the estimate 
\begin{equation} \label{e.tree} 
\sum_{s\in\mathbf  T} \abs{\ip  {\operatorname S_{\Prm}} \ind F, \alpha_{s},\ip  a_{s-}, \ind G ,}
\lesssim\Psi\bigl(\dense {\mathbf  T} \size {\mathbf  T}\bigr)  \abs {\sh{\mathbf T} }.
\end{equation}
Here $\Psi(x)=x\abs{\log c x} $, and inside the logarithm, $c $ is a small fixed constant,
to insure that $c\dense {\mathbf  T }\cdot \size {\mathbf  T }<\frac12 $, say.
\end{lemma}

Set 
\begin{equation*}
\operatorname{Sum}(\mathbf  S):=\sum_{s\in\mathbf  S} \abs{\ip  {\operatorname S_{\Prm}} \ind F, \alpha_{s},\ip  a_{s-}, \ind G ,}
\end{equation*}
We want to provide the  bound $\operatorname{Sum}(\mathcal A_v)\lesssim{} \abs F^{1/p}$ for $p_0<p<2$.  We have the 
trivial bound 
\begin{equation} \label{e.dense-size-count}
\operatorname{Sum}(\mathbf  S)\lesssim{}\Psi\bigl(\dense{\mathbf S}  \size{\mathbf S}\bigr) \operatorname{count}(\mathbf  S). 
\end{equation}
It is incumbent on us to provide a decomposition of $\mathcal A_v $ into sub-collections for which this last estimate 
is effective.

By inductive application of our principal organizational Lemma~\ref{l.dense+size},  $\mathcal A_v$ is the union of $\mathbf  S_{\delta,\sigma}^\ell $, $\ell=1,2$   
for $\delta,\sigma\in {\mathbf 2}:=\{2^n\mid n\in\mathbb Z\,, n\leq 0\}$, satisfying 
\begin{gather} 
\dense {\mathbf  S_{\delta,\sigma}^\ell} \lesssim\delta,
\\
\size {\mathbf  S_{\delta,\sigma}^\ell} \lesssim\sigma, 
\\
\operatorname{count}(\mathbf  S_{\delta,\sigma}^\ell)\lesssim{} 
\begin{cases} 
\min( \sigma^{-2-\kappa}\abs F , \delta^{-M}\sigma^{-p}{\abs F} 
                + \sigma^{1/\kappa}\delta^{-1}
                         , \delta^{-1}) & \ell=1,
\\
\min( \sigma^{-2}(\log 1/\sigma)^3\abs F, \delta^{-1}\sigma^{\kappa/50} ) & \ell=2 
\end{cases} 
\end{gather}
Using \eqref{e.dense-size-count}, we see that 
\begin{equation} \label{e.sum}
\begin{split}
\operatorname{Sum}(\mathbf  S_{\delta,\sigma}^1)&\lesssim{} 
\min( \Psi(\delta)\sigma^{-1-\kappa}\abs F , \delta^{-M+1}\sigma^{-p+1}{\abs F} 
                     + \sigma^{1/\kappa+1}, \sigma)
\\
\operatorname{Sum}(\mathbf  S_{\delta,\sigma}^2)&\lesssim{} 
\min( \Psi(\delta)\sigma^{-1}(\log 1/\sigma)^4\abs F, \sigma^{1+\kappa/50} )
\end{split}
\end{equation}

One can  check that for $\ell=1,2$, 
\begin{equation} \label{e.routine}
\sum_{\delta,\sigma\in\mathbf 2 }\operatorname{Sum}(\mathbf  S_{\delta,\sigma}^\ell)\lesssim{}\abs F^{1/p},\qquad p_0<p<2. 
\end{equation}
This completes the proof of Lemma~\ref{l.A}, aside from the proof of Lemma~\ref{l.dense+size}.

\begin{proof}[Proof of \eqref{e.routine}.] 
We can assume that $ \lvert  G\rvert =1$, and that $ \lvert  F\rvert \le 1 $, 
for otherwise the result follows from the known $ L ^{p}$ estimates, for $ p>2$ 
and measurable vector fields, see Theorem~\ref{t.laceyli1}.

The case of $ \ell =2$ in \eqref{e.routine} is straightforward.  
Notice that in \eqref{e.sum}, for $ \ell =2$, the two terms in the minimum
are roughly comparable, ignoring logarithmic terms, for 
\begin{equation*}
\delta \lvert  F\rvert \simeq \sigma ^{2+\kappa /50}\,.  
\end{equation*}
Therefore, we set 
\begin{align*}
T_1 &= \{ (\delta ,\sigma )\in \mathbf 2 \times \mathbf 2 
\mid  \delta \lvert  F\rvert\le \sigma ^{2+\kappa /50} \le \lvert  F\rvert \}\,, 
\\
T_2 &= \{ (\delta ,\sigma )\in \mathbf 2 \times \mathbf 2 
\mid  \sigma ^{2+\kappa /50}\le \delta \lvert  F\rvert\}
\end{align*}
and $ T_3=\mathbf 2 \times \mathbf 2-T_1-T_2$.  

We can estimate 
\begin{align*}
\sum _{ (\delta ,\sigma )\in T_1 } 
\operatorname{Sum}(\mathbf  S_{\delta,\sigma}^2)
& \lesssim 
\sum _{ (\delta ,\sigma )\in T_1 } 
\Psi(\delta)\sigma^{-1}(\log 1/\sigma)^4\abs F 
\\
& \lesssim 
\sum _{  \substack{\sigma \in \mathbf 2\\ \sigma ^{2+\kappa /50}  \le \lvert  F\rvert } }  
\sigma ^{1+\kappa /75}
\\
& \lesssim 
\lvert  F\rvert ^{1/p_0}\,, \qquad  p_0= \frac {2+\kappa /50} {1+\kappa /75}<2\,.  
\end{align*}
Notice that we have absorbed harmless logarithmic terms into a slightly smaller 
exponent in $ \sigma $ above. 

The second term is 
\begin{align*}
\sum _{ (\delta ,\sigma )\in T_2} 
\operatorname{Sum}(\mathbf  S_{\delta,\sigma}^2)
& \lesssim 
\sum _{ (\delta ,\sigma )\in T_2 } 
\sigma^{1+\kappa/50}
\\
& \lesssim 
\sum _{  \substack{\delta  \in \mathbf 2  } }  
(\delta F) ^{1/p_1}\,, \qquad p_1 = \frac {2+\kappa /50} {1+\kappa /50}<2\,, 
\\
& \lesssim 
\lvert  F\rvert ^{1/p_1}\,.
\end{align*}

The third term is 
\begin{align*}
\sum _{ (\delta ,\sigma )\in T_3} 
\operatorname{Sum}(\mathbf  S_{\delta,\sigma}^2)
& \lesssim 
\sum _{ (\delta ,\sigma )\in T_3 } 
\Psi(\delta)\sigma^{-1}(\log 1/\sigma)^4\abs F 
\\
& \lesssim 
\sum _{  \substack{\sigma \in \mathbf 2\\ \sigma ^{2+\kappa /50}  \ge \lvert  F\rvert } }  
\sigma ^{-1} \lvert  F\rvert ^{1- \kappa /75} 
\\
& \lesssim 
\lvert  F\rvert ^{1/p_0}\,.
\end{align*}
Here, we have again absorbed harmless logarithms into a slightly smaller 
power of $ \lvert  F\rvert $, and $ p_0<2$ is as in the first term.  

\bigskip 

The novelty in this proof is the proof of \eqref{e.routine} in the case of $ \ell =1$. 
We comment that if one uses the proof strategy just employed, that is only relying 
upon the first and last estimates from the minimum in \eqref{e.sum}, in the case 
of $ \ell =1$, one will only show that $ \lvert  F\rvert ^{1/2}  $.  

In the definitions below, we will have a choice of $ 0<\tau <1$, where 
$ \tau = \tau (M,p) \simeq M ^{-1} \cdot (2-p)  $ will only depend upon $ M$ and $ p$ in \eqref{e.sum}. 
($ \tau $ enters into the definition of $ T_4$ and $ T_5$ below.) 
The choice of $ 0<\kappa < \tau $ will be specified below.  
\begin{align*}
T_1 &= \{ (\delta ,\sigma )\in \mathbf 2 \times \mathbf 2 
\mid  \lvert  F\rvert ^{\frac 1 {(2+\kappa ) (1+\kappa )} }\le \sigma \}\,, 
\\
T_2 &= \{ (\delta ,\sigma )\in \mathbf 2 \times \mathbf 2 
\mid   \sigma < \lvert  F\rvert ^{\frac 1 {2-\kappa } }  \,,
\  \delta \ge \sigma ^{1/\kappa } \}\,, 
\\
T_3 &= 
\{ (\delta ,\sigma )\in \mathbf 2 \times \mathbf 2 
\mid   \sigma < \lvert  F\rvert ^{\frac 1 {2-\kappa } }  \,,
\  \delta >\sigma ^{1/\kappa } \}\,, 
\\
T_4 &= \{ (\delta ,\sigma )\in \mathbf 2 \times \mathbf 2
\mid     \lvert  F\rvert ^{\frac 1 {2-\kappa } } \le  \sigma < \lvert F\rvert ^{\frac 1 {(2+\kappa ) (1+\kappa )} }   \,,
\  \delta >\sigma ^{\tau  } \}\,, 
\\
T_5&=
\{ (\delta ,\sigma )\in \mathbf 2 \times \mathbf 2
\mid     \lvert  F\rvert ^{\frac 1 {2-\kappa } } \le  \sigma < \lvert F\rvert ^{\frac 1 {(2+\kappa ) (1+\kappa )} }   \,,
\  \delta \le \sigma ^{\tau  } \}\,,
\end{align*}
Let 
\begin{equation*}
\mathcal T (T)=\sum _{ (\delta ,\sigma )\in T} 
\operatorname{Sum}(\mathbf  S_{\delta,\sigma}^1)\,. 
\end{equation*}

Note that  for $ T_1$ we can use the first term in the minimum in \eqref{e.sum}. 
\begin{align*}
\mathcal T (T_1) & \lesssim 
	\sum _{(\delta ,\sigma )\in T_1  } 
	\delta \sigma ^{-1-\kappa }\lvert  F\rvert
\\
& \lesssim \sum _{\substack{\sigma  \in \mathbf 2\\ \sigma  \ge \lvert  F\rvert 
^{\frac 1 {(2+\kappa ) (1+\kappa )} }}} 
\sigma^{-1-\kappa } \lvert  F\rvert 
\\
& \lesssim \lvert  F\rvert ^{1- \frac 1 {2+\kappa }}\,. 
\end{align*}
This last exponent on $ \lvert  F\rvert $ is strictly larger than $ \frac12$ as desired. 
The point of the definition of $T_1 $ is that when it comes time to use the 
middle term of the minimum 
for $ \ell =1$ in \eqref{e.sum}, we can restrict attention to the term 
\begin{equation*}
 \delta^{-M+1}\sigma^{-p+1}{\abs F} \,. 
\end{equation*}

For the collection $ T_2$,  use the last term in the minimum in 
\eqref{e.sum}.  
\begin{align*}
\mathcal T (T_2) & \lesssim 
	\sum _{(\delta ,\sigma )\in T_2  } 
	\sigma 
\\
& \lesssim  
\sum _{\substack{\sigma \in \mathbf 2\\ \sigma  \le \lvert  F\rvert ^{\frac 1 {2-\kappa }}  }} 
\sigma \log 1/\sigma 
\\
& \lesssim 
\lvert  F\rvert ^{ \frac 1 {2-\kappa /2}}\,.
\end{align*}
Again, for $ 0<\kappa <1$, the exponent on $ \lvert  F\rvert $ 
above is strictly greater than $ 1/2$. 

The term $ T_3$ can be controlled with the first term in the minimum in 
\eqref{e.sum}.  
\begin{align*}
\mathcal T (T_3) & \lesssim 
	\sum _{(\delta ,\sigma )\in T_3  }   
	\delta \sigma ^{-1-\kappa } \lvert  F\rvert
\\
& \lesssim \sum _{\sigma \in \mathbf 2} \sigma \lvert  F\rvert \lesssim \lvert  F\rvert \,. 
\end{align*}

The term $ T_4$ is the heart of the matter.  It is here that we use the 
middle term in the minimum of \eqref{e.sum}, and that the role of $ \tau $ becomes 
clear. We estimate 
\begin{align*}
\mathcal T (T_4)  &\lesssim 
	\sum _{(\delta ,\sigma )\in T_4  }   
\delta ^{-M} \sigma ^{-p+1} \lvert  F\rvert
\\
& \lesssim \sum _{\substack{\delta \in \mathbf 2\\ \delta \ge \lvert  F\rvert ^{\tau }  }}
\delta ^{-M} \lvert  F\rvert ^{1- \frac {p-1} {2-\kappa }}
\\
& \lesssim \lvert  F\rvert ^{1- \frac {p-1} {2-\kappa }- M \tau }\,. 
\end{align*}
Recall that $ 1<p<2$, so that $ 0<p-1<1$.  Therefore, for $ 0<\kappa  $ 
sufficiently small, of the order of $ 2-p$, we will have 
\begin{equation*}
1- \frac {p-1} {2-\kappa } > \tfrac 12 + \tfrac {2-p}4\,. 
\end{equation*}
Therefore, choosing $ \tau \simeq (2-p)/M$ will leave us with a power on $ \lvert  F\rvert $
that is strictly larger than $ \tfrac 12$. 

The previous term did not specify $ \kappa >0$.  Instead it shows that for 
$ 0<\kappa <1$ sufficiently small, we can make a choice of $ \tau $,  that is 
independent of $ \kappa $, for which $ \mathcal T (T_4)$ admits the required control. 
The bound in the last term will specify a choice of $ \kappa $ on us.  We estimate 
\begin{align*}
\mathcal T (T_5)  &\lesssim 
	\sum _{(\delta ,\sigma )\in T_5  }   
\delta \sigma ^{-1-\kappa } \lvert  F\rvert 
\\& \lesssim 
\sum _{\substack{\sigma \in \mathbf 2\\ \sigma \ge  \lvert  F\rvert ^{\frac 1 {2-\kappa }}  }}
\sigma ^{-1-\kappa } \lvert  F\rvert ^{1+\tau } 
\\
& \lesssim 
\lvert  F\rvert ^{1+ \tau + \frac {1+\kappa } {2-\kappa }} 
\end{align*}
Choosing $ \kappa = \tau /6$ will result in the estimate 
\begin{equation*}
\lvert  F\rvert ^{\frac {1+\tau} 2 }\,, 
\end{equation*}
which is as required, so our proof is finished. 
\end{proof}

\begin{remark}\label{r.conditional} The resolution of Conjecture~\ref{j.conditional} would 
depend upon refinements of Lemma~\ref{l.dense+size}, as well as using the restricted 
weak type approach of \cite{mtt}. 
\end{remark}

\subsection*{Proof of Lemma~\ref{l.bessel}}
  We only consider tiles 
$s\in\mathcal {AT}(\Prm,{\mathsf {scl}})$, and sets ${\boldsymbol \omega}\in{\boldsymbol \Omega}$ which are associated to one of these tiles. 
For an element $a=\{a_s\}\in\ell^2( \mathcal {AT}(\Prm,{\mathsf {scl}}))$, 
\begin{equation*}
\operatorname T_{{\boldsymbol \omega}}a=\sum_{s\,:\,{\boldsymbol \omega}_{s}={\boldsymbol \omega}} a_s {\operatorname S}_\Prm\alpha_s 
\end{equation*}
For $\abs{{\boldsymbol \omega}_{s}}=\abs{{\boldsymbol \omega}_{s'}}$, note that 
$\operatorname {dist}({\boldsymbol \omega}_{s},{\boldsymbol \omega}_{s'})$ is measured in units of ${\mathsf {scl}}/\Prm$.  

By a lemma of Cotlar and Stein, it suffices to provide the estimate 
\begin{equation*}  
\norm \operatorname T_{{\boldsymbol \omega}}
\operatorname T_{{\boldsymbol \omega}'}^* .2.\lesssim{}\rho^{-3},\qquad \rho=1+\frac{\Prm}{{\mathsf {scl}}}\text{dist}({\boldsymbol \omega},{\boldsymbol \omega}').
\end{equation*}
Now, the estimate $\norm \operatorname T_{{\boldsymbol \omega}}.2.\lesssim1$ 
is obvious.  For the case ${\boldsymbol \omega}\not={\boldsymbol \omega}'$, 
by Schur's test, it suffices to see that 
\begin{equation} \label{e.schur} 
\sup_{s'\,:\,{\boldsymbol \omega}_{s'}={\boldsymbol \omega}'}
\sum_{s\,:\,{\boldsymbol \omega}_{s}=
{{\boldsymbol \omega}}}
\abs{\ip  {\operatorname S}_\Prm\alpha_s, {\operatorname S}_\Prm\alpha_{s'},}
\lesssim{}\rho^{-3}.
\end{equation}

For tiles $s'$ and $s$ as above, recall that $\ip  \varphi_s, \varphi_{s'},=0$, note that 
\begin{equation*}
\frac {\abs{R_{s'}\cap R_s }}{\abs{R_s}}
\lesssim{}\frac{\mathsf {scl}}{\Prm\,{\text{dist}({\boldsymbol \omega},
{\boldsymbol \omega}')}}
 \simeq \rho^{-1},
\end{equation*}
and in particular, for a fixed $s'$,  let $\mathbf  S_{s'}$ be those $s$ for which 
$\rho R_s \cap \rho R_{s'}\not=\emptyset$. 
Clearly,   
\begin{equation*}
\operatorname{card}(\mathbf  S_{s'})\lesssim{}
\frac {\abs{\rho R_s}}{\abs{2\rho R_{s'}
     \cap 2\rho R_s }}\rho \simeq\rho^2 
\end{equation*}

If for $r>1$,  $rR_s \cap rR_{s'}=\emptyset$, then 
it is routine to show that 
\begin{equation*}
\abs{ \ip  {\operatorname S}_\Prm\alpha_s, {\operatorname S}_\Prm\alpha_{s'},}\lesssim{} 
r^{-10}
\end{equation*}
And so we may directly sum over those $s\not\in\mathbf  S_{s'}$, 
\begin{equation*} 
\sum_{s\not\in\mathbf  S_{s'}}\abs{\ip  {\operatorname S}_\Prm\alpha_s, {\operatorname S}_\Prm\alpha_{s'},}\lesssim{} \rho^{-3}.
\end{equation*}

For those $s\in\mathbf  S_{s'}$, we estimate the inner product in frequency variables.  Recalling the definition of $\alpha_s=
(\operatorname T_{c(R_s)}\operatorname D^{\infty}_{\gamma_s R_s}\zeta)\varphi_s$, 
we have 
\begin{equation*}
\widehat\alpha_s=(\operatorname{Mod}_{-c(R_s)}
\operatorname D^1_{\gamma_s^{-1}\omega_s}\widehat\zeta)* \widehat\varphi_s.
\end{equation*}
Recall that $\zeta$ is a smooth compactly supported Schwartz function.
We estimate the inner product 
\begin{equation*}
\abs{ \ip  \widehat{{\operatorname S}_\Prm\alpha_s}, \widehat{{\operatorname S}_\Prm\alpha_{s'}},}
\end{equation*}
 without  appealing to cancellation.  Since we choose the function $\widehat\lambda$ 
to be supported in an annulus $\frac12 <\abs\xi<\frac32$ so that 
$\widehat{\lambda_\Prm}=\widehat{\lambda}(\cdot/\Prm)$ is supported in the annulus
$ \frac12 \Prm< |\xi|< \frac32\Prm$.
 We can restrict our attention to this same range of $\xi$. 
  In the region $\abs\xi>\Prm/4$, suppose, without loss of generality, that $\xi$ is closer to $\omega_s$ than $\omega_{s'}$. Then since $\omega_s$ and $\omega_{s'}$ are separated by an amount $\gtrsim{} \Prm\, 
  \text{dist}({\boldsymbol \omega},{\boldsymbol \omega}')$, 
\begin{align*}
\abs{ \widehat\alpha_s(\xi)\widehat\alpha_{s'}(\xi)}\lesssim{}&
\chi^{(2)}_{\omega_s}(\xi)\chi^{(2)}_{\omega_{s'}}(\xi)
\Bigl(\frac{\Prm}{\mathsf {scl}} \text{dist}({\boldsymbol \omega},{\boldsymbol \omega})\Bigr)^{-20}
\\&\lesssim \chi^{(2)}_{\omega_s}(\xi)\chi^{(2)}_{\omega_{s'}}(\xi)
 \rho^{-20}.
\end{align*}
Here, $\chi$ is the non--negative bump function in \eqref{e.zq}.
Hence, we have the estimate 
\begin{equation*}
\int \abs{{\widehat{\lambda_{\Prm}}}(\xi)}^2\abs{ \widehat\alpha_s(\xi)\widehat\alpha_{s'}(\xi)}d\xi\lesssim \rho^{-10}.
\end{equation*}
This is summed over the ${}\lesssim \rho^2$ possible choices of $s\in\mathbf  S_{s'}$, giving the estimate 
\begin{equation*}
\sum_{s\in \mathbf  S_{s'}}
\abs{ \ip  {\operatorname S}_\Prm\alpha_s,  {\operatorname S}_\Prm\alpha_{s'}, } \lesssim{} \rho^{-8}\lesssim{}\rho^{-3}.
\end{equation*}
This is the   proof of \eqref{e.schur}. And this concludes the proof of Lemma~\ref{l.bessel}.


\subsection*{Proof of the Principal Organizational Lemma~\ref{l.dense+size}}

Recall that we are to decompose  $\mathbf  S$   into four distinct 
 subsets satisfying the favorable estimates of that Lemma. 
 For the remainder of the proof set 
$\dense{\mathbf S}:=\delta$ and $\size{\mathbf S}:=\sigma$.  
Take $\mathbf  S_{\text{light}}$ to be all those 
$s\in\mathbf  S$ for which there is no tile $s' \in \mathcal {AT}$
of density at least $\delta/2$ for which $s \lesssim s'$.  It is clear that  this set so constructed has 
density at most $\delta/2$,  that this is a 
  set of tiles, and that  $\mathbf  S_1:=\mathbf  S-\mathbf  S_{\text{light}}$ is also 
. 

\smallskip

The next Lemma and proof comment on the method we use to obtain the  middle estimate in \eqref{e.large-1} which
depends upon the Lipschitz Kakeya Maximal Function Conjecture~\ref{j.laceyli2}. 
It will be used to obtain the important inequality \eqref{goodest} below.

\begin{lemma}\label{l.add1}
Suppose we have a collection of trees $\mathbf  T\in\mathcal T$, 
with  these conditions. 
\begin{description}
\item[a] For $ \mathbf T\in \mathcal T$ there is a $ 1$-tree $ \mathbf T_1\subset \mathbf T$
with 
\begin{equation}\label{e.add1}
\int _{R _{\mathbf T}} \Delta (\mathbf T_1) \; dx \ge \kappa \sigma \,,
\end{equation}
\item  [b] Each tree has  top element $s(\mathbf  T):= \omega_{\mathbf T} \times R_{\mathbf T}$  
of density at least $\delta$.  

\item  [c] The collections of tops $ \{ s (\mathbf T) \mid 
\mathbf T\in \mathcal T\}$ are pairwise incomparable under the order relation `$ \lesssim $'. 

\item [d]  For all $ \mathbf T\in \mathcal T$, 
$\gamma_{\mathbf  T}=\gamma_{\omega_{\mathbf T}\times R_{\mathbf T}} \geq \kappa^{-1/2}
\sigma^{-\kappa/5 N}$. Here, $N$ is the exponent on $\delta $ in 
Conjecture~\ref{j.laceyli2}.
\end{description}
Then we have
\begin{align} \label{e.middle}    
\sum_{\mathbf  T\in\mathcal T} \abs {R_{\mathbf T}}\lesssim{}&\delta^{-Np-1} \sigma^{-p(1+\kappa/4)}{\abs F}
                            + \sigma^{1/\kappa}\delta^{-1}.
\\
\label{e.last}
\sum_{\mathbf  T\in\mathcal T} \abs { R_{\mathbf T}}\lesssim{}&\delta^{-1}.
\end{align}
\end{lemma}

Concerning the role of $\gamma _{\mathbf T}$, recall from the definition, \eqref{e.gamma}, that 
$\gamma _s$ is a quantity that grows as does the 
ratio $ \scl s/ \vLip $, hence there are only $ \lesssim \log \sigma ^{-1} $ scales of tiles 
that do not satisfy the assumption \textbf{d} above.

\begin{proof}  
 Our primary 
interest is in \eqref{e.middle}, which is a consequence of our assumption 
about the Lipschitz Kakeya Maximal Functions, Conjecture~\ref{j.laceyli1}.

 Set 
\begin{equation*}
\mathsf s(\mathbf  T):={\omega}_{\mathbf T}\times \sigma^{-\kappa/10 N} R_{\mathbf T}\,. 
\end{equation*}
Let us begin by noting that 
\begin{align} \label{e.a1}
\kappa^{-1}\vLip\le{}\scl {\mathsf s(\mathbf  T)}&\le{}\kappa\prm {\mathsf s(\mathbf  T)},
\qquad \mathbf  T\in\mathcal T\,,
\\ \label{e.a2}
\dense {\mathsf s(\mathbf  T)}\ge{}&\delta\sigma^{\kappa/10 N}\,,
\qquad \mathbf  T\in\mathcal T\,,
  \\ \label{e.a3}
  \abs{ F \cap R_{\mathsf s(\mathbf  T)} }
  \ge{}& \sigma^{1+\kappa/4 N}\abs {R_{\mathsf s(\mathbf  T)}}\, .
  \end{align}
The conclusion \eqref{e.a1} is straightforward, as is \eqref{e.a2}.   
The inequality  \eqref{e.a3} follows from Lemma~\ref{l.1-tree}. 

Note that the length of $\sigma^{-\kappa/10 N} R_{\mathbf T}$ satisfies 
\begin{equation}\label{e.add2}
\begin{split}
\sigma^{-\kappa/10 N} \name L {R_{\mathbf T}} 
& \le \gamma _{\mathbf T} \name L {R_{\mathbf T}} 
\\
& \le \sqrt {\frac {\scl s} {\vLip}} \le(100 \vLip) ^{-1} \,.  
\end{split}
\end{equation}
This  is the condition \eqref{e.shortlength} that we impose in the definition 
of the Lipschitz Kakeya Maximal Functions.

Observe that we can regard $\prm {\mathsf s(\mathbf  T)}
\simeq \sigma^{\kappa/10}\Prm$ as a constant independent of $\mathbf  T$.  

  The point of these observations is that  our assumption about the Lipschitz Kakeya 
  Maximal Function applies to the  maximal function formed over the 
  set of tiles $\{{\mathsf s(\mathbf  T)} \mid \mathbf  T\in\mathcal T\}$.  
  And it will be applied below. 

 Let $\mathcal T_k$ be the collection of trees so that $ \mathbf T\in \mathcal T _k$ 
 if $ k\ge 0 $ is the smallest integer such that 
\begin{equation}\label{e..v}  
|(2^kR_{\mathbf T})\cap v^{-1}({\boldsymbol \omega_{\mathbf T}})\cap G|\geq 
      2^{20k/\kappa^2-1}\delta \abs{R _{\mathbf T}}\,.
\end{equation}
 Then since the density of $s(\mathbf  T)$ for every tree $\mathbf  T\in\mathcal T$ 
is at least $\delta$, we have $\mathcal T=\bigcup_{k=0}^{\infty}\mathcal T_k$.
We can apply Conjecture~\ref{j.laceyli1} to these collections, with the value of $ \delta $ in that 
Conjecture being $  2^{20k/\kappa^2-1}\delta$.

For each $\mathcal T_k$, we decompose it by the following algorithm. 
Initialize 
$$\mathcal T_k^{\rm selected}\leftarrow\emptyset,\,\,\,\,\,
  \mathcal T^{\rm stock}_k \leftarrow\mathcal T_k\,. $$
While $\mathcal T_k^{\rm stock}\neq \emptyset$, select $\mathbf  T\in \mathcal T^{\rm stock}_k $
such that ${\rm scl}(s(\mathbf  T))$ is minimal. Define $\mathcal T_k(\mathbf  T)$ by
$$ 
\mathcal T_k(\mathbf  T)= \{\mathbf  T'\in \mathcal T_k\,:\, (2^kR_{\mathbf T})\cap(2^kR_{\mathbf  T'})\neq \emptyset\,\, 
 {\rm and}\,\, {\boldsymbol \omega_{\mathbf T}}\subset {\boldsymbol \omega_{\mathbf  T'}}\}\,. 
$$
Update 
$$
\mathcal T_k^{\rm selected}\leftarrow \mathcal T_k^{\rm selected}\cup \{\mathbf  T\}\, ,
\qquad 
\mathcal T^{\rm stock}_k \leftarrow \mathcal T^{\rm stock}_k \backslash \mathcal T_k(\mathbf  T)\,.
$$
Thus we decompose $\mathcal T_k$ into 
$$
\mathcal T_k = \bigcup_{\mathbf  T\in \mathcal T_k^{\rm selected}}\bigcup_{\mathbf  T'\in\mathcal T_k(\mathbf  T) }
\{\mathbf  T'\}\,.
$$
And
$$
\sum_{\mathbf  T\in\mathcal T_k} {\abs {R_{\mathbf T}}} =
            \sum_{\mathbf  T\in \mathcal T_k^{\rm selected}}\sum_{\mathbf  T'\in \mathcal T_k(\mathbf  T)}  
                  {\abs {R_{\mathbf  T'}}}\,.
$$                  
Notice that $R_{\mathbf  T'}$'s are disjoint for all $\mathbf  T'\in\mathcal T_k(\mathbf  T)$ and 
they are contained in $5(2^kR_{\mathbf  T})$. 
This is so, since the tops of the trees are assumed to be incomparable with respect to the 
order relation `$ \lesssim $' on tiles. 

Thus we have  
\begin{align*}
\sum_{\mathbf  T\in\mathcal T_k} {\abs {R_{\mathbf T}}} 
  &\lesssim  \sum_{\mathbf  T\in \mathcal T_k^{\rm selected}}2^{2k}|R_{\mathbf T}| \\
  &\lesssim \delta^{-1}2^{-10k/\kappa^2}\sum_{\mathbf  T\in  \mathcal T_k^{\rm selected} } 
          \abs  {(2^kR_{\mathbf T})\cap v^{-1}({\boldsymbol \omega_{\mathbf T})\cap G}}\,. 
\end{align*}
Observe that $(2^kR_{\mathbf T})\cap v^{-1}({\boldsymbol \omega}_{\mathbf T})$'s are disjoint 
for all $\mathbf  T\in \mathcal T_k^{\rm selected}$. This and the fact that $\abs G\le{}1$ proves 
\eqref{e.last}.  To argue for \eqref{e.middle},   we  see that 
\begin{align*}
\sum_{\mathbf  T\in\mathcal T_k} {\abs {R_{\mathbf T}}} &\lesssim 
         \delta^{-1}2^{-10k/\kappa^2}\bigg|
          \bigcup_{\mathbf  T\in  \mathcal T_k^{\rm selected} } 
          {(2^kR_{\mathbf T})\cap v^{-1}({\boldsymbol \omega}_{\mathbf T})\cap G}\bigg|
\\
 & \lesssim    \delta^{-1}2^{-10k/\kappa^2}\bigg|
          \bigcup_{\mathbf  T\in  \mathcal T_k} 
          {(2^kR_{\mathbf T})\cap G}\bigg|\,.
\end{align*} 
At this point, Conjecture~\ref{j.laceyli1} enters.  Observe that we can estimate 
\begin{equation}\label{unionest}
\begin{split}
\ABs{ \bigcup_{\mathbf  T\in \mathcal T _k} 2 ^{k} R_{\mathbf T}}
& \lesssim 
\abs{ \{ \operatorname M _{ \delta ',v, (\sigma^{\kappa/10}\Prm) ^{-1} } \mathbf 1_{F} > 
\sigma^{1+\kappa/4 N} \} } 
\\
& \lesssim  (\delta')^{-Np}   \sigma^{-p(1+\kappa/4 N)}\abs F. 
\\
& \lesssim  (\delta)^{-Np}  2 ^{-k}  \sigma^{-p(1+\kappa/4 N)}\abs F. 
\end{split}
\end{equation}
Here, $ \delta '= 2 ^{20 k/ \kappa ^2 -1} \delta $, the choice of $ \delta '$ permitted 
to us by \eqref{e..v}, and we have used \eqref{e.a3} 
in the first line, to pass to the Lipschitz Kakeya Maximal Function.

Hence, 
\begin{align*}
\sum_{\mathbf  T\in\mathcal T} {\abs {\sh{\mathbf T}}} &\lesssim\sum_{k=0}^{\infty}
               \sum_{\mathbf  T\in\mathcal T_k} {\abs {R_{\mathbf T}}} \\
  &\lesssim \delta^{-1}\sum_{k=0}^{\infty}  2^{-10k/\kappa^2}      
           \bigg| \bigcup_{\mathbf  T\in \mathcal T_k}(2^kR_{\mathbf T})\cap G  \bigg|
\\
  &\lesssim \delta^{-1}\!\sum_{k\,:\,1\leq 2^k\leq \sigma^{-\kappa/10}}
                       \!2^{-k}
                \bigg| \bigcup_{\mathbf  T\in \mathcal T_k}(2^kR_{\mathbf T})\bigg|
 \\& \qquad +
                \delta^{-1}\!\sum_{k\,:\,2^k>\sigma^{-\kappa/10}}\!2^{-10k/\kappa^2}|G|\,.
\end{align*}
On the first sum in the last line, we use \eqref{unionest}, and on the second, 
we just sum the geometric series, and recall that $ \lvert  G\rvert =1$. 

\end{proof}

We can now begin the principal line of reasoning for the proof of Lemma~\ref{l.dense+size}.

\medskip

\subsubsection*{The Construction of $\mathbf  S_{\text{\rm large}}^1$.}  
We use an orthogonality, or $\operatorname T \operatorname T^\ast$ argument that has been used many times before, especially in 
\cite{laceythiele} and \cite{laceyli1}.  (There is a feature of the current application of the argument 
that is present due to the fact that we are working on the plane, and it is detailed by Lacey and Li \cite{laceyli1}.)
 
We may assume that all intervals ${\boldsymbol \omega}_{s}$ are contained in the upper half of the unit circle in the plane. 
Fix $\mathbf  S\subset\mathcal A_v$, and  $\sigma=\size{\mathbf S} $.

We construct a collection of trees $\mathcal T_{\text{\rm large}}^1$ for the collection 
$\mathbf  S_1$, and a corresponding collection of 
$1$--trees $\mathcal T_{\text{\rm large}}^{1,1}$,  with particular properties. 
We begin the recursion by  initializing 
\begin{gather*}
\mathcal T_{\text{\rm large}}^1\leftarrow\emptyset,\qquad \mathcal T_{\text{\rm large}}^{1,1}\leftarrow\emptyset, \\
\mathbf  S_{\text{\rm large}}^1\leftarrow\emptyset,\qquad  \mathbf  S^{\text{\rm stock}}\leftarrow\mathbf  S_1 .
\end{gather*}
In the recursive step, if 
$\size {\mathbf  S^{\text{\rm stock}}}<\frac12\sigma^{1+\kappa/100}$, 
then this  recursion stops.
Otherwise, we select a tree $\mathbf  T\subset\mathbf  S^{\text{\rm stock}}$ 
such that three conditions are met. 
\begin{description}
\item 
[a] The top of the tree $s(\mathbf  T)$ 
 (which need not be in the tree) satisfies $\dense {s(\mathbf  T)}\ge\delta/4$. 
 
\item [b]   $\mathbf  T$  contains a $ 1$--tree $ \mathbf T^1$ with 
\begin{equation}  \label{e.756}
\dashint _{R _{\mathbf T}} \Delta (\mathbf T^1)\; dx \ge \tfrac12\sigma^{1+\kappa/100}\,.
\end{equation}

\item [c]  And that ${\boldsymbol \omega}_{\mathbf T}$ is in the first place minimal and 
and in the second most clockwise among all possible choices of $\mathbf  T$. 
(Since all ${\boldsymbol \omega}_{s}$ are in the upper half of the unit circle, this condition 
can be fulfilled.)
\end{description}
  We take $\mathbf  T$ to be the maximal  tree in 
$\mathbf  S^{\text{\rm stock}}$  which satisfies these conditions.

We then update 
$$
\mathcal T_{\text{\rm large}}^1\leftarrow\{\mathbf  T\}\cup 
\mathcal T_{\text{\rm large}},\qquad \mathcal T_{\text{\rm large}}^{1,1}
\leftarrow\{\mathbf  T^1\}\cup \mathcal T_{\text{\rm large}}^{1,1},
$$
$$
 \mathbf  S^1_{\rm large}\leftarrow \mathbf  T \cup \mathbf  S^1_{\rm large} 
 \qquad \mathbf  S^{\text{\rm stock}}\leftarrow\mathbf  S^{\text{\rm stock}}-\mathbf  T.
$$

The recursion then repeats.   Once the recursion stops, we update 
\begin{equation*}
\mathbf  S_1\leftarrow\mathbf  S_{\text{\rm stock}}
\end{equation*}
It is this collection that we analyze in the next subsection. 

Note that it is a consequence of the recursion, and Remark~\ref{r.pseudo}, that 
the tops of the trees $ \{ s (\mathbf  T) \mid \mathbf T\in \mathcal T ^{1}
_{\textup{large}}\}$ are pairwise incomparable under $ \lesssim $.

\smallskip 
The bottom estimate of \eqref{e.large-1} is then immediate from the construction and \eqref{e.last}.  

First, we turn to the deduction of the first estimate of \eqref{e.large-1}.  
Let $\mathcal T_{\text{large}}^{1,(1)}$ be the set
$$
\mathcal T_{\text{large}}^{1,(1)}=\big\{\mathbf  T\in\mathcal T_{\text{large}}^{1}:\, \sum_{s\in\mathbf  T^1}\abs{\langle 
\operatorname S_{\Prm}\ind F, \beta_s  \rangle}^2 < \tfrac{1}{16}
  \sigma^{2+\kappa/50}|R_{\mathbf T}| \big\}\,.
$$
And let $\mathcal T_{\text{large}}^{1,(2)}$ be the set
$$ \mathcal T_{\text{large}}^{1,(2)}=
\big\{\mathbf  T\in\mathcal T_{\text{large}}^{1}:\, \sum_{s\in\mathbf  T^1}\abs{\langle 
\operatorname S_{\Prm}\ind F, \beta_s  \rangle}^2 \geq \tfrac{1}{16}
  \sigma^{2+\kappa/50}|R_{\mathbf T}| \big\}\,.
$$
In the inner products, we are taking $ \beta _s$, which is supported off of 
$ \gamma _s R_s$. 
Since $\mathbf  T\in\mathcal T_{\text{large}}^{1}$ satisfies 
\begin{equation}\label{e.charge1}
\dashint _{R _{\mathbf T}} \Delta (\mathbf T)\; dx\ge \tfrac12\sigma^{1+\kappa/100}\,,
\end{equation}
we have 
$$
\sum_{s\in\mathbf  T^1}\abs{\langle 
\operatorname S_{\Prm}\ind F, \alpha_s  \rangle}^2 \geq \tfrac{1}{4}
  \sigma^{2+\kappa/50}|R_{\mathbf T}|\,.
$$ 
Thus, if $\mathbf  T \in \mathcal T_{\text{large}}^{1,(1)}$, we have
$$
\sum_{s\in\mathbf  T^1}\abs{\langle 
\ind F, \varphi_s  \rangle}^2 \geq \tfrac{1}{8}
  \sigma^{2+\kappa/50}|R_{\mathbf T}|\,.
$$
The replacement of $\alpha_{s}$ by $\varphi_s$ in the inequality above 
is an important  point for us. That we can then drop the $\operatorname S_{\Prm}$ is immediate.

With this construction and observation, we claim that 
\begin{equation}  \label{e.upper-1}  
\sum_{\mathbf T\in \mathcal T_{\text{\rm large}}^{1,(1)}}\abs{R_{\mathbf T}}\lesssim{}
(\log 1 /\sigma ) ^2 
\sigma^{-2-\kappa/50}\abs F .
\end{equation}

\begin{proof}[Proof of \eqref{e.upper-1}.]  
This is a variant of the the argument for the `Size Lemma' in \cite{laceyli1}, and so we 
will not present all details.   Begin by making a further decomposition of the 
trees $ \mathbf T\in \mathcal T_{\text{\rm large}}^{1,(1)}$.  To each such tree, we have 
a $ 1$-tree $ \mathbf T ^{1}\subset \mathbf T$ which satisfies \eqref{e.756}.    
We decompose $ \mathbf T ^{1}$.  Set 
\begin{align*}
\mathbf T ^{1} (0) 
&= 
\Bigl\{ s\in \mathbf T ^{1} \mid  \frac {\lvert  \ip f ,\varphi _s ,\rvert } 
{\sqrt {\lvert  R_s\rvert }}
< \sigma ^{1+ \kappa /100} 
\Bigr\} \,, 
\\
\mathbf T ^{1} (j) 
&= 
\Bigl\{ s\in \mathbf T ^{1} \mid  4 ^{j-1} 
 \sigma ^{1+ \kappa /100}  \le 
\frac {\lvert  \ip f ,\varphi _s ,\rvert } 
{\sqrt {\lvert  R_s\rvert }}
< 4 ^{j} \sigma ^{1+ \kappa /100} 
\Bigr\} \,, 
\\& \qquad  1 \le  j \le  j_0=C \log 1/ \sigma  \,.  
\end{align*}

Now, set $ \mathcal T (j)$ to be those $ \mathbf T\in \mathcal T_{\text{\rm large}}^{1,(1)} $
for which 
\begin{equation} \label{e..F}
 \sum _{s\in \mathbf T ^{1} (j)} \lvert  \ip f, \varphi _s, \rvert ^2 
\ge (2j_0) ^{-1}  \sigma ^{2+ \kappa /50}  \lvert  R _{\mathbf T}\rvert \,, \qquad 0\le j\le j_0\,. 
\end{equation}
It is the case that each $ \mathbf T \in \mathcal T_{\text{\rm large}}^{1,(1)}$ 
is in some $ \mathcal T (j)$, for $ 0\le j\le j_0$.  

The central case is that of $ j=0$.  We can apply the `Size Lemma' of \cite{laceyli1} to deduce that 
\begin{align*}
\sum _{\mathbf T\in \mathcal T (0)} \lvert  R _{\mathbf T}\rvert 
& \le (2 j_0)  \sigma ^{-2-\kappa /50} \sum _{\mathbf T\in \mathcal T (0)} 
\sum _{ s\in \mathbf T ^{1} (0)} \lvert  \ip f ,\varphi _s ,\rvert ^2 
\\
&  \lesssim   (\log 1/ \sigma ) \sigma ^{-2-\kappa /50} \lvert  F\rvert\,.  
\end{align*}
The point here is that to apply the argument in the `Size Lemma' one needs 
an average case estimate, namely \eqref{e..F}, as well as a uniform control, namely 
the  condition defining $ \mathbf T ^{1} (0)$.   
This proves \eqref{e.upper-1} in this case.  

For $ 1\le j\le j_0$, we can apply 
the `Size Lemma' argument to the individual tiles in the collection 
\begin{equation*}
\bigcup \{ \mathbf T ^{1} (j) \mid T \in \mathcal T (j)\}\,. 
\end{equation*}
The individual tiles satisfy the definition of a $ 1$-tree.  And the defining 
condition of $ \mathbf T ^1 (j)$ is both the average case estimate, and the uniform 
control needed to run that argument.  In this case we conclude that 
\begin{align*}
 \sum _{\mathbf T \in \mathcal T (j)} 
\sum _{ s\in \mathbf T ^{1} (j)} \lvert  \ip f ,\varphi _s ,\rvert ^2  
\lesssim \lvert  F\rvert \,.  
\end{align*}
Thus, we can estimate 
\begin{equation*}
\sum _{\mathbf T\in \mathcal T (j)} \lvert  R _ {\mathbf T}\rvert \lesssim 
(\log 1/ \sigma ) \sigma ^{ -1 - \kappa /50 } \lvert  F\rvert \,. 
\end{equation*}
This summed over $ 1\le j \le j_0= C \log 1/ \sigma $ proves \eqref{e.upper-1}. 
\end{proof}

For $\mathcal T_{\text{\rm large}}^{1,(2)}$, we have 
\begin{align*}
\sum_{T\in \mathcal T_{\text{\rm large}}^{1,(2)}}\abs{R_{\mathbf T}} &\lesssim 
       \sigma^{-2-\kappa/50}\sum_{{\mathsf {scl}} \geq \kappa^{-1}\|v\|_{\rm Lip}}
        \sum_{s: {\rm scl}(s)={\mathsf {scl}}} \abs{\langle 
         \operatorname S_{\Prm}\ind F, \beta_s  \rangle}^2\\
  &\lesssim  \sigma^{-2-\kappa/50} |F|\sum_{{\mathsf {scl}} \geq \kappa^{-1}\|v\|_{\rm Lip}}
            \bigg(\frac{\|v\|_{\rm Lip}}{{\mathsf {scl}}}\bigg)^{100}\\
  &\lesssim  \sigma^{-2-\kappa/50} |F|\,,  
\end{align*}
since $\beta_s$ has fast decay. 
The Bessel inequality in the last display
can be obtained by using the same argument in the proof of Lemma~\ref{l.bessel}. 
Hence we get 
\begin{equation}  \label{e.upper-2}  
\sum_{T\in \mathcal T_{\text{\rm large}}^{1,(2)}}\abs{R_{\mathbf T}}\lesssim{}\sigma^{-2-\kappa/50}\abs F .
\end{equation}
Combining \eqref{e.upper-1} and \eqref{e.upper-2}, we obtain the first estimate 
of \eqref{e.large-1}.    

\bigskip

Second, we turn to  the deduction of the middle estimate of \eqref{e.large-1}, which 
relies upon the Lipschitz Kakeya Maximal Function. 
Let $\mathcal T_{\text{large}}^{1,\text{good}}$ be the set
$$
\big\{\mathbf  T\in\mathcal T_{\text{large}}^{1}:\, \gamma_{\mathbf  T}\geq \kappa^{-1/2}\sigma^{-\kappa/5N}
 \big\}\,.
$$
And let $\mathcal T_{\text{large}}^{1,\text{bad}}$ be the set
$$
\big\{\mathbf  T\in\mathcal T_{\text{large}}^{1}:\, \gamma_{\mathbf  T} < \kappa^{-1/2}\sigma^{-\kappa/5N}
 \big\}\,.
$$
The `good' collection can be controlled by facts which we have already marshaled together. 
In particular, we have been careful to arrange the construction so that Lemma~\ref{l.add1} 
applies.   
By the main conclusion of that Lemma, \eqref{e.middle}, we have 
\begin{equation}\label{goodest}
\sum_{\mathbf  T \in \mathcal T_{\text{large}}^{1,\text{good}} }\abs{R_{\mathbf T}} \lesssim 
     \delta^{-M}\sigma^{-1-3\kappa/4}{\abs{F}} + \sigma^{1/\kappa}\delta^{-1}\,.
\end{equation}
Here, $ M$ is a large constant that only depends upon $ N$ in Conjecture~\ref{j.laceyli2}.

For $\mathbf  T\in\mathcal T_{\text{large}}^{1,\text{bad}}$, there are at most 
$K=O(\log(\sigma^{-\kappa}))$ many possible scales for ${\rm scl}(\omega_{\mathbf T}\times 
R_{\mathbf T})$. Let ${\rm scl}(\mathbf  T)= {\rm scl}(\omega_{\mathbf T}\times R_{\mathbf T})$. Thus we have 
$$
\sum_{\mathbf  T\in\mathcal T_{\text{large}}^{1,\text{bad}}}\abs{R_{\mathbf T}} \lesssim 
 \sum_{m=0}^{K} \;
\sum_{\mathbf  T: {\rm scl}(\mathbf  T)=2^m\kappa^{-1}\|v\|_{\rm Lip}}|R_{\mathbf T}|\,.
$$ 
Since $\mathbf  T$ satisfies \eqref{e.charge1},  we have 
$$
{\abs{F\cap \gamma_{\mathbf T} R_{\mathbf T}}}\gtrsim \sigma^{1+\kappa/2}{\abs{R_{\mathbf T}}}\,.
$$ 
Thus, we get
$$
\sum_{\mathbf  T\in\mathcal T_{\text{large}}^{1,\text{bad}}}\abs{R_{\mathbf T}} \lesssim 
         \sigma^{-1-\kappa/2}\sum_{m=0}^{K}
         \int_F \sum_{\mathbf  T: {\rm scl}(\mathbf  T)=2^m\kappa^{-1}\|v\|_{\rm Lip}}
           \ind {\sigma ^{- \kappa } R_{\mathbf T}}(x)dx\,.
$$
For the tiles with a fixed scale, we have the following inequality, 
which is a consequence of Lemma~\ref{l.fixedscale}.
$$
\NOrm  \sum_{\mathbf  T: {\rm scl}(\mathbf  T)=2^m\kappa^{-1}\|v\|_{\rm Lip}}
               \ind { \sigma ^{- \kappa } R_{\mathbf T}}        .\infty.
 \lesssim \sigma^{-\kappa/5}\delta^{-1}\,.  
$$ 
Hence we obtain 
\begin{equation}\label{badest}
\sum_{\mathbf  T \in \mathcal T_{\text{large}}^{1,\text{bad}} } \abs{R_{\mathbf T}}\lesssim 
     \delta^{-1}\sigma^{-1-3\kappa/4}{\abs{F}}\,.
\end{equation}

Combining \eqref{goodest} and \eqref{badest}, we obtain
the middle estimate of \eqref{e.large-1}. Therefore, we complete the proof
of \eqref{e.large-1}.

\bigskip


\subsubsection*{The Construction of $\mathbf  S_{\text{\rm large}}^2$.}
It is important to keep in mind that we have only removed trees of nearly 
maximal size, with tops of a given density.  In the collection of tiles that 
remain, there can be trees of large size, but they cannot have a top 
with nearly maximal density.  We repeat the $\operatorname T \operatorname  T^ \ast $ 
construction of the previous step in the proof, with two significant changes.

We construct a collection of trees $\mathcal T_{\text{\rm large}}^2$ from  the collection 
$\mathbf  S_1$, and a corresponding collection of 
$1$--trees $\mathcal T_{\text{\rm large}}^{2,1}$,  with particular properties. 
We begin the recursion by  initializing 
\begin{gather*}
\mathcal T_{\text{\rm large}}^2\leftarrow\emptyset\,,\qquad \mathcal T_{\text{\rm large}}^{2\,,1}\leftarrow\emptyset\,, \\
\mathbf  S_{\text{\rm large}}^2\leftarrow\emptyset\,,\qquad  \mathbf  S^{\text{\rm stock}}\leftarrow\mathbf  S_1 \,.
\end{gather*}
In the recursive step, if $\size {\mathbf  S^{\text{\rm stock}}}<\sigma/2$, then this  recursion stops.
Otherwise, we select a tree $\mathbf  T\subset\mathbf  S^{\text{\rm stock}}$ 
such that two conditions are met:  
\begin{description}
\item [a] $\mathbf  T$ satisfies $\|\Delta(\mathbf  T)\|_2\geq 
\frac{\sigma}{2}|R_{\mathbf T}|^{1/2} $. 

\item [b]  
${\boldsymbol \omega}_{\mathbf T}$ is both minimal and 
most clockwise among all possible choices of $\mathbf  T$. 
\end{description}
We take $\mathbf  T$ to be the maximal 
 tree in $\mathbf  S^{\text{\rm stock}}$  
which satisfies these conditions. 
We take $\mathbf  T^1\subset\mathbf  T$ to be a $ 1$--tree so that 
\begin{equation}\label{e.charge2}
\dashint _{R _{\mathbf T}} \Delta (\mathbf T^1)\; dx \ge \kappa \sigma \,. 
\end{equation}
This last inequality must hold by Lemma~\ref{l.jn}.

We then update 
\begin{gather*}
\mathcal T_{\text{\rm large}}^2\leftarrow \{\mathbf  T\}\cup \mathcal T_{\text{\rm large}},
\qquad \mathcal T_{\text{\rm large}}^{2,1}\leftarrow \{\mathbf  T^1\}\cup \mathcal T_{\text{\rm large}}^{2,1},
\\ \mathbf  S^{\text{\rm stock}}\leftarrow \mathbf  S^{\text{\rm stock}}-\mathbf  T.
\end{gather*}
The recursion then repeats.

Once the recursion stops, it is clear that the size of $\mathbf  S^{\text{stock}}$ 
is at most $\sigma/2$, and so we take 
$\mathbf  S_{\text{small}}:=\mathbf  S_{\text{stock}}$.  

The estimate 
\begin{equation*}
\sum_{\mathbf  T\in\mathcal T_{\text{large}}^2}\abs{R_{\mathbf T}}
\lesssim{}\sigma^{-2}\abs F 
\end{equation*} 
then is a consequence of the $TT^*$ method, as indicated in the previous step of the proof. 
That is the first estimate claimed in 
\eqref{e.large-2}.

\medskip 

What is significant is the second estimate of \eqref{e.large-2}, 
which involves the density. The point to observe is this. 
Consider any tile $s$ of density
at least $\delta/2$.  
Let $\mathcal T_s$ be those trees $\mathbf  T\in\mathcal T_{\text{large}}^2$ with top $\omega _{s(\mathbf  T)}\supset \omega _{s} $ 
and $ R _{s (\mathbf T)} \subset K R _{s}$.
By the construction of $\mathbf  S^1_{\text{large}}$, we must have
\begin{equation*}
\dashint _{R _{\mathbf s}} \Delta (\mathbf T^1)\; dx \le  \sigma ^{1+ \kappa /100}\,, 
\end{equation*}
for the maximal $ 1$--tree $ T ^{1} $ contained in   $\bigcup_{\mathbf  T\in\mathcal T_s}\mathbf  T$.
But, in addition, the tops of the trees in $\mathcal T_{\text{large}}^2$ are pairwise incomparable with respect to 
the order relation `$ \lesssim $,'
hence we conclude that 
\begin{equation*}
\frac{\sigma^2}4\sum_{\mathbf  T\in\mathcal T_s}\abs{R_{\mathbf T}}\lesssim\sigma^{2+\kappa/50}\abs{R_s}.
\end{equation*}
Moreover, by the construction of $\mathbf  S_{\text{light}}$, for each 
$\mathbf  T\in\mathcal T_{\text{large}}^2$ we must be able to select some tile 
$s$ with density at least $\delta/2$ and $ \omega _{s(\mathbf  T)} \supset \omega _s$ 
and $ R _{s (\mathbf T)} \subset K R _{s}$.  

Thus, we let $\mathbf  S^*$ be the maximal tiles of density at least $\delta/2$. 
Then, the inequality \eqref{e.last} applies to this collection. 
And, therefore, 
\begin{equation*}
\sum_{\mathbf  T\in\mathcal T_{\text{large}}^2}\abs{R_{\mathbf T}}{}\le{}
	\sigma^{\kappa/50}\sum_{s\in\mathbf  S^*}\abs{R_s}\lesssim\sigma^{\kappa/50}\delta^{-1}.
\end{equation*}
This completes the proof of second estimate of \eqref{e.large-2}. 
\qed

\subsection*{The Estimates For a Single Tree} 

\subsubsection*{The Proof of Lemma~\ref{l.1-tree}.}
It is a routine matter to check that for any $1 $--tree we have 
\begin{equation*}
\sum _{s\in \mathbf  T } \abs { \ip  f, \varphi_s , }^2 \lesssim \norm f.2.^2. 
\end{equation*}
Indeed, there is a strengthening of this estimate relevant to our concerns here.  
Recalling the notation \eqref{e.zq}, we have 
\begin{equation}  \label{e..b}
\NOrm \Bigl[
\sum _{s\in \mathbf  T } \frac {\abs { \ip  f, \varphi_s , }^2} {\lvert  R_s\rvert }
\mathbf 1_{R_s} 
\Bigr] ^{1/2} .p. 
\lesssim{} \norm \chi_{R_{\mathbf T}}^{(\infty)} f.p. \,, \qquad 1<p<\infty \,. 
\end{equation}
This is variant of the Littlewood-Paley inequalities, with some additional spatial  localization 
in the estimate.

Using this inequality for $ p=1 +\kappa /100$ and the assumption of the Lemma, we have 
\begin{align} \notag 
\sigma^ {1+ \kappa /100}&{}\le{}\Bigl[\dashint_{R_{\mathbf T}} \Delta_{\mathbf  T} \; dx\Bigr] ^{1+ \kappa /100}
\\&{}\le{}\dashint_{R_{\mathbf T}}\Delta_{\mathbf  T}^{1+ \kappa /100} \; dx
\notag 
\\ \notag 
&{}\le{}\abs{R_{\mathbf T}}^{-1} 
\NOrm \Bigl[
\sum _{s\in \mathbf  T } \frac {\abs { \ip  f, \varphi_s , }^2} {\lvert  R_s\rvert }
\mathbf 1_{R_s} 
\Bigr] ^{1/2} . {1+ \kappa /100}. ^{{1+ \kappa /100}} 
\\  \label{e..c}
& \lesssim \abs{R_{\mathbf T}}^{-1}\int_F \chi_{R_{\mathbf T}}^{(\infty)}\; dx.
\end{align}
This inequality can only hold if  $\abs{F\cap \sigma ^{-\kappa} 
R_{\mathbf T} }\ge{}\sigma ^{1+\kappa}\abs{R_{\mathbf T}} $.
\qed

\subsubsection*{The Proof of Lemma~\ref{l.tree}.}

This Lemma is closely related to the Tree Lemma of \cite{laceyli1}. 
Let us recall that result in a form that we need it.  
We need analogs of the definitions of density and size that do not incorporate truncations of the various 
functions involved.  Define 
\begin{equation*}
\operatorname {\sf dense}(s):=  \int_{G \cap v^{-1}({\boldsymbol \omega}_{s}) } \chi_{R_s} ^{(1)} (x) \; dx.  
\end{equation*}
 (Recall the notation from \eqref{e.zq}.)   
 \begin{equation*}
  \operatorname {\sf dense}( \mathbf  T):=\sup_{s\in\mathbf  T}  \operatorname {\sf dense}(s). 
 \end{equation*}
 Likewise define 
 \begin{equation*}
 \operatorname {\sf size}(\mathbf  T):=\sup_{\substack{ \mathbf  T'\subset \mathbf  T 
 \\ \text{$\mathbf  T' $ is a \tree 1 } } } 
 \Bigl[ \abs{R_ {\mathbf  T'} }^{-1 }\sum _{s\in\mathbf  T' }\abs{ \ip  \ind F ,  \varphi_s
 ,}^2 \Bigr] ^{1/2} 
 \end{equation*}
 Then, the proof of the  Tree Lemma of \cite{laceyli1} will give us this inequality:
 For $\mathbf  T $ a 
 tree,
 \begin{equation}  \label{e.tree-ll}
 \sum_{s\in\mathbf  T} \abs{\ip  \operatorname S_{\Prm} \ind F, \varphi_{s},\ip  \phi_s, \ind G ,}
\lesssim\operatorname {\sf dense} (\mathbf  T)  
\operatorname {\sf size} (\mathbf  T) \abs {R_{\mathbf T} }.
\end{equation}

Now,  consider a 
 tree $\mathbf  T $ with $\dense{\mathbf T} =\delta$, and $\size{\mathbf T} =\sigma$, where we insist upon using 
the original definitions of density and size.  If in addition, 
$\gamma_s\ge{}K (\sigma \delta)^{-1}  $ for all $s\in \mathbf  T $, we would then have the inequalities 
\begin{align*}
\operatorname {\sf dense }(\mathbf  T) \lesssim\delta, 
\\ 
\operatorname {\sf size }(\mathbf  T) \lesssim\sigma, 
\end{align*}
This places \eqref{e.tree-ll} at our disposal, but this is not quite the 
estimate we need, as the functions $\varphi_s $ and $\phi_s $ that occur in \eqref{e.tree-ll} are not truncated in the 
appropriate way, and it is this matter that we turn to next.  
Recall that
\begin{equation*}
\varphi_s=\alpha_s+\beta_s\,, 
\qquad \int \alpha _{s} (x-y v (x)) \psi _{ s} (y)\; dy= \alpha _{s-} (x)+ \alpha _{s+}
(y)\,.  
\end{equation*}
One should recall the displays \eqref{e.psi_s}, \eqref{e.apm}, and  \eqref{e.AaA}.

As an immediate consequence  of the definition of $\beta_s $, we have 
\begin{equation*}
\int_{\mathbb R^2} \abs{\beta_s(x)} \; dx \lesssim\gamma_s^{-2}\sqrt{\abs{R_s}}. 
\end{equation*}
Hence,  if we replace $\varphi_s $  by $\beta_s $, we have 
\begin{align*}
\sum_{s\in\mathbf  T} \abs{\ip  \operatorname S_{\Prm} \ind F, \beta_{s},\ip  \phi_s, \ind G ,}
& \lesssim \sum_{s\in\mathbf  T} \gamma_s^{-2}\sqrt{\abs{R_s}} \abs{ \ip  \phi_s, \ind G , } 
\\& \lesssim\sigma\delta{} \sum_{s\in\mathbf  T} \gamma_s^{-1}\abs{R_s} 
\\& \lesssim\sigma\delta\abs{R_{\mathbf T}}.
\end{align*}
And by a very similar argument, one sees corresponding bounds, in which we replace the $\phi_s $ by different functions. 
Namely, recalling the definitions of $a_{s\pm} $ in \eqref{e.apm} and estimate \eqref{e.AaA}, we have 
\begin{align} \label{e..AaA}
 \sum_{s\in\mathbf  T} \sqrt{\abs{R_s}}\abs{\ip  a_{s+}, \ind G ,}
& \lesssim \sigma  \sum_{s\in\mathbf  T}  
\sqrt {\lvert  R_s\rvert } \int _G \Bigl( \frac { \norm v. \textup{Lip}. } {\name {scl} s}  \Bigr)
^{10} \chi ^{(2)} _{R_s} (x) \; dx  
\\ \notag 
& \lesssim \sigma \delta  \sum_{s\in\mathbf  T}  
\Bigl( \frac { \norm v. \textup{Lip}. } {\name {scl} s}  \Bigr)
^{10} \lvert  R_s\rvert
\\ \notag 
 & \lesssim \sigma \delta \lvert  R _{\mathbf T}\rvert \,.  
\end{align}
Similarly, we have 
\begin{align*}
\sum_{s\in\mathbf  T} \sqrt{\abs{R_s}}\abs{\ip  \phi_s-a_{s+}-a_{s-}, \ind G ,} &\lesssim  \sigma \delta \abs{R_{\mathbf T}}, 
\end{align*}
Putting these estimates together proves our Lemma, in particular \eqref{e.tree}, under the assumption that 
$\gamma_s\ge{}K (\sigma \delta)^{-1}  $ for all $s\in \mathbf  T $. 

\medskip 

Assume that $\mathbf  T $ is a tree with $\scl s=\scl {s'} $ for all $s,s'\in\mathbf  T $. That is, the scale of the tiles in the tree is fixed. 
Then, $\mathbf  T $ is in particular a \tree1, so that by an application of the definitions and Cauchy--Schwartz, 
\begin{align*}
\sum_{s\in\mathbf  T} \abs{\ip  \operatorname S_{\Prm} \ind F, \alpha_{s},\ip  a _{s-}, \ind G ,} 
&{}\le{}\delta\sum_{s\in\mathbf  T} \abs{\ip  \operatorname S_{\Prm} \ind F, \alpha_s,} \sqrt{\abs{R_s}}
\\&{}\le{} \delta\sigma\abs{R_{\mathbf T} }.  
\end{align*}
But,  $\gamma_s\ge1 $ increases as does $\sqrt{\scl s}$.
Thus,  any tree $\mathbf  T $ with $\gamma_s\le{}K (\sigma \delta)^{-1}  $ for all $s\in \mathbf  T $, 
 is a union of $O(\abs{ \log \delta\sigma}) $ trees for which the last estimate holds.  
\qed


\chapter[Almost Orthogonality]{Almost Orthogonality Between Annuli}

\section*{Application of the Fourier Localization Lemma} 

We are to prove Lemma~\ref{l.modelsumed}, and in doing so rely upon 
a technical lemma on Fourier localization, 
Lemma~\ref{l.fl} below. 
We can take a choice of $1<\ensuremath{\alpha}<\frac98$, 
and assume, after a dilation, that 
$\norm v.C^\ensuremath{\alpha}.=1$.

The first inequality we establish is this.
\begin{lemma}\label{l.addC}  Using the notation of of Lemma~\ref{l.modelsumed}, 
and assuming that $ \norm v . C ^{\alpha }. \lesssim 1$, we have the estimate  
$\norm \mathcal C.2.\lesssim{}1$, where 
\begin{equation*}
\mathcal C=\sum_{\Prm\ge1}\mathcal C_\Prm
\end{equation*}
where the $\mathcal C_\Prm$ are defined in \eqref{e.Cj}.   
\end{lemma}

We have already established Lemma~\ref{l.model}, and so in particular know that
$\norm \mathcal C_\Prm.2. \lesssim1$.   Due to the imposition of the Fourier restriction 
in the definition of these operators, it is immediate that $\mathcal C_\Prm\mathcal C_{\Prm'} ^{\ast } 
\equiv 0$ for $ \Prm\neq \Prm'$.   We establish that 
\begin{equation} \label{e.cotlar}
\begin{split}
\norm \mathcal C_\Prm^*\mathcal C_{\Prm'}.2.
\lesssim
\max(\Prm,\Prm')^{-\delta}\,,
\\
\delta=\tfrac1{128}(\ensuremath{\alpha}-1)\,, \qquad \abs{ \log \Prm(\Prm') ^{-1} } >3 \,.
\end{split}
\end{equation}
Then, it is entirely elementary to see that $\mathcal C$ is a bounded operator. 
Let $\operatorname P_\Prm$ be the Fourier projection of 
$f$ onto the frequencies $\Prm<\abs{\ensuremath{\xi}}<2\Prm$.  Observe, 
\begin{align*}
\norm \mathcal C f .2.^2=& \NOrm \sum_{\Prm\ge1}\mathcal C_\Prm \operatorname P_\Prm f .2.^2
\\{}\le{}&{}\sum_{\Prm\ge1}\sum_{\Prm'>1}\ip \mathcal C_{\Prm}\operatorname P_\Prm f,\mathcal C_{\Prm'}\operatorname P_{\Prm'}f ,
\\ {}\le{}& 2\norm f.2.\sum_{\Prm\ge1} \sum_{\Prm' >1} \norm \mathcal C_\Prm^*\mathcal C_{\Prm'}\operatorname P_{\Prm'} f.2.
\\ &\lesssim \norm f.2.^2\Bigl(1+\sum_{\Prm\ge1} \sum_{\Prm' >1}  \max(\Prm,\Prm')^{-\delta} \Bigr)
\\ &\lesssim \norm f.2.^2. 
\end{align*}

\smallskip

There are only $O(\log \Prm)$ possible values of $\Scl$ that contribute to $\mathcal C_\Prm$,
and likewise for $\mathcal C_{\Prm'}$.  
Thus, if we define 
\begin{equation}\label{e.sclscl'}
\mathcal C _{\Prm,\Scl} f = 
\sum_{\substack{s\in\mathcal {AT}(\Prm)\\ \scl s=\Scl} }
\ip f,\varphi_s, \phi _s \,, 
\end{equation}
it suffices to prove 

\begin{lemma}\label{l.addC'}
 Using the notation of of Lemma~\ref{l.modelsumed}, 
and assuming that $ \norm v . C ^{\alpha }. \lesssim 1$, 
 we have 
\begin{equation*}
\norm \mathcal C _{\Prm,\Scl} ^{\ast } 
\mathcal C _{\Prm',\Scl'} .2. \lesssim (\max (\Prm, \Prm')) ^{-\delta }\,. 
\end{equation*}
Here, we can take $ \delta '=\tfrac 1 {100} (\alpha -1)$, and the inequality holds 
for all $ \lvert  \log \Prm (\Prm') ^{-1} \rvert>3 $, $ 1< \Scl \le \Prm$ and $ 1< \Scl' \le \Prm'$. 
\end{lemma}

\begin{proof}[Proof of Lemma~\ref{l.modelsumed}.]
In this proof, we assume that Lemma~\ref{l.addC} and Lemma~\ref{l.addC'} are 
established.  The first Lemma clearly establishes the first (and more important) 
claim of the Lemma.

Let us prove the inequality \eqref{e.scalefixed}.  Using the notation of this 
section, this inequality is as follows.
\begin{equation}  \label{e.Q}
\NOrm \sum _{\Prm=-\infty } ^{\infty } \mathcal C _{\Prm,\Scl} . 2. 
\lesssim{}(1+ \log( 1+{\mathsf {scl}}^{-1} \lVert v\rVert_{C^\alpha})).
\end{equation}
This inequality holds for all choices of $ C ^{\alpha }$ vector fields $ v$. 

Note that Lemma~\ref{l.addC'} implies immediately 
\begin{equation*}
\NOrm \sum _{\Prm=3 } ^{\infty } \mathcal C _{\Prm,\Scl} . 2. 
\lesssim 1 \,, \qquad \norm v. C ^{\alpha }. =1 \,. 
\end{equation*}
We are however in a scale invariant situation, so that this inequality implies 
this equivalent form, independent of assumption on the norm of the vector field. 
\begin{equation} \label{e.QQ}
\NOrm \sum _{\Prm\ge 8  \norm v. C ^{\alpha }. } ^{\infty } \mathcal C _{\Prm,\Scl} . 2. 
\lesssim 1 \,.
\end{equation}

On the other hand, Lemma~\ref{l.fixedscale}, implies that independent of any assumption 
other than measurability, we have have the inequality 
\begin{equation*}
\norm \mathcal C _{\Prm,\Scl} .2. \lesssim 1\,. 
\end{equation*}
To prove \eqref{e.Q}, use the inequality \eqref{e.QQ}, and this last inequality 
together with the simple fact that for a fixed value of $ \Scl$, there are at most 
$ \lesssim  1+ \log( 1+{\mathsf {scl}}^{-1} \lVert v\rVert_{C^\alpha})) $ values 
of $ \Prm$ with $ \Scl \le \Prm \le 8\lVert v\rVert_{C^\alpha} $. 

\end{proof}

We use the notation 
\begin{gather*}
\mathcal {AT}(\Prm,\Scl):=\{  s\in\mathcal {AT}(\Prm)\;:\; \scl s=\Scl\}\,,
\end{gather*}
Observe that as the scale is fixed,  we have a Bessel inequality for 
the functions $ \{\varphi _{s}\mid  s \in \mathcal {AT}(\Prm,\Scl) \}$.  Thus, 
\begin{align*}
\norm \mathcal C _{\Prm,\Scl} ^{\ast } 
\mathcal C _{\Prm',\Scl'} f.2. ^2 
& = \NOrm 
\sum _{s\in \mathcal {AT}(\Prm,\Scl)} 
\sum _{s\in \mathcal {AT}(\Prm',\Scl')} 
\ip \phi _s, \phi _{s'}, \ip \varphi _{s'}, f, \varphi _{s} .2. ^2 
\\
& \lesssim 
\sum _{s\in \mathcal {AT}(\Prm,\Scl)} \ABs{ 
\sum _{s\in \mathcal {AT}(\Prm',\Scl')} 
\ip \phi _s, \phi _{s'}, \ip \varphi _{s'}, f,   }^2\,.  
\end{align*}
At this point, the Schur test suggests itself, and indeed, we need a quantitative 
version of the test, which we state here.

\begin{proposition}\label{p.Schur} Let $ \operatorname A= \{a _{{i,j}}\}$ be a matrix acting on 
$ \ell ^2 (\mathbb N )$ by 
\begin{equation*}
\operatorname A x= \Bigl\{ \sum _{j} a _{i,j} x_j \Bigr\} \,. 
\end{equation*}
Then, we have the following bound on the operator norm of $ \operatorname A$. 
\begin{equation*}
\norm \operatorname  A . . ^2 
\lesssim 
\sup _{j} \sum _{i} \lvert a _{i,j}\rvert   
\cdot
\sup _{i} \sum _{j}  \lvert  a _{i,j}\rvert 
\end{equation*}
\end{proposition}

We assume   that $1\le\Prm<\tfrac 18\Prm'$.  For a subset $ \mathcal S
\subset  \mathcal {AT}(\Prm,\Scl) \times  \mathcal {AT}(\Prm',\Scl')$ 
Consider the operator and definitions below.
\begin{align*}
\operatorname A _{\mathcal S} f
&= \sum _{ (s,s')\in \mathcal S} 
 \ip \phi _s, \phi _{s'}, \ip \varphi _{s'}, f, \varphi _{s} \,,
\\
\operatorname {FL} (s, \mathcal S)
& = 
\sum _{s'\in  \mathcal {AT}(\Prm',\Scl')} \lvert  \ip \phi _{s}, \phi _{s'}, \rvert \,,  
\\
\operatorname {FL} ( \mathcal S)
& = \sup _{s} \operatorname {FL} (s, \mathcal S)\,.
\end{align*}
Here `$\operatorname {FL} $' is for 
`Fourier Localization' as this term is to be controlled by Lemma~\ref{l.fl}.   
We will use the notations $ \operatorname {FL} (s',\mathcal S)$, 
and $ \operatorname {FL}' (\mathcal S)$, 
which are defined similarly, 
with the roles of $ s$ and $ s'$ reversed.  By Proposition~\ref{p.Schur}, 
we have the inequality 
\begin{equation} \label{e.addU}
\norm \operatorname A _{\mathcal S} . 2.  ^2 
\lesssim 
\operatorname {FL} ( \mathcal S) 
\cdot 
\operatorname {FL}' ( \mathcal S)\,. 
\end{equation}
We shall see that typically $ \operatorname {FL} ( \mathcal S) $ 
will be somewhat large, but is balanced out by  $\operatorname {FL} '( \mathcal S)  $.

We partition $  \mathcal {AT}(\Prm,\Scl) \times  \mathcal {AT}(\Prm',\Scl')$
into three disjoint subcollections $ \mathcal S _{u}$, $ u=1,2,3$, defined as 
follows.   In this display, $ (s,s')\in \mathcal {AT}(\Prm,\Scl) \times  \mathcal
{AT}(\Prm',\Scl')$.
\begin{align}\label{e.S1}
\mathcal S_1 
&= \Bigl\{ (s,s')  
\mid \frac{\Scl'}{\Prm'}\ge\frac\Scl\Prm\Bigr\}\,,
\\ \label{e.S2}
\mathcal S_2
&= \Bigl\{ (s,s')  
\mid \frac{\Scl'}{\Prm'}<\frac\Scl\Prm\,, \  \Scl<\Scl'\Bigr\} \,, 
\\ \label{e.S3}
\mathcal S_3
&= \Bigl\{ (s,s')  
\mid \frac{\Scl'}{\Prm'}<\frac\Scl\Prm\,, \ \Scl'<\Scl\Bigr\} \,. 
\end{align}

A further modification to these collections must be made, but it is not 
of an essential nature.  
For an integer $ j\ge1$, and $ (s,s')\in \mathcal S _{u}$, for $ u=1,2,3$, 
write $  (s,s')\in \mathcal S _{u,j}$ if $ j$ is the smallest integer such that 
$ 2 ^{j+2} R_s \cap 2 ^{j+2} R _{s'}\neq \emptyset$.

We apply the inequalities (\ref{e.addU}) 
to the collections $ \mathcal S _{u,j}$, to prove the inequalities 
\begin{equation}\label{e.Suj}
\norm \operatorname A _{\mathcal S _{u,j}}. 2. 
\lesssim
2 ^{-j} (\Prm') ^{-\delta '}
\end{equation}
where $ \delta '=\tfrac 1 {100} (\alpha -1)$.  This proves Lemma~\ref{l.addC'}, 
and so completes the proof of Lemma~\ref{l.addC}.

In applying (\ref{e.addU}) it will be very easy 
to estimate  $ \operatorname {FL}(s, \mathcal S)$, with a term that decreases 
like say $ 2 ^{-10j}$.  The difficult part is to estimate either 
$ \operatorname {FL} (s, \mathcal S)$ or $ \operatorname {FL}' ( \mathcal S) $
by a term with decreases faster than a small power of $ (\Prm') ^{-1} $. 
for which we use Lemma~\ref{l.fl}.

Considering a term $ \ip \phi _{s}, \phi _{s'},$, the inner product is trivially 
zero if $ \boldsymbol \omega _{s}\cap \boldsymbol \omega _{s'}=\emptyset$.  We assume 
that this is not the case below. 
To apply Lemma~\ref{l.fl},  
fix $e\in{\boldsymbol \omega}_{s} '\cap{\boldsymbol \omega}_{s}{} $. 
Let \ensuremath{\alpha} be a Schwartz function on \ensuremath{\mathbb R} 
with $\widehat {\ensuremath{\alpha}}$ 
supported on $[\Prm',2\Prm']$, and identically one on $\frac34[\Prm',2\Prm']$. 
Set $\widehat{\beta} (\theta):= \widehat{\alpha}({\theta}-\frac32\Prm')$.  
We will convolve $\phi_s$ with \ensuremath{\beta } in the direction $e$, and $\phi_{s'}$
with \ensuremath{\alpha}  also in the direction 
$e$, thereby obtaining orthogonal functions.  
 
Define
\begin{gather} 
\label{e.Ie}
\operatorname I _{e} g (x) =
\int_ \ensuremath{\mathbb R} g(x-ye) \beta  (y)\; dy,
\\
\label{e.DD} 
\begin{split}
\Delta_{s}&=\ensuremath{\phi}_s- \operatorname I _{e} \phi _{s} 
\\
\Delta_{s'}&=\ensuremath{\phi}_{s'}- \operatorname I _{e} \phi _{s'}
\end{split}
\end{gather}
By construction, we have 
\begin{align*}
\ip \phi _{s}, \phi _{s'}, &= 
\ip \operatorname I _{e} \phi _{s}+\Delta _{s} , 
 \operatorname I _{e} \phi _{s'}+\Delta _{s'} , 
\\
&= 
\ip \operatorname I _{e} \phi _{s}, \Delta _{s'}, 
+\ip \Delta _{s},  \operatorname I _{e} \phi _{s'}  , 
+\ip \Delta_s,\Delta_{s'},\,.
\end{align*}
It falls to us to estimate terms like 
\begin{gather} \label{e.001}
\sup _{s'}
\sum_{s\in\mathbf S _{\ell ,j}}\abs{\ip \Delta_s, \operatorname I _{e}\phi_{s'},},
\\ \label{e.002}
\sup _{s'} \sum_{s\in\mathbf S _{\ell ,j}}\abs{\ip \operatorname I _{e}  \phi_s,\Delta_{s'},},
\\ \label{e.003}
\sup _{s'}\sum_{s\in\mathbf S _{\ell ,j}}\abs{\ip \Delta_s,\Delta_{s'},}.
\end{gather}
as well as the dual expressions, with the roles of $s $ and $ s'$ reversed. 

\medskip

The differences  $ \Delta _s$ and $ \Delta  _{s'}$ are frequently 
controlled by Lemma~\ref{l.fl}. 
Concerning application of this Lemma to $ \Delta _{s}$, observe that 
\begin{equation*}
\operatorname {Mod} _{-c (\omega _s)} \Delta _{s} 
= \operatorname {Mod} _{-c (\omega _s)} \phi _s - 
\int [ \operatorname {Mod} _{-c (\omega _s)} \phi _s (x-y e)] \widetilde \beta (y)\; dy
\end{equation*}
where $ \widetilde \beta (y)= \operatorname e ^{ (c (\omega _s) \cdot e)y} \beta (y)$. 
Now the Fourier transform of $ \beta $ is identically one in a neighborhood of the 
origin of width comparable to $ \Prm'$, where as $ \lvert  c (\omega _s) \cdot e\rvert $ 
is comparable to $ \Prm$.  
Since we can assume that $ \Prm'>\Prm+3$, say, the function $ \widetilde \beta $ 
meets the hypotheses of Lemma~\ref{l.fl}, namely it is Schwarz function with Fourier 
transform identically one in a neighborhood of the origin, and the width of that 
neighborhood is comparable to $ \Prm'$. 
And so $ \Delta _s$ is bounded by the 
bounded by the three terms in \eqref{e.tech1}---\eqref{e.tech3} below.  In 
these estimates, we  take $ 2 ^{k} \simeq \Prm' >1 $. 
By a similar argument, one sees that Lemma~\ref{l.fl} also applies to $ \Delta _{s'}$.

We will let 
$\Delta_{s,m}$, for $m=1,2,3$, denote the
terms that come from \eqref{e.tech1}, \eqref{e.tech2}, and \eqref{e.tech3} 
respectively.  We use the corresponding notation for $ \Delta_{s,m}$, for $ m=1,2,3$. 
A nice feature of these estimates, is that while $\Delta_{s}$
and $\Delta_{s'}$ depend 
upon the choice of $e\in{\boldsymbol \omega}_{s'}\cap {\boldsymbol \omega}_{s}{}$,
the upper bounds in the first two estimates do not depend upon the choice of $e$.
While the third estimate does, the dependence of the set $F_s$ on the choice of $e$ is rather weak.

In application of (\ref{e.tech2}), the functions $ \Delta _{s,2}$ will be very small, 
due to the term $ (\Prm') ^{-10}$ which is on the right in (\ref{e.tech2}).  
This term is so much smaller than all other terms involved in this argument that these 
terms are very easy to control.  So we do not explicitly discuss the case 
of $ \Delta _{s,2}$, or $ \Delta _{s',2}$ below.

In the analysis of the terms \eqref{e.001} and \eqref{e.002}, we frequently only 
need to use an inequality such as $ \lvert  \operatorname I _{e} \phi _{s'}\rvert 
\lesssim \chi ^{(2)} _{R _{s'}}$.  When it comes to the analysis of 
\eqref{e.003}, the function $ \Delta _{s'}$ obeys the same inequality, so that 
these sums can be controlled by the same analysis that controls \eqref{e.001}, 
or \eqref{e.002}.   So we will explicitly discuss these cases below. 

In order for $ \ip \phi _{s}, \phi _{s'}, \neq \emptyset $, 
we must necessarily have $ \boldsymbol \omega _{s}\cap  \boldsymbol \omega _{s'} \neq 
\emptyset $. 
Thus, we update all $ \mathcal S _{\ell ,j}$ as follows. 
\begin{equation*}
\mathcal S _{\ell ,j}
\leftarrow \{ (s,s')\in \mathcal S _{\ell ,j} \mid  \boldsymbol \omega _{s}\cap
\boldsymbol
\omega _{s'} \neq\emptyset \} \,. 
\end{equation*}

\bigskip

\subsection*{The Proof  of (\ref{e.Suj}) for $ \mathcal S _{1,j}$, $ j\ge 1$.} 
Recall the definition of $ \mathcal S _{1,j}$ from (\ref{e.S1}).  
In particular, for $ (s,s')\in \mathcal S _{1,j}$, we must have 
$ \boldsymbol  \omega _{s}\subset \boldsymbol  \omega _{s'}$. 

We will use the inequality \eqref{e.addU}, and show that for $ 0<\epsilon <1$, 
\begin{align}  \label{e.caseI-a} 
\operatorname{FL}(\mathcal S _{1,j})&
 \lesssim2^{-10j}   (\Prm')^{-\widetilde \alpha } 
 \sqrt {\frac {\Scl' \cdot \Prm'} {\Scl \cdot \Prm}}
\\ \label{e.caseI-b} 
\operatorname {FL}'(\mathcal S _{1,j}) &\lesssim2^{2j} (\Prm')^{\epsilon  } 
 \cdot  \sqrt {\frac {\Scl \cdot \Prm} {\Scl' \cdot \Prm'}} \,. 
\end{align}
Notice that in the second estimate, we permit some slow increase in the 
estimates as a function of $ 2 ^{j}$ and $ \Prm'$.  But, due to the 
form of the estimate of the Schur test in \eqref{e.addU}, this 
slow growth is acceptable.  

The  terms inside the square root in these two estimates 
cancel out.  These inequalities conclude the proof of the inequality 
\eqref{e.Suj} for the collection  $ \mathcal S _{1,j}$, $ j\ge 1$. 

\smallskip

We prove \eqref{e.caseI-a}.  For this, we use Lemma~\ref{l.fl}.  
That is, we should bound the several terms 
\begin{gather} 
\label{e.1case1}
 \sum _{s'\,:\, (s,s')\in \mathbf S_{1,j}}
 \lvert  \ip \Delta_{s}, \operatorname I _{e}\phi _{s'},\rvert  \,,
\\ \label{e.2case1}
 \sum _{s'\,:\, (s,s')\in \mathbf S_{1,j}}
\lvert  \ip \operatorname I _{e} \phi  _{s}, \Delta  _{s'},\rvert  \,,
\\ \label{e.3case1}
 \sum _{s'\,:\, (s,s')\in \mathbf S_{1,j}}
\lvert  \ip  \Delta_{s}, \Delta  _{s'},\rvert  \,.
\end{gather}
Here $\Delta _s $ and $ \Delta _{s'}$ are as in \eqref{e.DD}.
And, $ \operatorname I _{e}$ is defined as in \eqref{e.Ie}. We can 
regard the tile $s $ as fixed, and so fix a choice of  $ e\in \boldsymbol \omega _{s}$.  
In the next two cases, we will need to estimate the same expressions as above. 
In all three cases,  Lemma~\ref{l.fl} is applied with $ 2 ^{k} \simeq \Prm'$, 
and we can take  $ \epsilon $ in this Lemma to be $ \epsilon = \tfrac 1 {400} (\alpha -1)$. 
For ease of notation, we set 
\begin{equation} \label{e.tilde-a}
\widetilde \alpha = (\alpha -1)(1-\epsilon )  ^2- \epsilon >0 
\end{equation}

As we have already mentioned, we do not explicitly discuss the upper bound 
on the estimate for \eqref{e.3case1}.

\smallskip 

\subsubsection*{The Upper Bound on \eqref{e.1case1}}
We write $ \Delta _{s}=\Delta _{s,1}+\Delta _{s,2}+\Delta _{s,3}$, where 
these three terms are those on the right in \eqref{e.tech1}---\eqref{e.tech3} respectively. 
Note that 
\begin{equation} \label{e.I<}
\lvert  \operatorname I _{e} \phi _{s'}\rvert \lesssim \chi ^{(2) } _{R _{s'}}\,,  
\end{equation}
since $ \operatorname I _{e}$ is convolution in the long direction of $ R _{s'}$, at the scale 
of $ (\Prm') ^{-1} $, which is much smaller than the length of $ R _{s'}$ in the direction $
e$.  Therefore, we can estimate  the term in \eqref{e.1case1} by 
\begin{align} \notag 
 \sum _{s'\,:\, (s,s')\in \mathbf S_{1,j}} 
\lvert  \ip \Delta _{s,1},  \operatorname I _{e}\phi _{s'},\rvert 
& \lesssim  
(\Prm') ^{-\widetilde \alpha }  2 ^{-10j} 
\\ & \qquad  \notag
\int \chi ^{(2) } _{R_s} \Bigl\{ \sum _{s'\,:\, (s,s')\in \mathbf S_{1,j}} 
\chi ^{(2)} _{R _{s'}}\Bigr\} \; dx 
\\  \label{e.Vv}
& \lesssim (\Prm') ^{-\widetilde \alpha  }  2 ^{-10j} 
 \sqrt {\frac {\Scl' \cdot \Prm'} {\Scl \cdot \Prm}}\,. 
\end{align}
This is as required to prove \eqref{e.caseI-a} for these sums.

For the terms associated with $ \Delta _{s,3}$, we have 
\begin{align*}
\sum _{s'\,:\, (s,s')\in \mathbf S_{1,j}} 
\lvert  \ip \Delta _{s,3}, \operatorname I _{e}\phi _{s'},\rvert 
& \lesssim 
\sum _{s'\,:\, (s,s')\in \mathbf S_{1,j}} 
\int _{F_s} \lvert  R_s\rvert ^{-1/2} \cdot \chi ^{(2) } _{R_s'} \; dx 
\\
& \lesssim 2 ^{-10j} \lvert  F_s\rvert \sqrt{ \Prm' \cdot \Scl' \cdot \Prm \cdot \Scl}
\\
& \lesssim 2 ^{-10j} (\Prm') ^{-\alpha +\epsilon }
 \sqrt {\frac {\Scl' \cdot \Prm'} {\Scl \cdot \Prm}}\,. 
\end{align*}
That is, we only rely upon the estimate \eqref{e.Fs1}. 
This completes the analysis of \eqref{e.1case1}.  
(As we have commented above, we do not explicitly discuss the case of 
$ \Delta _{s,2}$.)

\smallskip

\subsubsection*{The Upper Bound for \eqref{e.2case1}.}
Since $ \boldsymbol  \omega _{s}\subset \boldsymbol  \omega _{s'}$, the only 
facts about $ \Delta _{s'}$ we need are 
\begin{equation}\label{e.789789}
\begin{split}
\int  _{ (\Prm') ^{\epsilon }R_s'} \lvert  \Delta _{s'}\rvert\; dx 
& \lesssim (\Prm') ^{-\widetilde \alpha +\epsilon } 
 \sqrt {\frac 1 {\Scl' \cdot \Prm'} }\,,
 \\
 \lvert  \Delta _{s'} (x)\rvert &\lesssim 
(\Prm') ^{-\widetilde \alpha } 
 \chi ^{(2)} _{R _{s'}} (x)\,, 
 \qquad x\not\in (\Prm') ^{\epsilon } R _{s'}. 
\end{split}
\end{equation}
Indeed, this estimate is a straightforward consequence of the 
various conclusions of Lemma~\ref{l.fl}. (We will return to this estimate 
in other cases below.)

These inequalities, with $ \lvert  \operatorname I _{e} \phi _{s}\rvert 
\lesssim \chi ^{(2)} _{R _{s}}$, permit us to estimate 
\begin{align*}
\eqref{e.2case1} 
& \lesssim 
2 ^{-20j}  \lvert  R_s\rvert ^{-1/2} 
 \sum _{s'\,:\, (s,s')\in \mathbf S_{1,j}} 
\int  _{ (\Prm') ^{\epsilon }R_s'} \lvert  \Delta _{s'}\rvert\; dx 
\\
& \lesssim 
2 ^{-20j} (\Prm') ^{-\widetilde \alpha }
   \sqrt {\frac  {\Scl \cdot \Prm} {\Scl' \cdot \Prm'}} 
 \times \sharp \{s'\,:\, (s,s')\in \mathbf S_{1,j}\}
 \\& 
\lesssim 2 ^{-20j} (\Prm') ^{-\widetilde \alpha }
 \sqrt {\frac {\Scl' \cdot \Prm'} {\Scl \cdot \Prm}}\,. 
\end{align*}
which is the required estimate. Here of course we use the estimate 
\begin{equation*}
 \sharp \{s'\,:\, (s,s')\in \mathbf S_{1,j}\}
 \lesssim 2 ^{2j} \frac {\Scl' \cdot \Prm'} {\Scl \cdot \Prm}\,. 
\end{equation*}

\bigskip

We now turn to the proof of \eqref{e.caseI-b}, where it is important that 
we justify the  small term 
\begin{equation*}
 \sqrt {\frac {\Scl \cdot \Prm} {\Scl' \cdot  \Prm'}}  
\end{equation*}
on the right in \eqref{e.caseI-b}.  
We  estimate the terms dual to \eqref{e.1case1}---\eqref{e.3case1}, namely 
\begin{gather} 
\label{e.1case4}
 \sum _{s\,:\, (s,s')\in \mathbf S_{1,j}}
 \lvert  \ip \Delta_{s}, \operatorname I _{e _{s}} \phi _{s'},\rvert  \,,
\\ \label{e.2case5}
 \sum _{s\,:\, (s,s')\in \mathbf S_{1,j}}
\lvert  \ip \operatorname I _{e _{s}} \phi  _{s}, \Delta  _{s'},\rvert  \,,
\\ \label{e.3case6}
 \sum _{s\,:\, (s,s')\in \mathbf S_{1,j}}
\lvert  \ip \Delta_{s}, \Delta  _{s'},\rvert  \,.
\end{gather}
Here, for each choice of tile $s $, we make a choice of $ e _{s}\in 
\boldsymbol \omega _{s}\subset \boldsymbol \omega _{s'}$.

\subsubsection*{The Upper Bound on \eqref{e.1case4}.}
We have an inequality analogous to \eqref{e.I<}. 
\begin{equation}\label{e.I<<}
\lvert  \operatorname I _{e _{s}} \phi _{s'}\rvert  
\lesssim \chi ^{(2)} _{R _{s'}}\,. 
\end{equation}
Note that as we can view $ s'$ as fixed, all the tiles 
$ \{s\,:\, (s,s')\in \mathbf S_{1,j}\}$ have the same 
approximate spatial location. 
Let us single out a tile $ s_0$ in this collection.  Then, for 
all $ s$, we have $ R_s \subset 2 ^{j+2} R _{s_0}$.

Recalling the specific information about the support of 
the functions of $ \Delta _{s}$ from \eqref{e.tech1},  \eqref{e.tech3} and \eqref{e.Fs2},
it follows that 
\begin{align*}
\sum _{s\,:\, (s,s')\in \mathbf S_{1,j}} \lvert  \Delta _{s}\rvert 
\lesssim 2 ^{2j} (\Prm') ^{\epsilon }
\chi ^{(2)} _{2 ^{j+2} R _{s_0}} \,. 
\end{align*}
In particular, we do not claim any decay in $ \Prm'$ in this estimate. 
(The small growth of $ (\Prm') ^{\epsilon }$ above arises from the overlapping 
supports of the functions  $ \Delta _{s}$, as detailed in Lemma~\ref{l.fl}.) 
Therefore, we can estimate 
\begin{align*}
\sum _{s\,:\, (s,s')\in \mathbf S_{1,j}}
 \lvert  \ip \Delta_{s}, \operatorname I _{e _{s}} \phi _{s'},\rvert 
 & \lesssim 2 ^{2j} (\Prm') ^{ \epsilon } 
 \int \chi ^{(2)} _{2 ^{j} R _{s_0}} \chi ^{(2)} _{R _{s'}} \; dx 
 \\
 & \lesssim 2 ^{-10j} (\Prm') ^{\epsilon } 
\sqrt {\frac {\Scl \cdot \Prm} {\Scl' \cdot  \Prm'}}  \,. 
 \end{align*}
This is as required in \eqref{e.caseI-b}.

\begin{remark}\label{r......} It is the analysis of the term 
\begin{equation*}
\sum _{s\,:\, (s,s')\in \mathbf S_{1,j}}
 \lvert  \ip \Delta_{s,3}, \operatorname I _{e _{s}} \phi _{s'},\rvert 
\end{equation*}
which prevents us from obtaining a decay in $ \Prm'$, at least in some 
choices of the parameters $ \Scl\,, \Prm\,, \Scl'$, and $ \Prm'$. 
\end{remark}

\subsubsection*{The Upper Bound on \eqref{e.2case5}.}
The fact about $ \Delta _{s'}$ we need is the 
simple inequality  $ \lvert  \Delta _{s'}\rvert \lesssim \chi ^{(2)} _{R _{s'}} $. 

As in the previous case, we turn to the fact that all the 
tile $ \{s\,:\, (s,s')\in \mathbf S_{1,j}\}$ have the same 
approximate spatial location. 
Single out a tile $ s_0$ in this collection, so that 
$ R_s \subset 2 ^{j+2} R _{s_0}$ for all such $ s$. 

Our claim is that 
\begin{equation} \label{e.t67}
\sum _{s\,:\, (s,s')\in \mathbf S_{1,j}}
 \lvert  \operatorname I _{e_s} \phi _{s} \rvert
\lesssim  2 ^{2j} \chi ^{(2)} _{2 ^{2j} R _{s}} \,. 
\end{equation}
(We will have need of related inequalities below.)
Suppose that $ s \in \{s\,:\, (s,s')\in \mathbf S_{1,j}\}$. 
These intervals all have the same length, namely $ \Scl/\Prm$.
And $ x\not\in \operatorname {supp} (\phi _s)$ implies $ v (x)\not\in \boldsymbol  \omega _{s}$, so
that by the Lipschitz assumption on the vector field 
\begin{equation*}
\operatorname {dist} ( x, \operatorname {supp} (\phi _s)) 
 \gtrsim \operatorname {dist} (v (x), \boldsymbol  \omega _{s})
 \,. 
\end{equation*}
This means that 
\begin{equation}\label{e.dist}
\lvert  \operatorname I _{e _{ s}} \phi _{ s} (x)\rvert
\lesssim  
\chi ^{(2)} _{R _{ s}}  \Bigl( 1+ 
\Prm'  \cdot \operatorname {dist} (v(x), \boldsymbol  \omega _{s})
\Bigr) ^{-10}\,. 
\end{equation}
Here, we recall that the operator  $ \operatorname I _{e}$ is dominated by 
the operator which averages on 
spatial scale $ (\Prm') ^{-1} $ in the direction $ e$. Moreover, we have 
\begin{equation} \label{e.WWW}
\Prm'  \cdot \operatorname {dist} (\boldsymbol  \omega _{s},  \overline{\boldsymbol \omega} )
\gtrsim \Scl\,. 
\end{equation}
Here, we partition the unit circle into 
disjoint intervals $ \overline{\boldsymbol \omega} \in 
\overline \Omega $ of length $ \lvert  \overline{\boldsymbol \omega}\rvert 
\simeq \Scl/\Prm$, so that for all $ s\in \{s\,:\, (s,s')\in \mathbf S_{1,j}\}$, 
we have $ \boldsymbol  \omega _{s}\in \overline \Omega $.

In fact, the  term on the left   in \eqref{e.WWW} can be taken to be integer multiples of 
$ \Scl$. 
Combining these observations proves \eqref{e.t67}. 
Indeed, we can estimate the term in \eqref{e.t67} as follows.  
For $ x$, fix $ \overline{\boldsymbol \omega}\in \overline \Omega $ with $ v (x)\in 
\overline{\boldsymbol \omega}$.  Then,  
 \begin{align*}
\sum _{s\,:\, (s,s')\in \mathbf S_{1,j}}
 \lvert  \operatorname I _{e_s} \phi _{s} \rvert
& \lesssim 
\sum _{s\,:\, (s,s')\in \mathbf S_{1,j}}
\chi ^{(2)} _{R _{\overline s}}  (x) \Bigl( 1+ 
\Prm'  \cdot \operatorname {dist} ( \overline{\boldsymbol \omega}, \boldsymbol  \omega _{s})
\Bigr) ^{-10}\,. 
\end{align*}
The important point is that the term involving the distance allows us to 
sum over the possible values of $ \boldsymbol  \omega _{s}\subset \boldsymbol  \omega _{s'}$
to conclude \eqref{e.t67}. 

To finish this case, we can estimate 
\begin{align*}
\sum _{s\,:\, (s,s')\in \mathbf S_{1,j}}
 \lvert  \ip  \operatorname I _{e_s } \phi _{s}, \Delta _{s'},\rvert 
 &   
\lesssim 2 ^{-10j} 
\sqrt {\frac {\Scl \cdot \Prm} {\Scl' \cdot  \Prm'}}  \,. 
\end{align*}
This completes the upper bound on \eqref{e.2case5}.

\begin{figure}
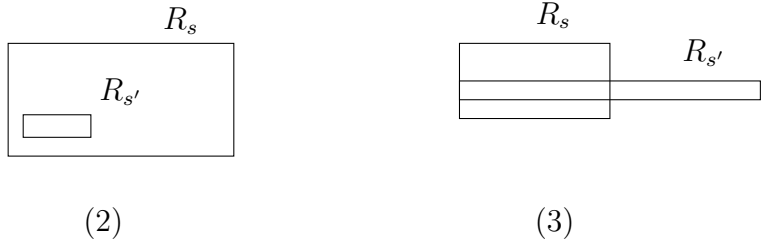
 \label{f.cases}  \centering 
\begin{pgfpicture}{0cm}{0cm}{10cm}{3cm}
%
%
\pgfrect[stroke]{\pgfxy(1,1)}{\pgfxy(3,1.5)} 
\pgfrect[stroke]{\pgfxy(1.2,1.25)}{\pgfxy(.9,.3)} 
\pgfputat{\pgfxy(3.3,2.8)}{\pgfbox[center,center]{$R_s$}}
\pgfputat{\pgfxy(2.5,1.85)}{\pgfbox[center,center]{$R_{s'}$}}
\pgfputat{\pgfxy(2.0,0)}{\pgfbox[base,left]{(2)}}
\pgfrect[stroke]{\pgfxy(7,1.5)}{\pgfxy(2,1)} 
\pgfrect[stroke]{\pgfxy(7,1.75)}{\pgfxy(4,.25)} 
\pgfputat{\pgfxy(10.25,2.2)}{\pgfbox[center,bottom]{$R_{s'}$}}
\pgfputat{\pgfxy(8.25,2.7)}{\pgfbox[center,bottom]{$R_{s}$}}
\pgfputat{\pgfxy(8,0)}{\pgfbox[base,left]{(3)}}
\end{pgfpicture}\caption{The relative positions of $R_s$ and $R_{s'}$ in 
for pairs $ (s,s')\in \mathcal S _{\ell }$, for $ \ell =2$ and $ \ell =3$ respectively.} 
\end{figure}

\subsection*{The Proof (\ref{e.Suj}) for $ \mathcal S _{2,j}$, $ j\ge 1$.} 
In this case, note that the assumptions imply that we can 
assume that ${\boldsymbol \omega}_{s'}\subset{\boldsymbol \omega}_{s}$, and that 
dimensions of the rectangle $R_{s'}$ are smaller than those for $R_s$ in both directions. 
See Figure~\ref{f.cases}.

We should  show these two inequalities, in analogy to \eqref{e.caseI-a} and 
\eqref{e.caseI-b}. 
\begin{align}  \label{e.caseII-a} 
\operatorname{FL}(\mathcal S _{2,j})&
 \lesssim2^{-10j}   (\Prm')^{-\widetilde \alpha } 
 \sqrt {\frac {\Scl' \cdot \Prm'} {\Scl \cdot \Prm}}
\\ \label{e.caseII-b} 
\operatorname {FL}'(\mathcal S _{2,j}) &\lesssim2^{-10j} 
 (\Prm')^{-\widetilde \alpha } 
 \cdot  \sqrt {\frac {\Scl \cdot \Prm} {\Scl' \cdot \Prm'}} \,. 
\end{align}
Here, $ \widetilde \alpha $ is as in \eqref{e.tilde-a}.

For the proof of \eqref{e.caseII-a}, we should analyze the sums 
\begin{gather} 
\label{e.1case2}
 \sum _{s'\,:\, (s,s')\in \mathbf S_{2,j}}
 \lvert  \ip \Delta_{s}, \operatorname I _{e_{s'}}\phi _{s'},\rvert  \,,
\\ \label{e.2case2}
 \sum _{s'\,:\, (s,s')\in \mathbf S_{2,j}}
\lvert  \ip \operatorname I _{e_{s'}} \phi  _{s}, \Delta  _{s'},\rvert  \,,
\\ \label{e.3case2}
 \sum _{s'\,:\, (s,s')\in \mathbf S_{2,j}}
\lvert  \ip  \Delta_{s}, \Delta  _{s'},\rvert  \,.
\end{gather}
These inequalities are in analogy to \eqref{e.1case1}---\eqref{e.3case1}, 
and $ e _{s'}\in {\boldsymbol \omega}_{s'}\subset{\boldsymbol \omega}_{s}$.

\subsubsection*{The Upper Bound on \eqref{e.1case2}.}
Fix the tile $ s$.  Fix a translate $ \overline {R}_s$ of $ R_s$ with 
$ 2 ^{j} R_s\cap 2 ^{j} \overline {R}_s = \emptyset $, but 
$ 2 ^{j+1}R_s \cap 2 ^{j+2} \overline {R}_s\neq \emptyset $.  Let us consider 
\begin{equation} \label{e.OverLine}
\overline {\mathcal S} _{2,j} =\{ (s,s') \in \mathcal S _{2,j}\mid  R _{s'} 
\subset \overline {R}_s
\}
\end{equation}
and we restrict the the sum in \eqref{e.1case2} to this collection of tiles. 
Note that with $ \lesssim 2 ^{2j}$ choices of $ \overline {R}_s$, we can exhaust the 
collection $ \mathcal S _{2,j}$.  So we will prove a slightly stronger estimate 
in the parameter $ 2^j$ for the restricted collection  $ \overline {\mathcal S} _{2,j}$. 

The point of this restriction is that we can appeal to an inequality similar to 
\eqref{e.t67}.  Namely, 
\begin{equation} \label{e.T67}
\sum _{s'\,:\, (s,s')\in \overline {\mathcal S} _{2,j}}
 \lvert  \operatorname I _{e_s'} \phi _{s'} \rvert
\lesssim    
 \sqrt {\frac {\Scl' \cdot \Prm'} {\Scl \cdot \Prm}}
\chi ^{(2)} _{\overline {R}_s} \,. 
\end{equation}
Note that the term in the square root takes care of the differing $ L ^2 $ 
normalizations of $ \phi _{s'}$ and $ \chi ^{(2)} _{\overline {R}_s} $. 
Indeed, the proof of \eqref{e.t67} is easily modified to give this inequality.

Next, we observe that the analog of \eqref{e.789789} holds for $ \Delta _{s}$. 
Just replace $ s'$ in \eqref{e.789789} with $ s$. 
It is a consequence that we have 
\begin{equation*}
\sum _{s'\,:\, (s,s')\in \overline {\mathcal S} _{2,j}}
\lvert  \ip \Delta _{s}, \operatorname I _{e _{s'}} \phi _{s'}, \rvert 
\lesssim  2 ^{-12j}
(\Prm') ^{-\widetilde \alpha } 
 \sqrt {\frac {\Scl' \cdot \Prm'} {\Scl \cdot \Prm}}\,. 
\end{equation*}
This is enough to finish  this case.

\subsubsection*{The Upper Bound on \eqref{e.2case2}.}
Let us again appeal to the notations $ \overline {R}_s$ and $ \overline {\mathcal S} _{2,j}$ 
as in \eqref{e.OverLine}. 

We have the estimates 
\begin{gather*}
\lvert  \operatorname I _{e _{s'}} \phi _{s}\rvert  
\lesssim \chi ^{(2)} _{R _{s}}\,. 
\end{gather*}
As for the sum over $ \Delta _{s'}$, we have an analog of the estimates 
\eqref{e.789789}.  Namely, 
\begin{align*}
\sum _{s'\,:\, (s,s')\in \overline {\mathcal S} _{2,j}}
\int _{ (\Prm') ^{\epsilon } \overline {R}_s}
\lvert  \Delta _{s',1} \rvert \; dx 
& \lesssim  
 (\Prm') ^{-\widetilde \alpha + \epsilon  } 
\sqrt {\frac {\Scl' \cdot \Prm'} {\Scl \cdot \Prm}} 
\\
\sum _{s'\,:\, (s,s')\in \overline {\mathcal S} _{2,j}}
\lvert  \Delta _{s',1} \rvert 
&\lesssim 
\sqrt {\frac {\Scl' \cdot \Prm'} {\Scl \cdot \Prm}} 
\chi ^{(2)} _{\overline {R}_s} \,, 
\qquad x\not\in (\Prm') ^{\epsilon }\overline {R}_s\,. 
\end{align*}
Note that we again have to be careful to accommodate the different normalizations 
here.  The proof of \eqref{e.789789} can be modified to prove this estimate.

Putting these two estimates together clearly proves that 
\begin{equation*}
 \sum _{s'\,:\, (s,s')\in \mathbf S_{2,j}}
\lvert  \ip \operatorname I _{e_{s'}} \phi  _{s}, \Delta  _{s'},\rvert  
\lesssim 2 ^{-10j} (\Prm') ^{-\widetilde \alpha } 
 \sqrt {\frac {\Scl' \cdot \Prm'} {\Scl \cdot \Prm}} \,, 
\end{equation*}
as is required.

\medskip 

We now turn to the proof of the inequality \eqref{e.caseII-b}, which 
will follow from appropriate upper bounds on the sums below. 
\begin{gather} 
\label{e.1case22}
 \sum _{s\,:\, (s,s')\in \mathbf S_{2,j}}
 \lvert  \ip \Delta_{s}, \operatorname I _{e }\phi _{s'},\rvert  \,,
\\ \label{e.2case22}
 \sum _{s\,:\, (s,s')\in \mathbf S_{2,j}}
\lvert  \ip \operatorname I _{e } \phi  _{s}, \Delta  _{s'},\rvert  \,,
\\ \label{e.3case22}
 \sum _{s\,:\, (s,s')\in \mathbf S_{2,j}}
\lvert  \ip  \Delta_{s}, \Delta  _{s'},\rvert  \,.
\end{gather}
Here, we can regard $  s'$ as a fixed tile, and 
$ e\in \boldsymbol \omega _{s'}\subset \boldsymbol \omega _{s}$. 
In this case, observe that we  have the inequality 
\begin{equation} \label{e.s2js2j}
\sharp \{s\,:\, (s,s')\in \mathbf S_{1,j}\} \lesssim 2 ^{2j}\,. 
\end{equation}
This is so since $ R _{s}$ has larger dimensions in both directions than 
does $ R _{s'}$.

\subsubsection*{The Upper Bound on \eqref{e.1case22}.}
We use the decomposition of $ \Delta _{s}=\Delta _{s,1}+\Delta _{s,2}+\Delta _{s,3}$. 
In the first case, we can estimate 
\begin{align*}
 \sum _{s\,:\, (s,s')\in \mathbf S_{2,j}}
 \lvert  \ip \Delta_{s,1}, \operatorname I _{e }\phi _{s'},\rvert  
 & \lesssim 
 2 ^{2j} \sup  \lvert  \ip \Delta_{s,1}, \operatorname I _{e }\phi _{s'},\rvert  
 \\
 & \lesssim 2 ^{-10j} (\Prm') ^{-\widetilde \alpha } 
 \sqrt {\frac {\Scl \cdot \Prm} {\Scl' \cdot \Prm'} }\,. 
\end{align*}

For the last case, of $ \Delta _{s,3}$, we estimate 
\begin{align*} 
 \sum _{s\,:\, (s,s')\in \mathbf S_{2,j}}
 \lvert  \ip \Delta_{s,3}, \operatorname I _{e }\phi _{s'},\rvert  
 & \lesssim 
 2 ^{2j} \sup  \lvert  \ip \Delta_{s,3}, \operatorname I _{e }\phi _{s'},\rvert  
 \\ 
 & \lesssim  2 ^{2j}\min \Bigl\{ 
\abs{ F_s} \cdot  \sqrt{\Scl \cdot \Prm  \cdot \Scl' \cdot \Prm'}
\,,\, 
\\ & \qquad 
2 ^{-30j}
 \sqrt {\frac {\Scl \cdot \Prm} {\Scl' \cdot \Prm'} } 
 \Bigr\}\,.
\end{align*}
Examining the two terms of the minimum, note that by \eqref{e.Fs2}, 
\begin{align*}
\abs{ F_s} \cdot  \sqrt{\Scl \cdot \Prm  \cdot \Scl' \cdot \Prm'}
& \lesssim (\Prm') ^{-\alpha +\epsilon }  
 \sqrt {\frac {\Scl' \cdot \Prm'} {\Scl \cdot \Prm}}
 \\
 & \lesssim (\Prm') ^{-\alpha +1 +\epsilon }
  \sqrt {\frac 1{\Scl' \cdot \Prm' \cdot \Scl \cdot \Prm}}
  \\
  & \lesssim (\Prm') ^{-\alpha +1 +\epsilon }
   \sqrt {\frac {\Scl \cdot \Prm} {\Scl' \cdot \Prm'}}\,. 
\end{align*}
Here it is essential that we have the estimate \eqref{e.Fs1} as stated, 
with $ \lvert  F_s\rvert \lesssim (\Prm') ^{-\alpha +\epsilon } \lvert  R_s\rvert  $. 
This is an estimate of the desired form, but without any decay in the parameter $ j$. 
The second term in the minimum does have the decay in $ j$, but does not have the 
decay in $ \Prm'$.  Taking the geometric mean of these two terms finishes the proof, 
provided $ (\alpha -\epsilon)/2> \widetilde \alpha  $, which we can assume by taking 
$ \alpha $ sufficiently close to one.

\subsubsection*{The Upper Bound on \eqref{e.2case22}.}
Using the inequality $ \lvert  \operatorname I _{e} \phi _{s}\rvert 
\lesssim  \chi ^{(2)} _{R _{s}}$, and the inequalities \eqref{e.789789} and \eqref{e.s2js2j}, 
it is easy to see that 
\begin{equation*}
 \sum _{s\,:\, (s,s')\in \mathbf S_{2,j}}
\lvert  \ip \operatorname I _{e } \phi  _{s}, \Delta  _{s'},\rvert  
\lesssim  2 ^{-12j}
(\Prm') ^{-\widetilde \alpha }    \sqrt {\frac {\Scl \cdot \Prm} {\Scl' \cdot \Prm'}}\,. 
\end{equation*}
This is the required estimate. 

\subsection*{The Proof  of (\ref{e.Suj}) for $ \mathcal S _{3,j}$, $ j\ge 1$.} 
In this case, we have that the length of the rectangles $R_{s'}$ are greater
than those of the rectangles $R_s$, as 
depicted in Figure~\ref{f.cases}.  We  show that 
\begin{align}  \label{e.caseIII-a} 
\operatorname{FL}(\mathcal S _{3,j} )& \lesssim (\Prm') ^{\epsilon }
\sqrt {\frac {\Scl' \cdot \Prm'} {\Scl \cdot \Prm}} 
\\ \label{e.caseIII-b}  
\operatorname{FL}'(\mathcal S _{3,j} )& \lesssim2^{-10j}
(\Prm')^{-\widetilde \alpha   }  \sqrt {\frac {\Scl \cdot \Prm} {\Scl' \cdot \Prm'} }\,.  
\end{align}
In particular, we do not claim any decay in the term $ \operatorname {FL} (\mathcal S
_{3,j})$, in fact permitting a small increase in the parameter $ \Prm'$.  Recall that 
$ 0<\epsilon <1$ is a small quantity.  See \eqref{e.tilde-a}. 
But due to the form of the estimate in Proposition~\ref{p.Schur}, with the 
decay  in $ 2 ^{j} $ and $ \Prm'$ in the estimate \eqref{e.caseIII-b}, these 
two estimates still prove \eqref{e.Suj} for $ \mathcal S _{3,j}$.

For the proof of \eqref{e.caseIII-a}, we analyze the sums 
\begin{gather} 
\label{e.1case3}
 \sum _{s'\,:\, (s,s')\in \mathbf S_{3,j}}
 \lvert  \ip \Delta_{s}, \operatorname I _{e_{s'}}\phi _{s'},\rvert  \,,
\\ \label{e.2case3}
 \sum _{s'\,:\, (s,s')\in \mathbf S_{3,j}}
\lvert  \ip \operatorname I _{e_{s'}} \phi  _{s}, \Delta  _{s'},\rvert  \,,
\\ \label{e.3case3}
 \sum _{s'\,:\, (s,s')\in \mathbf S_{3,j}}
\lvert  \ip  \Delta_{s}, \Delta  _{s'},\rvert  \,.
\end{gather}
Here, $ e _{s'}\in \boldsymbol \omega _{s'}\subset \boldsymbol \omega _{s}$. 

\subsubsection*{The Upper Bound on \eqref{e.1case3}.}
Regard $ s$ as fixed.  We employ a variant of the notation established in 
\eqref{e.OverLine}.  Let $ \widetilde R_s$ be a rectangle with 
in the same coordinates axes as $ R_s$. In the direction $ e_s$, let it have 
length $ 1/\Scl'$, that is the (longer) length of the rectangles $ R _{s'}$, 
and let it have the same width of $ R _{s}$.   Further assume that 
$ 2 ^{j} R _s \cap \widetilde R_s= \emptyset $ but $ 2 ^{j+4} R_s 
\cap \widetilde R_s\neq  \emptyset $. (There is an obvious change in these 
requirements for $ j=1$.)
Then, define 
\begin{equation*}
\widetilde {\mathcal S} _{3,j}
= \{ (s,s') \in \mathcal S _{3,j} \mid R _{s'} \subset \widetilde R_s\}\,.
\end{equation*}
With $ \lesssim 2 ^{2j}$ choices of $ \widetilde R_s$, we can exhaust the 
collection $ \mathcal S _{3,j}$.  Thus, we prove a slightly stronger 
estimate in the parameter $ 2 ^j$  for the collection 
$ \widetilde {\mathcal S} _{3,j}$. 

The main point here is that we have an analog of the estimate \eqref{e.t67}: 
\begin{equation*}
\sum _{s' \,:\, (s,s')\in \widetilde {\mathcal S} _{3,j}} 
\lvert  \operatorname I _{e} \phi _{s'}\rvert
\lesssim 
\sqrt {\frac {\Prm'} {\Prm}} \chi ^{(2)} _{\widetilde R_s}\,. 
\end{equation*}
The term in the square root takes 
into account the differing $ L^2$ normalizations 
between the $ \phi _{s'} $ and $ \chi ^{(2)} _{\widetilde R_s}$. 
The proof of \eqref{e.t67} can be modified to prove the estimate above. 

We also have the analogs of the estimate \eqref{e.789789}.  Putting these two together 
proves that 
\begin{align*}
\sum _{s'\,:\, (s,s')\in \mathbf S_{3,j}}
 \lvert  \ip \Delta_{s}, \operatorname I _{e_{s'}}\phi _{s'},\rvert  
 & \lesssim
    2 ^{-10j}(\Prm') ^{-\widetilde \alpha } 
    \sqrt {\frac {\lvert  R_s\rvert }  {\widetilde R_s}}
    \sqrt {\frac {\Prm'} {\Prm}}  
 \\
 & \lesssim 2 ^{-10j}(\Prm') ^{-\widetilde \alpha } 
 \sqrt {\frac {\Scl' \cdot \Prm'} {\Scl \cdot \Prm}}\,. 
\end{align*}  
That is, we get the estimate we want with decay in $ \Prm'$, we do not claim in general.

\subsubsection*{The Upper Bound on \eqref{e.2case3}.}
We use the inequality 
\begin{equation*}
\lvert  \operatorname I _{e_{s'}} \phi _{s}\rvert \lesssim 
\chi ^{(2)} _{R _{s}} \,. 
\end{equation*}
And we use the decomposition $ \Delta _{s'}=\Delta _{s',1}+\Delta _{s',2}+\Delta _{s',3}$. 

For the case of $ \Delta _{s',1}$, we have $ \boldsymbol \omega _{s'}\subset 
\boldsymbol \omega _{s}$.  And the supports of the functions 
$ \Delta _{s'}$ are well localized with respect to the vector field.  See \eqref{e.tech1}. 
Thus, in particular we have 
\begin{equation*}
\sum _{s'\,:\, (s,s')\in \mathbf S_{3,j}} 
 \lvert  \Delta _{s'}\rvert \lesssim  (\Prm') ^{\epsilon }
 \sqrt { {\Scl' \cdot \Prm'} } \,. 
\end{equation*}

Hence, we have 
\begin{align*}
\sum _{s'\,:\, (s,s')\in \mathbf S_{3,j}}
 \lvert  \ip \chi ^{(2)} _{R _{s}} , \Delta _{s'},\rvert
 & \lesssim 
 (\Prm') ^{-\epsilon  } 
 \sqrt {\frac {\Scl' \cdot \Prm'} {\Scl \cdot \Prm}}
\end{align*}
which is the desired estimate.

\begin{remark}\label{r....} It is the analysis of the sum 
\begin{equation*}
\sum _{s'\,:\, (s,s')\in \mathbf S_{3,j}}
 \lvert  \ip \operatorname I _{e} \phi _{s}, \Delta _{s',3},\rvert
\end{equation*}
that prevents us from obtaining decay in the parameter $ \Prm'$ 
for certain choices of parameters $ \Scl\,,\Prm\,, \Scl'$ and $ \Prm'$. 
This is why we have formulated \eqref{e.caseIII-a} the way we have. 
\end{remark}

\bigskip

For the proof of \eqref{e.caseIII-b}, we analyze the sums 
\begin{gather} 
\label{e.1case33}
 \sum _{s\,:\, (s,s')\in \mathbf S_{3,j}}
 \lvert  \ip \Delta_{s}, \operatorname I _{e}\phi _{s'},\rvert  \,,
\\ \label{e.2case33}
 \sum _{s\,:\, (s,s')\in \mathbf S_{3,j}}
\lvert  \ip \operatorname I _{e} \phi  _{s}, \Delta  _{s'},\rvert  \,,
\\ \label{e.3case33}
 \sum _{s\,:\, (s,s')\in \mathbf S_{3,j}}
\lvert  \ip  \Delta_{s}, \Delta  _{s'},\rvert  \,.
\end{gather}
Here $ e _{s'}\in \boldsymbol  \omega _{s'}\subset \boldsymbol  \omega _{s} $, 
and one can regard the interval $ \boldsymbol  \omega _{s'}$ as fixed. 
It is essential that we obtain the decay in $ 2 ^{j}$ and $ \Prm'$ in these 
cases. 

Indeed, these cases are easier, as the sum is over $ s$.  For fixed $ s'$, there 
is a unique choice of interval $ \boldsymbol  \omega _{s}\supset \boldsymbol  \omega _{s'}$. 
And the rectangles $ R_s $ are shorter than $ R _{s'}$, but wider.  Hence, 
\begin{equation}\label{e.4<4}
\sharp \{ s\,:\, (s,s')\in \mathbf S_{3,j}\} \lesssim 2 ^{2j} \frac {\Scl} {\Scl'}\,. 
\end{equation}

\subsubsection*{The Upper Bound on \eqref{e.1case33}.}
We use the decomposition $ \Delta _{s}=\Delta _{s,1}+\Delta _{s,2}+\Delta _{s,3}$, 
and the inequality $ \lvert  \operatorname I _{e_{s'}}\phi _{s'},\rvert \lesssim 
\chi ^{(2)} _{R _{s'}}$.  

For the sum associated with $ \Delta _{s,1}$, we have 
\begin{align*}
\sum _{s\,:\, (s,s')\in \mathbf S_{3,j}}
 \lvert  \ip \Delta_{s,1}, \operatorname I _{e_{s'}}\phi _{s'},\rvert  
 & \lesssim  (\Prm') ^{-\widetilde \alpha }
 \sum _{s\,:\, (s,s')\in \mathbf S_{3,j}}
 \ip \chi ^{(2)} _{R _{s}}, \chi ^{(2)} _{R _{s'}},
 \\
 & \lesssim 2 ^{-12j}  (\Prm') ^{-\widetilde \alpha } 
  \sqrt {\frac {\Scl' \cdot \Prm} {\Scl \cdot \Prm'}}
 \\ & \qquad \times 
 \sharp \{ s\,:\, (s,s')\in \mathbf S_{3,j}\}
 \\
 & \lesssim 
 2 ^{-10j} (\Prm') ^{-\widetilde \alpha } 
 \sqrt {\frac {\Scl' \cdot \Prm'} {\Scl \cdot \Prm}}\,. 
 \end{align*}
This is the required estimate.

For the sum associated with $ \Delta _{s,3}$, the critical properties 
are those of the corresponding sets $ F _{s}$, described in \eqref{e.Fs1} 
and \eqref{e.Fs2}.  Note that the sets 
\begin{equation*}
\sum _{s\,:\, (s,s')\in \mathbf S_{3,j}} \mathbf 1 _{F_s} \lesssim (\Prm') ^{2 \epsilon }\,.
\end{equation*}
 On the other hand, 
\begin{align*}
\sum _{s\,:\, (s,s')\in \mathbf S_{3,j}} \lvert  F_s\rvert
&
\lesssim 2 ^{2j} \frac {\Scl} {\Scl'}
\sup _{s\,:\, (s,s')\in \mathbf S_{3,j}} \lvert  F_s\rvert
\\
& \lesssim 2 ^{2j} (\Prm') ^{-\alpha +\epsilon } 
\frac 1 {\Scl' \cdot \Prm}\,. 
\end{align*}
Here, we have used the estimate \eqref{e.4<4}.  

This permits us to estimate 
\begin{align*}
\sum _{s\,:\, (s,s')\in \mathbf S_{3,j}}
 \lvert  \ip \Delta_{s,3}, \operatorname I _{e_{s'}}\phi _{s'},\rvert  
 & \lesssim   2 ^{-10j}(\Prm') ^{- \alpha +3\epsilon  }
 \sqrt {\frac {\Scl \cdot \Prm'} {\Scl' \cdot \Prm}}
\end{align*}
Note that the parity between the `primes' is broken in this estimate. 
By inspection, one sees that this last term is at most 
\begin{equation*}
\lesssim 2 ^{-10j} (\Prm') ^{-\widetilde \alpha } 
 \sqrt {\frac {\Scl' \cdot \Prm'} {\Scl \cdot \Prm}}\,.
\end{equation*}
Indeed, the claimed inequality amounts to 
\begin{equation*}
(\Prm') ^{-\alpha +3\epsilon } \Scl \lesssim (\Prm') ^{-\widetilde \alpha }\Scl' \,. 
\end{equation*}
We have to permit $ \Scl'$ to be as small as $ 1$, whereas $ \Scl$ can be 
as big as $ \Prm$.  But $ \alpha >1$, and  $ \Prm<\Prm'$, 
so the inequality above is trivially true.  
This completes the analysis of \eqref{e.1case33}.

\subsubsection*{The Upper Bound on \eqref{e.2case33}.}
We only need to use the inequality $ \lvert  \operatorname I _{e _{s'}} \phi _{s}\rvert 
\lesssim \chi ^{(2)} _{R _{s}}$, and the inequalities \eqref{e.789789}. 
It follows that 
\begin{align*}
\sum _{s\,:\, (s,s')\in \mathbf S_{3,j}}
\lvert  \ip \operatorname I _{e_{s'}} \phi  _{s}, \Delta  _{s'},\rvert  
& \lesssim 
\sum _{s\,:\, (s,s')\in \mathbf S_{3,j}} 
\ip \chi ^{(2)} _{R _{s}}, \lvert  \Delta _{s'}\rvert ,
\\
& \lesssim 2 ^{-10j}  (\Prm') ^{-\widetilde \alpha } 
 \sqrt {\frac {\Scl' \cdot \Prm'} {\Scl \cdot \Prm}} \,. 
\end{align*}

\section*{The Fourier Localization Estimate}

The precise form of the inequalities quantifying the Fourier localization
effect follows.

\begin{fl}\label{l.fl}
  Let $1<\alpha<2$,  $\ensuremath{\epsilon}<(\ensuremath{\alpha}-1)/20$, 
  and  $v$ be a vector field with 
$\norm  v.C^\alpha.\le1$.
 Let  $s$ be a tile with 
\begin{equation*}
1<\scl s=\Scl\le\prm s=\Prm<\tfrac1{16}2^k.
\end{equation*}
Let 
\begin{equation*}
f_s=\operatorname{Mod}_{-c(\omega_s)}\phi_s
\end{equation*}
Let $\zeta$ be a smooth function on \ensuremath{\mathbb R},  with $\ind {(-2,2)}
\le\widehat\zeta\le\ind {(-3,3)}$ and set 
$\zeta_{2^k}(y)=2^{k}\zeta(y2^{k})$.  We have this inequality
valid for all unit vectors $e$ with $\abs{e-e_s}\le\abs{{\boldsymbol \omega}_s}$. 
\begin{align} \notag 
\ABs{ f_s(x)-\int_\ensuremath{\mathbb R} & f_s(x-ye)\zeta_{2^k}(y)\; dy}
\\ \label{e.tech1}
\lesssim & (\Scl 2^{(\alpha -1) k})^{-1+\epsilon} 
	 \chi_{R_s}^{(2)}(x) \ind {\widetilde{{\boldsymbol\omega}_s}}(v(x))
\\
\label{e.tech2} &\quad+(2^{k}\Scl)^{-10} \chi^{(2)}_{R_s} (x)
\\
\label{e.tech3}{}&{}\quad+\abs{R_s}^{-1/2}\ind {F_s}(x)\,,
\end{align}
where $\widetilde{{\boldsymbol\omega}_s}$ is a sub arc of the 
unit circle, with $\widetilde{{\boldsymbol\omega}_s}=\lambda{\boldsymbol\omega}_s$,
and $1<\lambda<2^{\ensuremath{\epsilon}k}$.
Moreover,  the sets $F_s\subset\mathbb R^2$ satisfy 
\begin{gather} 
\label{e.Fs1} 
\abs{F_s}\lesssim{}2^{-(\alpha -\epsilon )k}(1+\Scl^{-1})^{\alpha-1} \abs {R_s}, 
\\ \label{e.Fs2} 
F_s\subset 2^{\ensuremath{\epsilon} k}R_s\cap v^{-1}(\widetilde{{\boldsymbol\omega}_s})
\cap
\Bigl\{ \ABs{\frac{\partial (v\!\cdot \!e_{s\perp})}{\partial e}
             }>2^{(1-\epsilon)k} \frac \Scl\Prm \Bigr\}.
\end{gather}
\end{fl}

The appearance of the set $ F_s$ is explained in part because
the only way for the function $ \phi _s$ to oscillate quickly along the 
direction $ e _{s}$ is that the vector field moves back and forth across the 
interval $ \boldsymbol \omega _s$ very quickly.  This sort of behavior, as it 
turns out, is the only obstacle to the frequency localization described in this Lemma. 

Note  that the degree of localization improves in $k$.  In \eqref{e.tech1}, it is 
important that we have the localization in terms of the directions of the vector 
field. The terms in \eqref{e.tech2} will be very small in all the instances that 
we apply this lemma. The third estimate \eqref{e.tech3} is the most complicated, as 
it depends upon the exceptional set.  The form of the exceptional set in \eqref{e.Fs2}
 is not so important, but the size estimate, as a function of $ \alpha >1$, in \eqref{e.Fs1}
 is.  

\smallskip

\begin{proof}
We collect  some elementary  estimates.  Throughout this argument, $\vec y:=y\, e\in\mathbb R^2$. 
\begin{equation} \label{e.st1} 
\int_{\abs y>t2^{-k}} \abs {y2^k}\abs{\zeta_{2^k}(y)}\; dy\lesssim{}t^{-N}, \qquad t>1.
\end{equation}
This estimate holds for all $N>1$.  Likewise, 
\begin{equation} \label{e.st2}  
\int_{\abs u>t\Scl^{-1}}\abs{u\Scl}\abs{\Scl\, \psi(\Scl \, u)}\; du\lesssim{}t^{-N}, \qquad t>1.
\end{equation}
More significantly, we have for all $  x\in\mathbb R^2$,  
\begin{equation}\label{e.ft1} 
\int_{\mathbb R^2} \operatorname e^{i\xi_0 y}\varphi_{R_s}^{(2)}(x-\vec y)\zeta_{2^k}(y)\; dy=\varphi_{R_s}^{(2)}(x) 
\qquad   -2^k<\xi_0< 2^{k},
\end{equation}
where $\varphi^{(2)}_{R_s}={\operatorname T}_{c(R_s)}{\operatorname D}^2_{R_s}
\varphi$.
This is seen by taking the Fourier transform. Likewise, by \eqref{e.zc-Fourier}, 
for vectors $v_0$ of unit length,
\begin{equation*}
\int_{\mathbb R} \operatorname  e^{-2\pi iu\lambda_0}\varphi_{R_s}^{(2)}(x-u v_0) \Scl\, \psi(\Scl \, u)\; du\not=0
\end{equation*}
implies that 
\begin{equation}\label{e.ft2} 
\Scl\le\lambda_0+\xi \cdot v_0\le{}\tfrac98\Scl,\quad
\text{for some $\xi\in\text{supp}(\widehat{\varphi_{R_s}^{(2)}})$.}
\end{equation}

At this point, it is useful to recall that we have specified the 
frequency support of $\varphi$ to be in a small ball of radius 
$\kappa$ in \eqref{e.zvfs}.  This has the implication that 
\begin{equation}\label{e.supp}
\abs{\xi\cdot e_s}\le{}\kappa\Scl,\quad \abs{\xi\cdot e_{s\perp}}\le\kappa\Prm \qquad 
\xi\in\text{supp}(\widehat{\varphi_{R_s}^{(2)}})
\end{equation} 

\smallskip 
We begin the main line of the argument, which comes in two stages.  In the first stage, we address 
the issue of the derivative below exceeding a `large' threshold.  
\begin{equation*}
e\cdot\mathbf D v(x) \cdot e_{s\perp}=\frac{\partial v \cdot e_{s\perp}} {\partial e} 
\end{equation*}
We shall find that this happens on a relatively small set, the set $F_s$ of the Lemma. 
Notice that due to the eccentricity of the rectangle $R_s$, we can only hope to have some control 
over the derivative in the long direction of the rectangle, and $e$ essentially points in the 
long direction.   We are interested in derivative in  the direction $e_{s\perp}$   
as that is the direction that $v$ must move to cross the interval ${\boldsymbol \omega} _{s2}$. 
A substantial portion of the 
technicalities below are forced upon us due to the few choices of
scales $1\le{}\Scl\le{} 2^{\varepsilon{}k}$, for 
some small positive \ensuremath{\varepsilon}.\footnote{The scales of approximate length one 
are where the smooth character of the vector field helps the least. 
The argument becomes especially easy in the case 
that $\sqrt{\Prm}\le{}\Scl$, as in the case, $\abs{{\boldsymbol\omega}_s} {}\gtrsim{}\Scl^{-1}$.}

Let $0<\varepsilon_1,\varepsilon_2<\epsilon $ to be specified in the argument below. 
In particular, we take 
\begin{equation*}
0<\varepsilon_1\le{}\min\bigl( \tfrac 1{1200},\kappa\tfrac{\alpha-1}{20}\bigr),\qquad 0<\varepsilon_2<\tfrac1{18}(\alpha-1).
\end{equation*}
We have the estimate 
\begin{equation*}
\abs{f_s(x)}+\ABs{\int_\ensuremath{\mathbb R}f_s(x-\vec y)\zeta_{2^k}(y)\; dy}
\lesssim{}2^{-10k} \chi_{R_s}^{(2)}(x), \qquad x\not\in 2^{\varepsilon_1 k}R_s.
\end{equation*}
This follows from \eqref{e.st1} and the fact that the direction $e$ differs from $e_s$ by an no more than the measure of the 
angle of uncertainty for $R_s$.  This  is as claimed in \eqref{e.tech2}.
We need only consider $x\in  2^{\varepsilon_1 k}R_s$.

Let us define the sets $F_s$, as in \eqref{e.tech3}.   
Define 
\begin{equation*}
\lambda_s:=\begin{cases} 
	2^{\varepsilon_1 k} & \frac\Scl\Prm<2^{-2\varepsilon_1 k} 
	\\
	8 & \text{otherwise}
	\end{cases}
\end{equation*}
Let $\lambda \omega_s$ denote the interval on the unit circle  
with length $\lambda \abs{{\boldsymbol \omega _s}}$,
and the
same center as ${\boldsymbol \omega}_s$.\footnote{We have defined
$\lambda_s$ this way so that $\lambda_s {\boldsymbol \omega}_s$ makes sense.} 
This is our $ \widetilde {\boldsymbol \omega} $ of the Lemma; the set 
$ F_s$ of the Lemma is 
\begin{equation}\label{e.FsDefined}
F_s:=2^{\varepsilon_1 k}R_s\cap v^{-1}(\lambda_s {\boldsymbol\omega}_s)\cap 
\Bigl\{ \ABs{\frac{\partial (v\!\cdot \!e_{s\perp})}{\partial e}
             }>2^{(1-\varepsilon_2)k} \frac \Scl\Prm \Bigr\}.
\end{equation}
And so to satisfy \eqref{e.Fs2}, we should take $\varepsilon_1<1/1200$.

Let us argue that the measure of $F_s$ satisfies \eqref{e.Fs1}. 
Fix a line $\ell$ in the direction of $e$.  We should see that 
the one dimensional measure 
\begin{equation} \label{e.FcapLine}
\abs{ \ell\cap F_s }
\lesssim2^{-k(\alpha - \epsilon )}(1+\Scl^{-1})^{\alpha-1}\Scl^{-1}.
\end{equation}
For we can then integrate over the choices of $\ell$ to get the estimate in  \eqref{e.Fs1}.

The set $ \ell \cap F_s$ is viewed as a subset of $ \mathbb R $. 
It consists of open intervals $A_n=(a_n,b_n)$, $1\le n\le N$.  List them so that
$b_n<a_{n+1}$ for all $n$.  Partition the integers 
$\{1,2,\ldots,N\}$ into sets of consecutive integers $I_\sigma=[m_\sigma,n_\sigma]\cap\mathbb N$ 
so that for all points  $x$ between the left-hand endpoint of 
$A_{m_\sigma}$ and the right-hand endpoint of $A_{n_\sigma}$, 
 the derivative  $\partial (v\cdot e_{s\perp})/\partial e$ has the same sign. 
  Take the intervals  of integers $I_\sigma$ to be maximal with respect to this property. 

For $x\in F_s$, the partial derivative of $v$, in the direction that is transverse 
to $\lambda_s {\boldsymbol \omega}_s$,
 is large with respect to the length of $\lambda_s {\boldsymbol \omega}_s$. 
 Hence, $v$ must pass across $\lambda_s {\boldsymbol \omega}_s$ in a small amount of time: 
\begin{equation*}
\sum_{m\in I_\sigma}\abs{A_m}\lesssim{}2^{-(1-\varepsilon_1-\varepsilon_2)k}\qquad \text{for all $\sigma$.}
\end{equation*}

Now consider intervals $A_{n_\sigma}$ and $A_{1+n_\sigma}=A_{m_{\sigma+1}}$. 
By definition, there must be a change of sign 
of $\partial v(x)\cdot e_{s\perp}/\partial e$ between these two intervals.  
And so there is a change 
in this derivative that is at least as big as $2^{(1-\varepsilon_2)k} \frac\Scl\Prm$. 
The partial derivative is also H\"older   
continuous of index $\alpha-1$, which implies that $A_{n_\sigma}$ and $A_{m_{\sigma+1}}$ 
cannot be very close, specifically 
\begin{equation*}
\text{dist}(A_{n_\sigma},A_{m_{\sigma+1}})
\ge{}\bigl(2^{(1-\varepsilon_2)k} \frac \Scl\Prm \bigr)^{\alpha-1}
\end{equation*}
As all of the intervals $A_{n}$ lie in an interval of length $2^{\varepsilon_1k}\Scl^{-1}$, 
it follows that there can be at most 
\begin{equation*}
1\le{}\sigma{}
\lesssim2^{\varepsilon_1 k}\Scl^{-1}
\bigl(2^{(1-\varepsilon_2)k} \frac \Scl\Prm \bigr)^{-\alpha+1}
\end{equation*}
intervals $I_\sigma$. Consequently, 
\begin{align*}
\abs{\ell\cap F_s  } 
&\lesssim2^{-(1-2\varepsilon_1-\varepsilon_2+(1-\varepsilon_2)(\alpha-1))k}
\Scl^{-1} \bigl( \frac \Prm\Scl\bigr)^{\alpha-1}
\\&\lesssim 
2^{-(\alpha -2\varepsilon_1-2\varepsilon_2)k}\Scl^{-\alpha+1}\Scl^{-1} 
\end{align*}
We have already required  $0<\varepsilon_1<\frac \epsilon {600}$ and taking  
$0<\varepsilon_2<\frac \epsilon {600} $ will  achieve the estimate  \eqref{e.FcapLine}. 
This completes the proof of \eqref{e.Fs1}.

\medskip 

The second stage of the proof begins, in which we make a detailed estimate of the
difference in question, seeking 
to take full advantage of the Fourier properties \eqref{e.st1}---\eqref{e.ft2}, 
as well as the derivative information encoded into the 
set $F_s$.   

We   consider the    difference in \eqref{e.tech1} in the case of 
$x\in 2^{\varepsilon_1 k}R_s-{} v^{-1}(\lambda_s {\boldsymbol\omega}_s)$.  
In particular, $x$ is not in the support of $f_s$, and due to the smoothness of the 
vector field, the distance of $x$ to the support of $f_s$ is at least 
\begin{equation*}
\gtrsim2^{\varepsilon_1 k}\frac\Scl\Prm 
\end{equation*}
so that by \eqref{e.st1}, we can estimate 
\begin{align*}
\ABs{f_s(x)-\inr f_s(x-\vec y)\zeta_{2^k}(y)\; dy}{}
\lesssim{}& (2^{\varepsilon_1 k}\Scl)^{-N}\abs{R_s}^{-1/2}
\end{align*}
which is the  estimate \eqref{e.tech2}. 

\smallskip 

We turn to the proof of \eqref{e.tech1}. 
For $x\in 2^{\varepsilon_1 k}R_s \cap v^{-1}  (\lambda_s {\boldsymbol\omega}_s)$, 
we always have  the bound 
\begin{equation*}
\ABs{f_s(x)-\inr f_s(x-\vec y)\zeta_{2^k}(y)\; dy}\lesssim{}2^{10\varepsilon_1 k/\kappa} 
\chi^{(2)}_{R_s}(x)^{10}\ind {\lambda _s{\boldsymbol\omega}_s}(x).
\end{equation*}
It is essential that we have $\abs{e-e_s}\le{}\abs{{\boldsymbol \omega}_s}$ for this to be true, 
and \ensuremath{\kappa} enters in on the right hand side 
through the  definition \eqref{e.zq}.  

We establish the bound 
\begin{equation}\label{e.100k}
\begin{split}
\ABs{f_s(x)-\int _{\mathbb R }f_s(x- \vec y) \zeta_{2^k}(y)\; dy}
&\lesssim{}(\Scl 2^{(\alpha-1) k})^{-1}\abs{R_s}^{-1/2}, 
\\
x\in 2^{\varepsilon_1 k}R_s\cap v^{-1} & (\lambda_s {\boldsymbol\omega}_s)\cap F_s^c.
\end{split}\end{equation}
We take the geometric mean of these two estimates, and specify that 
$0<\varepsilon_1<\kappa\frac{\alpha-1}{20}$ to conclude 
 \eqref{e.tech1}.

\smallskip

It remains to consider $x\in 2^{\varepsilon_1 k}R_s\cap v^{-1}(\lambda_s
{\boldsymbol\omega}_s)\cap F_s^c$, and now some detailed calculations are needed. 
To ease the burden of notation, we set 
\begin{equation*}
\operatorname{exp}(x):=e^{-2\pi i u c(\omega_s)\cdot v(x)},\qquad 
\Phi(x,x')=\varphi_{R_{s}}^{(2)}(x-uv(x')),
\end{equation*}
 with the dependency on $u$ being suppressed, and define
\begin{equation*}
w(du,d\vec y):=\Scl\,\psi(\Scl \,  u)\zeta_{2^k}(y)\;du\,d \vec y.
\end{equation*}
In this notation, note that 
\begin{align*}
f_s & = \operatorname {Mod} _{-c (\omega _s)} \phi _s 
\\
& = \int _{\mathbb R } \operatorname e ^{ c (\omega _s) (x-u v (x)- c (\omega _s)) x}
 \varphi ^{(2)} _{R_s} (x-u v (x)) \, \mathsf {scl} \psi (\mathsf {scl} u) \; du
 \\
 & = \int _{\mathbb R } \operatorname {exp} (x) \Phi (x,x,) \mathsf {scl} \psi (\mathsf {scl}
 u) \; du
 \\
 & = \int _{\mathbb R ^2  } \int _{\mathbb R }  \operatorname {exp} (x) \Phi (x,x,)  
  w(du,d\vec y)\,, 
\end{align*}
since $ \zeta $ has integral on $ \mathbb R ^2 $.   In addition, we have 
\begin{align*}
\int _{\mathbb R ^2 } f _{s} (x-\vec y e ) \zeta _{2 ^{k}} (\vec y)\; d \vec y 
& = 
\int _{\mathbb R ^2 } \int _{\mathbb R } 
\operatorname e ^{ c (\omega _s) (x-u v (x-\vec y)- c (\omega _s)) x}
\\& \qquad \times 
 \varphi ^{(2)} _{R_s} (x-u v (x-\vec y)) \, \mathsf {scl} \psi (\mathsf {scl} u) \; du \, d\vec y
 \\
 & = 
 \int _{\mathbb R ^2 } \int _{\mathbb R }  
 \operatorname{exp}(x-\vec y)\Phi(x-
\vec y,x-\vec y) w(du,d\vec y) \,. 
\end{align*}
We are to estimate the difference between these two expressions, which is the 
 difference  of 
\begin{gather*}
\operatorname{Diff}_1(x):=\inrr \inr \operatorname{exp}(x)\Phi(x,x)-
\operatorname{exp}(x-\vec y)\Phi(x-\vec y,x) w(du,d\vec y)
\\
\operatorname{Diff}_2(x):=\inrr \inr \operatorname{exp}(x-\vec y)\{\Phi(x-
\vec y,x-\vec y)-\Phi(x-\vec y,x)\} w(du,d\vec y)
\end{gather*}
The analysis of both terms is quite similar.  We begin with the first term. 

\medskip

Note that by \eqref{e.ft1}, we have 
\begin{equation*}
\operatorname{Diff}_1(x)=\inr \inr \{\operatorname{exp}(x)-\operatorname{exp}(x-\vec y)\}\Phi(x-\vec y,x) w(du,d\vec y).
\end{equation*}  
We make a first order approximation to the difference above. 
Observe that 
\begin{align} \nonumber 
\operatorname{exp}(x)-\operatorname{exp}(x-\vec y)&= \operatorname{exp}(x)\{1-\operatorname{exp}(x-\vec y)
	\overline{\operatorname{exp}(x)}\}
\\ \label{e.e(x)}
&= \operatorname{exp}(x)\{1-e^{-2\pi iu[c(\omega_s)\cdot {\mathbf D} v(x) \cdot e ] \vec y}\}
\\ \nonumber & \qquad +O(\abs u\Prm \abs{\vec y}^\alpha).
\end{align}
In the Big--Oh term, $\abs u$ is typically of the order $\Scl^{-1}$, and 
$\abs {\vec y}$ is of the order $2^{-k}$.  Hence, direct integration leads to the 
estimate of this term by 
\begin{align*}
\inr \inr  \abs u\Prm \abs{y}^\alpha\abs{\Phi(x-\vec y,x)} \cdot & \abs{w(du,d\vec y)} 
\\
\lesssim & \abs{R_s}^{-1/2}\frac \Prm \Scl 2^{-\alpha{}k}
\\
\lesssim & \abs{R_s}^{-1/2}(\Scl 2^{(\alpha-1)k})^{-1}.
\end{align*}
This is  \eqref{e.100k}. 

The term left to estimate is 
\begin{align*}
\operatorname{Diff}_1'(x):=
\inr \inr \operatorname{exp}(x)(1-&e^{-2\pi iu[c(\omega_s)\cdot {\mathbf D} v(x) \cdot e ] \vec y})
\\&\Phi(x-\vec y,x) w(du,d\vec y)\,.
\end{align*}
Observe that by \eqref{e.ft1}, the integral in $y$ is zero if 
\begin{equation*}
\abs{u[c(\omega_s)\cdot {\mathbf D} v(x) \cdot e ]}\le{} 2^k. 
\end{equation*}
Here we recall that $c(\omega_s)=\frac32\Prm\, e_{s\perp}$. By the definition of 
$F_s$,  the partial derivative is small, namely 
\begin{equation*}
\abs{e_{s\perp}\cdot {\mathbf D} v(x) \cdot e}\lesssim{}2^{(1-\varepsilon_1)k}\frac \Scl\Prm .
\end{equation*}
Hence, the integral in $y$ in $\operatorname{Diff}_1'(x)$ can be non--zero only for 
\begin{equation*}
\Scl \abs{u}\gtrsim2^{\varepsilon_1 k}.
\end{equation*}
By \eqref{e.st2}, it follows that in this case we have the estimate 
\begin{equation*}
\abs{ \operatorname{Diff}_1'(x)}\lesssim{}2^{-2k}\abs{R_s}^{-1/2}
\end{equation*}
This estimate holds for $x\in 2^{\varepsilon_1 k}R_s\cap v^{-1}
(\lambda_s {\boldsymbol\omega}_s)\cap F_s^c$  and this completes the 
proof of the upper bound \eqref{e.100k} for the first difference.

\smallskip 
We consider the second difference  $\operatorname{Diff}_2$. The term $v(x-\vec y)$ occurs twice
in this term,  in $\operatorname{exp}(x-\vec y)$, and in $\Phi(x-\vec y,x-\vec y)$.  
We will use the approximation \eqref{e.e(x)}, and similarly, 
\begin{align*}
\Phi(x-\vec y,x-&\vec y)-\Phi(x-\vec y,x)
\\
&=\varphi_{R_s}^{(2)}(x-\vec y-uv(x-\vec y))-\varphi_{R_s}^{(2)}(x-\vec y-uv(x))
\\
&= \varphi_{R_s}^{(2)}(x-\vec y-uv(x)-u{\mathbf D} v(x)\vec y)
\\&\qquad-\varphi_{R_s}^{(2)}(x-\vec y-uv(x))
+O(\Prm\,\abs u\abs{y}^\alpha)
\\
&=\Delta\Phi(x,\vec y)+O(\Prm\,\abs u\abs{\vec y}^\alpha)
\end{align*}
The Big--Oh term gives us, upon integration in $u$ and $\vec y$,  a term that is no more than 
\begin{equation*}\lesssim \abs{R_s}^{-1/2}\frac \Prm\Scl 2^{-\alpha k}
\lesssim\abs{R_s}^{-1/2}(\Scl 2^{(\alpha-1)k})^{-1}.
\end{equation*}
This is as required by \eqref{e.100k}. 

We are left with estimating 
\begin{equation*}
\operatorname{Diff}_2'(x):=\inr\inr 
e^{-2\pi iuc(\omega_s)\cdot (v(x)- {\mathbf D} v(x)\cdot \vec y)}\Delta\Phi(x,\vec y)\, w(du,d\vec y).
\end{equation*}
By \eqref{e.ft1}, the integral in $y$ is zero if both of these conditions hold. 
\begin{gather*}
\abs{uc(\omega_s){\mathbf D} v (x)\cdot e }<2^{k},
\\
\abs{[uc(\omega_s){\mathbf D} v(x) -\xi-u\xi {\mathbf D}v(x)]\cdot e }<2^{k},
\qquad \xi\in\text{supp}(\widehat {\varphi_{R_s}^{(2)}})
\end{gather*}
Both of these conditions are phrased in terms of the derivative which is controlled as $x\not\in F_s$.  
In fact, the first condition already occurred in the first case, and it is satisfied if 
\begin{equation*}
\Scl\abs u\lesssim{}2^{\varepsilon_1 k}.
\end{equation*}
Recalling the conditions \eqref{e.supp}, the second condition is also satisfied for 
the same set of values for $u$.  The application of \eqref{e.st2}
 then yields a very small bound after integrating $\abs{u}\gtrsim2^{\varepsilon_1 k}
\Scl^{-1} $.    This completes the proof our  technical Lemma. 
\end{proof}

\backmatter

\end{document}